\documentclass[oneside,english]{amsart}
\usepackage[T1]{fontenc}
\usepackage[latin9]{inputenc}
\usepackage{geometry}
\usepackage{textcomp}
\usepackage{amsbsy}
\usepackage{amstext}
\usepackage{amsthm}
\usepackage{amssymb}

\makeatletter
\numberwithin{equation}{section}
\numberwithin{figure}{section}
\theoremstyle{plain}
\newtheorem{thm}{\protect\theoremname}[section]
\theoremstyle{plain}
\newtheorem{cor}[thm]{\protect\corollaryname}
\theoremstyle{plain}
\newtheorem{conjecture}[thm]{\protect\conjecturename}
\theoremstyle{remark}
\newtheorem{rem}[thm]{\protect\remarkname}
\theoremstyle{plain}
\newtheorem{lem}[thm]{\protect\lemmaname}
\theoremstyle{definition}
\newtheorem{defn}[thm]{\protect\definitionname}
\theoremstyle{plain}
\newtheorem{prop}[thm]{\protect\propositionname}

\def\makebbb#1{
    \expandafter\gdef\csname#1\endcsname{
        \ensuremath{\Bbb{#1}}}
}\makebbb{R}\makebbb{N}\makebbb{Z}\makebbb{C}\makebbb{H}\makebbb{E}\makebbb{H}\makebbb{P}\makebbb{B}\makebbb{Q}\makebbb{E}

\usepackage{babel}

\usepackage{babel}

\makeatother

\usepackage{babel}

\makeatother

\usepackage{babel}
\providecommand{\conjecturename}{Conjecture}
\providecommand{\corollaryname}{Corollary}
\providecommand{\definitionname}{Definition}
\providecommand{\lemmaname}{Lemma}
\providecommand{\propositionname}{Proposition}
\providecommand{\remarkname}{Remark}
\providecommand{\theoremname}{Theorem}

\begin{document}
\title{Kähler-Einstein metrics arising from micro-canonical measures and
Hamiltonian dynamics}
\author{Robert J. Berman}
\begin{abstract}
We introduce new probabilistic and variational constructions of (twisted)
Kähler-Einstein metrics on complex projective algebraic varieties,
drawing inspiration from Onsager's statistical mechanical model of
turbulence in two-dimensional incompressible fluids. The probabilistic
construction involves microcanonical measures associated with the
level sets of the pluricomplex energy, which give rise to maximum
entropy principles. These, in turn, yield novel characterizations
of Kähler-Einstein metrics, as well of Fano varieties admitting such
metrics. Additionally, connections to Hamiltonian dynamics are uncovered,
resulting in a new evolution equation, that generalizes both the 2D
incompressible Euler equation and the 2D semi-geostrophic equation.
\end{abstract}

\maketitle

\section{Introduction}

The primary aim of this work is to introduce new probabilistic and
variational constructions of Kähler-Einstein metrics, along with their
twisted counterparts, complementing the constructions in \cite{berm8,berm8 comma 5}
and \cite{bbgz,berm6,bbegz}, respectively. The inspiration comes
from Onsager's statistical mechanical description of turbulence in
two-dimensional incompressible fluids in terms of point vortices and
the emergence of large-scale long-lived coherent structures, such
as vortices and jets \cite{o}. In particular, this work builds upon
the mathematical developments of Onsager's approach in \cite{e-s,clmp2,e-h-t,ki2}. 

\subsection{\label{subsec:Setup intro}Complex-geometric setup}

We start by introducing the complex-geometric setup (see Section \ref{subsec:Preliminaries polarized}
for further background). Let $X$ be a compact complex manifold of
dimension $n$ endowed with a big holomorphic line bundle $L.$ Denoting
by $H^{0}(X,kL)$ the space of all global holomorphic sections with
values in the $k$ th tensor power of $L$ (adopting additive notation
for tensor powers) this means that 
\begin{equation}
N:=N_{k}:=\dim H^{0}(X,kL)=\text{vol\ensuremath{(L)}}k^{n}/n!+o(k^{n})\label{eq:def of N k}
\end{equation}
 for a strictly positive number $\text{vol \ensuremath{(L)}, }$called
the \emph{volume} of $L.$ In particular, $N\rightarrow\infty$ iff
$k\rightarrow\infty.$ 

We fix the data $(\left\Vert \cdot\right\Vert ,dV)$ consisting of
a Hermitian metric $\left\Vert \cdot\right\Vert $ on $L$ and a volume
form $dV$ on $X$ with (at worst) divisorial singularities. We denote
by $\theta$ the curvature two-form of the metric $\left\Vert \cdot\right\Vert $
on $L$ multiplied by $i/2\pi$ and assume that it is \emph{not }the
case that $\theta\geq0$ and $\theta^{n}=CdV$ for some constant $C$
(otherwise the function $e(\beta)$ in Thm \ref{thm:micro variational principle low energ intro}
would be constant; see Remark \ref{rem:e beta constant}). We will
introduce new probabilistic and variational constructions of solutions
to the following \emph{twisted Kähler-Einstein equation} induced by
the data $(\theta,dV)$ and a parameter $\beta\in\R:$
\begin{equation}
\mbox{\ensuremath{\mbox{Ric}}\ensuremath{\omega}}+\beta\omega=\beta\theta+\ensuremath{\mbox{Ric}}\ensuremath{dV}\label{eq:twisted KE intro}
\end{equation}
 for a positive current $\omega$ in the first Chern class of $L$
such that $\log(\omega_{\beta}^{n}/dV)\in L^{1}(X)$ (the solution
is a Kähler metric iff $L$ is ample and $\ensuremath{dV}$ is a smooth
volume form \cite{au,y,bbegz}). The starting point of the constructions
is the basic fact that a solution $\omega_{\beta}$ can be recovered
from its volume form $\omega_{\beta}^{n}:$
\begin{equation}
\omega_{\beta}:=\frac{i}{2\pi\beta}\partial\bar{\partial}\left(\log\frac{\omega_{\beta}^{n}}{dV}\right)+\theta,\label{eq:omega beta in terms of mu}
\end{equation}
 when $\beta\neq0.$ We recall that when $\beta=0,$ the equation
\ref{eq:twisted KE intro} is the \emph{Calabi-Yau equation} \cite{y,begz,bbgz}
\begin{equation}
\mbox{\ensuremath{\mbox{Ric}}\ensuremath{\omega}}=\ensuremath{\mbox{Ric}}\ensuremath{dV}\,\,\,\,\:\,\,\,\left(\iff\omega^{n}=\frac{\text{vol}\ensuremath{(L)}}{\int_{X}dV}dV\right)\label{eq:cy eq}
\end{equation}
 Moreover, when the data $(\theta,dV)$ satisfies the following compatibility
condition 
\begin{equation}
\ensuremath{\text{Ric}}\ensuremath{dV=-\beta\theta}\label{eq:theta is Ricci}
\end{equation}
a solution $\omega$ to the corresponding equation \ref{eq:twisted KE intro}
for $\beta\in\{\pm1,0\}$ is a \emph{Kähler-Einstein metric,} i.e.
a Kähler metric with constant Ricci curvature: 
\begin{equation}
\mbox{\ensuremath{\mbox{Ric}}\ensuremath{\omega}}=-\beta\omega.\label{eq:KE eq intro}
\end{equation}
When $\beta=\pm1$ the compatibility condition \ref{eq:theta is Ricci}
implies that $L=\pm K_{X}$ (where $K_{X}$ denotes the canonical
line bundle of $X$) and it amounts to taking the metric $\left\Vert \cdot\right\Vert $
on $L$ to be the one induced by any multiple of $dV.$ More generally,
by taking the volume form $dV$ to be singular along a given divisor
$\Delta$ in $X$ one gets, when $\beta\in\{\pm1,0\}$ Kähler-Einstein
metrics that are singular along $\Delta;$ see Section \ref{subsec:K=0000E4hler-Einstein-metrics}. 

As recalled in Section \ref{subsec:Singular_complex_varieties} the
setup above (and the results below) essentially contains the case
when $X$ has log terminal (klt) singularities. Such complex varieties
first appeared in the Minimal Model Program in birational algebraic
geometry and have come to play a prominent role in Kähler geometry
\cite{s-t,berman6ii,li1,l-x-z}. But it should be stressed that the
results below are new already when $X$ and $dV$ are smooth and $L$
is ample.

\subsubsection{Energy}

The pair $(\left\Vert \cdot\right\Vert ,dV)$ induces a Hilbert space
structure on $H^{0}(X,kL).$ The Slater determinant $\Psi^{(k)}$
of the Hilbert space $H^{0}(X,kL)$ is the holomorphic section of
$(kL)^{\boxtimes N_{k}}\rightarrow X^{N_{k}}$ determined, up to a
choice of sign, by demanding that it be totally anti-symmetric and
normalized: 
\[
(N_{k}!)^{-1}\int_{X^{N_{k}}}\left\Vert \Psi^{(k)}\right\Vert ^{2}dV^{\otimes N_{k}}=1
\]
Concretely,
\[
\Psi^{(k)}(x_{1},...,x_{N_{k}}):=\det\left(\Psi_{i}^{(k)}(x_{j})\right)_{i,j\leq N_{k}}
\]
where $\Psi_{1}^{(k)},...,\Psi_{N_{k}}^{(k)}$ is a fixed orthonormal
basis in the Hilbert space $H^{0}(X,kL).$ Consider the following
lsc function on $X^{N}$ (that only depends on the data $(\theta,dV)$):

\[
E^{(N_{k})}(x_{1},...,x_{N_{k}}):=-\frac{1}{N_{k}k}\log\left\Vert \Psi^{(k)}(x_{1},...,x_{N_{k}})\right\Vert ^{2},
\]
(in the statistical mechanical framework, discussed in Section \ref{subsec:Equivalence-of-the},
it represents the \emph{energy per particle}). Under the embedding
defined by the empirical measure 
\begin{equation}
\delta_{N}:\,\,X^{N}/S_{N}\rightarrow\mathcal{P}(X):\,\,\,\delta_{N}(x_{1},\ldots,x_{N}):=\frac{1}{N}\sum_{i=1}^{N}\delta_{x_{i}}\label{eq:emp measure intro}
\end{equation}
 from the space $X^{N}/S_{N}$ of configurations of $N$ points on
$X$ into the space $\mathcal{P}(X)$ of probability measure, the
following convergence, established in \cite{berm8,berm8 comma 5},
holds in the sense of Gamma-convergence (see Definition \ref{def:Gamma}): 

\emph{
\[
\lim_{N_{k}\rightarrow\infty}E^{(N_{k})}(x_{1},...,x_{N_{k}})=E(\mu),
\]
} where $E(\mu)$ is the \emph{pluricomplex energy of} $\mu$ (relative
to $\theta)$ introduced in \cite{bbgz} , normalized so that 
\[
e_{\text{\ensuremath{\text{min }}}}:=\inf_{\mathcal{P}(X)}E(\mu)=0
\]
(see Section \ref{subsec:The pluricomplex energy}). Set 
\[
e_{0}:=E(dV/\int_{X}dV)
\]
The assumptions on $(\theta,dV)$ ensure that $e_{0}>0.$

\subsubsection{Entropy}

The \emph{entropy} $S(\mu)$ of a probability measure $\mu,$ relative
$dV$ and the \emph{specific entropy $S(e)$ at an energy level }$e$
are defined by 
\begin{equation}
S(\mu):=-\int_{X}\log(\mu/dV)\mu,\,\,\,\,S(e):=\sup_{\mu\in\mathcal{P}(X)}\left\{ S(\mu):\,\,E(\mu)=e\right\} \label{eq:def of S intro}
\end{equation}
(adopting the sign convention used in the statistical mechanics literature).
More precisely, if $\mu$ is not absolutely continuous wrt $dV$ then
$S(\mu):=-\infty$ and if $\{\mu:E(\mu)=e\}$ is empty, then $S(e):=-\infty.$
Given $e\in\R$ a probability measure $\mu^{e}$ is called a \emph{maximum
entropy measure} (of energy $e)$ if it attains the supremum defining
$S(e).$ 

\subsection{Main results in the low energy regime $(e\protect\leq e_{0})$}

Given $e$ and $\epsilon\in]0,\infty]$ consider the uniform probability
measure on the ``lower energy shell'' $\left\{ E^{(N)}\in]e-\epsilon,e[\right\} $
in $X^{N}:$

\begin{equation}
\mu_{]e-\epsilon,e[}^{(N)}:=1_{\left\{ E^{(N)}\in]e-\epsilon,e[\right\} }dV^{\otimes N}/\int_{\left\{ E^{(N)}\in]e-\epsilon,e[\right\} }dV^{\otimes N}.\label{eq:def of lower micro intro}
\end{equation}
We will view the empirical measure $\delta_{N}$ (formula \ref{eq:emp measure intro})
as a random measure, i.e. a random variable on the probability space
$(X^{N},\mu_{]e-\epsilon,e[}^{(N)}).$
\begin{thm}
\label{thm:LDP for micro intro}Given $e\in]0,e_{0}]$ and $\epsilon\in]0,\infty],$ 

\[
\lim_{N\rightarrow\infty}N^{-1}\log\int_{\left\{ E^{(N)}\in]e-\epsilon,e[\right\} }dV^{\otimes N}=S(e)
\]
 and $\delta_{N}$ converges in probability towards the unique maximum
entropy measure $\mu^{e}$ of energy $e.$ Additionally, if $e\in]0,e_{0}[,$
the convergence is \emph{exponential}, i.e. for any given $\delta>0$
there exists $C>0$ such that
\[
\text{Prob \ensuremath{\left(d(\frac{1}{N}\sum_{i=1}^{N}\delta_{x_{i}},\mu^{e})\geq\delta\right)\leq Ce^{-N/C}}},
\]
where $d$ denotes a fixed metric on $\mathcal{P}(X),$ metrizing
the weak topology. 
\end{thm}

The proof of the previous theorem leverages, in particular, the strict
concavity of $S(e)$ established in the following:
\begin{thm}
\label{thm:micro variational principle low energ intro}Given a number
$e\in]0,e_{0}],$ the specific entropy $S$ coincides with its concave
envelope at $e$ and 
\begin{itemize}
\item There exists a unique maximizer $\mu^{e}$ of the entropy $S(\mu)$
on the subspace of all probability measure $\mu$ satisfying $E(\mu)=e$
or equivalently $E(\mu)\leq e.$ In particular, 
\[
S(\mu^{e})=S(e),
\]
 where $S(e)$ is the specific entropy.
\item The maximizer $\mu^{e}$ is the normalized volume form $\omega_{\beta}^{n}/\int_{X}\omega_{\beta}^{n}$
of the unique solution $\omega_{\beta}$ to the twisted Kähler-Einstein
equation \ref{eq:twisted KE intro} for $\beta\geq0$ such that $E(\omega_{\beta}^{n}/\int_{X}\omega_{\beta}^{n})=e.$ 
\item The corresponding function $e\mapsto\beta(e)$ gives a strictly decreasing
continuous map between $[0,e_{0}]$ and $[0,\infty]$ with the property
that
\[
\beta(e)=\partial S(e)/\partial e
\]
\item As a consequence, $S(e)$ is $C^{1}$ and strictly increasing and
strictly concave on $[0,e_{0}].$
\end{itemize}
\end{thm}

\begin{cor}
\label{cor:conv towards S etc in low energ intro}Given $e\in]0,e_{0}]$
and $\epsilon\in]0,\infty]$ the following convergence holds: 
\begin{equation}
\lim_{N\rightarrow\infty}\int_{X^{N-1}}\mu_{]e-\epsilon,e[}^{(N)}=\mu^{e}.\label{eq:conv of first marginal in Cor}
\end{equation}
\end{cor}

In particular, when $X$ is a \emph{variety of general type, }i.e.
$K_{X}$ is big, we can take $L=K_{X}$ and $(\theta,dV)$ satisfying
the compatibility relation \ref{eq:theta is Ricci}. By \cite{au,y,begz,bbgz},
$X$ admits a unique Kähler-Einstein metric $\omega_{\text{KE}}$
and taking $e=E(\mu_{\text{KE}}),$ where $\mu_{\text{KE}}$ denotes
the normalized volume form $\omega_{\text{KE}}^{n}/\int\omega_{\text{KE}}^{n}$
of $\omega_{\text{KE}},$ the results above yield new constructions
of $\mu_{\text{KE}}$ (and thus of $\omega_{\text{KE}},$ using formula
\ref{eq:omega beta in terms of mu} with for $\beta=1).$ Moreover,
for a general big line bundle $L$ the previous result yield a new
construction of the solution $\omega$ of the Calabi-Yau equation
\ref{eq:cy eq} by letting $e$ increase towards $e_{0}:$
\[
\omega=\lim_{e\rightarrow e_{0}}\frac{i}{2\pi dS(s)/de}\partial\bar{\partial}\left(\log\frac{\mu^{e}}{dV}\right)
\]

\subsection{\label{subsec:The-high-energy intro}Main results in the high energy
regime on Fano varieties $(e>e_{0})$}

Now assume that $X$ is a\emph{ Fano manifold.} This means that its
anti-canonical line bundle $-K_{X}$ is ample. We will then take $L:=-K_{X}$
and the metric on $L$ to be the one induced by $dV$ and assume that
the curvature $\theta$ of the metric is semi-positive, $\theta\geq0.$
More generally, we will allow $X$ to have log terminal singularities
(as discussed in Section \ref{subsec:Singular_complex_varieties}).
In this setup the twisted KE equation \ref{eq:twisted KE intro},
for $\beta\in[-1,0[$ coincide with \emph{Aubin's equation} 
\begin{equation}
\mbox{\ensuremath{\mbox{Ric}}\ensuremath{\omega}}=-\beta\omega+(1+\beta)\theta,\label{eq:Aubin intro}
\end{equation}
appearing in Aubin's method of continuity for producing a Kähler-Einstein
metric on a Fano manifold by decreasing $\beta$ from $0$ towards
$-1$ (with $-\beta$ representing ``time''). The positivity assumption
on $\theta$ ensures that a solution - if it exists - is unique for
any $\beta\in]-1,0[$ \cite{b-m,bern,bbegz,c-d-s}. However, for $\beta=-1,$
i.e. when $\omega$ is a \emph{Kähler-Einstein (KE) metric, }the uniqueness
only holds modulo the action of the group $\text{Aut }(X)_{0}$ of
automorphisms of $X$ homotopic to the identity. By the solution of
the Yau-Tian-Donaldson (YTD) conjecture $X$ admits a KE metric iff
$X$ is \emph{K-polystable} \cite{c-d-s,berman6ii,li1,l-x-z}. This
is an algebro-geometric notion akin to stability in Geometric Invariant
Theory. In this case, Aubin's equation can be solved when $\beta\geq-1.$
In general, by \cite{bbj,c-r-z,li1}, the invariant 
\[
R(X):=\sup\left\{ -\beta\in[0,1]:\,\,\,\exists\omega\,\text{solving equation }\ref{eq:Aubin intro}\right\} 
\]
is strictly positive and may be expressed in algebro-geometric terms
as $R(X)=\min\{1,\delta(X)\},$ where $\delta(X)$ is the invariant
introduced in \cite{f-o} and expressed as a stability threshold in
\cite{bl-j}. 

Given $e\in]e_{0},\infty[$ and $\epsilon\in]0,\infty[$ we now consider
the uniform probability measure on the ``upper energy shell'' $\left\{ E^{(N)}\in]e,e+\epsilon[\right\} $
in $X^{N}:$
\begin{equation}
\mu_{]e,e+\epsilon[}^{(N)}:=1_{\left\{ E^{(N)}\in]e,e+\epsilon[\right\} }dV^{\otimes N}/\int_{\left\{ E^{(N)}\in]e,e+\epsilon[\right\} }dV^{\otimes N},\label{eq:def of upper micro}
\end{equation}
We define the \emph{critical energy} $e_{c}$ (depending on $(X,\theta)$)
as follows: if $X$ admits a KE metric, then 
\begin{equation}
e_{c}:=\text{\ensuremath{\inf}}\left\{ E(\mu):\,\,\mu\in\mathcal{P}(X):\,\,\mu\,\,\text{is Kähler-Einstein}\right\} \,\,\left(\in]e_{0},\infty[\right)\label{eq:def of E c intro}
\end{equation}
where $\mu$ is called KE if it is the volume form of a KE metric.
Otherwise, $e_{c}:=\infty.$
\begin{thm}
\label{thm:micro var principle Fano intro}Assume given $e\in]0,e_{c}[.$
When $e_{c}<\infty$ we also allow $e=e_{c},$ but if $\text{Aut }(X)_{0}$
is non-trivial the uniqueness results below then require the additional
assumption $\theta>0$ on the regular locus $X_{\text{reg}}$ of $X.$ 
\begin{itemize}
\item The specific entropy $S$ coincides with its concave envelope at $e.$
\item There exists a unique maximizer $\mu^{e}$ of the entropy $S(\mu)$
on the subspace of all probability measure $\mu$ with pluri-complex
energy $E(\mu)=e$ (or equivalently $E(\mu)\geq e).$ In particular,
\[
S(\mu^{e})=S(e),
\]
 where $S(e)$ is the specific entropy (formula \ref{eq:def of S intro}). 
\item The maximizer $\mu^{e}$ is the normalized volume form $\omega_{\beta}^{n}/\int_{X}\omega_{\beta}^{n}$
of the unique solution $\omega_{\beta}$ to Aubin's equation \ref{eq:Aubin intro}
such that $E(\omega_{\beta}^{n}/\int_{X}\omega_{\beta}^{n})=e,$ where
$\beta<0$ iff $e>e_{0}.$ 
\item the corresponding function $e\mapsto\beta(e)$ gives a strictly decreasing
continuous map between $[0,e_{c}]$ and $[-R(X),\infty]$ with the
property that
\[
\beta(e)=\partial S(e)/\partial e
\]
\item As a consequence, $S(e)$ is $C^{1}$ and strictly concave on $]0,e_{c}[,$
strictly increasing on $]0,e_{0}[$ and strictly decreasing on $]e_{0},e_{c}[.$ 
\end{itemize}
\end{thm}

When $X$ is K-polystable and $\text{Aut }(X)_{0}$ is trivial the
previous theorem yields a new variational principle for the unique
Kähler-Einstein metric of $X,$ saying its normalized volume form
$\mu_{\text{KE}}$ uniquely maximizes the entropy $S(\mu)$ among
all measures $\mu$ sharing the same pluricomplex energy as $\mu_{\text{KE}}.$
More generally, when $\text{Aut }(X)_{0}$ is non-trivial and $\theta>0$
on $X_{\text{reg}}$ the variational principle holds for the Kähler-Einstein
volume form with minimal pluricomplex energy,

In general, for any Fano variety $S(e)$ is continuous on all of $]0,\infty[$
and unimodal with a maximum at $e=e_{0}.$ However, strict concavity
may fail when $e>e_{c}.$ For example, by the following result it
fails on any homogeneous Fano manifold, for example on $\P^{n}.$ 
\begin{cor}
\label{cor:not K poly}If $X$ is not K-polystable (or equivalently:
$X$ does not admit a KE metric) then $S(e)$ is strictly concave
on $]0,\infty[.$ The converse also holds if $\text{Aut }(X)_{0}$
is non-trivial. In this case $X$ is thus K-polystable iff $S(e)$
is not strictly concave. 
\end{cor}

It would be interesting to know if it is necessary to assume that
$\text{Aut }(X)_{0}$ is non-trivial in the equivalence appearing
in the previous theorem. Anyhow, since $X$ is K-polystable iff $X\times\P^{1}$
is K-polystable \cite{z} the previous theorem shows that, in general,
$X$ is K-polystable iff the specific entropy associated with $X\times\P^{1}$
is not strictly concave on $]0,\infty[.$

In the light of Theorem \ref{thm:LDP for micro intro} it is natural
to pose the following
\begin{conjecture}
\label{conj:intro}Assume given a number $e\in]e_{0},e_{c}[$ and
$\epsilon\in]0,\infty].$ When $X$ admits a unique Kähler-Einstein
metric we also allow $e=e_{c}.$ Then 
\begin{equation}
\lim_{N\rightarrow\infty}N^{-1}\log\int_{\left\{ E^{(N)}\in]e,e+\epsilon[\right\} }dV^{\otimes N}=S(e),\,\,\,\,\lim_{N\rightarrow\infty}\delta_{N}=\mu^{e},\label{eq:conv towards S(e) in conj Fano intro}
\end{equation}
 where the convergence of $\delta_{N}$ holds in probability and $\mu^{e}$
denotes the unique maximum entropy measure $\mu^{e}$ of energy $e.$ 
\end{conjecture}

We establish conditional results in the direction of the conjecture,
assuming the convergence of the partition functions conjectured in
\cite{berm8,be2} (see (Theorems \ref{thm:conditional Fano}, \ref{thm:cond II}).
Since the latter conjecture was established in \cite{be2} when $n=1$
(building on \cite{k,clmp}) we arrive at the following
\begin{thm}
\label{thm:n is one intro}Assume that $\dim X=1$ (i.e. $X=\P^{1}$)
and $e\in[e_{0},e_{c}[$ (and $\epsilon\in]0,\infty]).$ Then the
convergence \ref{eq:conv towards S(e) in conj Fano intro} towards
$S(e)$ holds. If moreover $\theta>0,$ then $\delta_{N}$ converges
\emph{exponentially} in probability at speed $N$ towards $\mu^{e}.$
As a consequence, 
\begin{equation}
\lim_{N\rightarrow\infty}\int_{X^{N-1}}\mu_{]e,e+\epsilon[}^{(N)}=\mu^{e}.\label{eq:conv of first marginal in Cor-1}
\end{equation}
\end{thm}

As discussed in Section \ref{subsec:Generalization-to-log Fano},
the results above naturally generalize to \emph{log Fano varieties}
$(X,\Delta)$ - as appearing in the most general formulation of the
Yau-Tian-Donaldson conjecture \cite{li1,l-x-z} - consisting of a
normal projective variety $X$ and an effective divisor $\Delta$
on $X$ such that $-(K_{X}+\Delta)$ defines an ample $\Q-$line bundle. 

\subsection{\label{subsec:Relations-to-Hamiltonian}Relations to Hamiltonian
dynamics }

Consider the case when $dV$ is a smooth volume form and $L$ is ample.
Then $dV$ can, after rescaling, be expressed as the volume form of
a unique Kähler form $\omega_{X}$ in $c_{1}(X).$ The form $\omega_{X}$
defines a symplectic form on $X$ and thus induces a symplectic form
$\omega_{X^{N}}$ on the $N-$fold product $X^{N}.$ The Hamiltonian
flow of $E^{(N)}$ on $(X^{N},\omega_{X^{N}})$ for initial data of
finite $E^{(N)}(=e\in\R)$ exists for all positive times $t$ (see
Section \ref{sec:Hamiltonian-flows-and}). The measures $\mu_{]e-\epsilon,e[}^{(N)}$
and $\mu_{]e,e+\epsilon[}^{(N)}$ on $X^{N}$ (defined by \ref{eq:def of lower micro intro}
and \ref{eq:def of lower micro intro}, respectively) are invariant
under the Hamiltonian flow of $E^{(N)}$ (since $E^{(N)}$ and $dV^{\otimes N}$
are). Assume, for example, that the limit $\mu^{e}$ of $\mu_{]e-\epsilon,e[}^{(N)}$
exists as $\epsilon\rightarrow0$ and that $\mu^{e}$ is \emph{ergodic.}
Then time-averages along a trajectory $(x_{1}(t),...,x_{N}(t))$ in
$X^{N}$ are, for almost any initial configuration with $E^{(N)}=e,$
distributed according to the probability measure $\mu^{e}$ on $X^{N}.$
Moreover, as explained in Section \ref{sec:Hamiltonian-flows-and},
formal arguments suggest that the Hamiltonian flow of $E^{(N)}$ on
$X^{N}$ converges, as $N\rightarrow\infty,$ to a solution of the
following evolution equation on $\mathcal{P}(X):$
\begin{equation}
\frac{\partial\rho_{t}}{\partial t}=-\nabla\cdot(\rho_{t}J\nabla u_{\rho_{t}}),\,\,\,\frac{1}{\text{Vol \ensuremath{(L)}}}(\frac{i}{2\pi}\partial\bar{\partial}u_{\rho_{t}}+\theta)^{n}=\rho_{t}\omega_{x}^{n},\label{eq:incomr complex Euler}
\end{equation}
 where the gradient and the scalar product are defined wrt the Kähler
metric $\omega_{X}.$ This equation is a higher dimensional generalization
of the incompressible 2D Euler equation (corresponding to $n=1)$.
Accordingly, it will be called the\emph{ (incompressible) complex
Euler-Monge-Ampère equation}. When $n=2$ and $X$ is a torus, $X=(S^{1}\times iS^{1})\times(S^{1}\times iS^{1}),$
solutions to equation \ref{eq:incomr complex Euler} that are independent
of the imaginary factors correspond to solution to the periodic semi-geostrophic
equation in $\R^{2}$ \cite{be-br,Loe} (describing large-scale atmospheric
flows). 

As explained in Section \ref{subsec:Formal-proof using Ham} the equation
\ref{eq:incomr complex Euler}, can be viewed as a Hamiltonian flow
of $E(\mu)$ on $\mathcal{P}(X)$ (or more precisely, along the coadjoint
orbits in $\mathcal{P}(X)$ of the group of Hamiltonian diffeomorphisms
of $X,$ endowed with the Kirillov-Kostant-Souriau symplectic form\emph{).}
Assuming ergodicity, Theorems \ref{thm:LDP for micro intro}, \ref{thm:micro variational principle low energ intro}
thus suggests that, as $N\rightarrow\infty,$ large-time averages
of a trajectory of the Hamiltonian flow on $X^{N},$ emanating from
almost any initial configuration of energy $e\in]0,e_{0}[,$ are distributed
according to the maximum entropy measure $\mu^{e}.$ \footnote{Even if ergodicity does not hold the microcanonical measures are still
relevant for the large time-behavior under the weaker form of ergodicity
discussed in \cite[page 865]{e-s}.} This indicates that the density of $\mu^{e}$ is a stationary solution
of the evolution equation \ref{eq:incomr complex Euler}, which is
confirmed in Section \ref{subsec:Conserved-quantities}. In the light
of the analogy with turbulent 2D flows, this suggests a picture where
twisted Kähler-Einstein metrics emerge as a coherent large-scale structure
in a fluid on $X,$ microscopically described by the Hamiltonian dynamics
of the ``discrete'' pluricomplex energy $E^{(N)}.$

\subsection{Relations to statistical mechanics and previous results}

\subsubsection{\label{subsec:Equivalence-of-the}Equivalence of the microcanonical
and canonical ensembles }

Let $(X,\omega_{X})$ be a compact symplectic manifold and consider
a lsc symmetric function $E^{(N)}$ on $X^{N}.$ Assume that, for
a given $e\in\R,$ the level set $\{E^{(N)}=e\}$ is a submanifold
of $X^{N}.$ Then the probability measures $\mu_{]e-\epsilon,e[}^{(N)}$
and $\mu_{]e,e+\epsilon[}^{(N)}$ in formula \ref{eq:def of lower micro intro}
and formula \ref{eq:def of upper micro}, respectively, both converge,
as $\epsilon\rightarrow0,$ towards the same measure $\mu^{e},$ supported
on the submanifold $\{E^{(N)}=e\}$ of $X^{N},$ called the \emph{microcanonical
measure} in statistical mechanics. The corresponding probability space
$(X^{N},\mu^{e})$ is the \emph{microcanonical ensemble} representing
the equilibrium state of of $N$ interacting particles on $X$ at
a fixed energy level $e.$ However, since, in general, $\{E^{(N)}=e\}$
need not be regular regularized microcanonical measures of the form
$\mu_{]e-\epsilon,e[}^{(N)}$ and $\mu_{]e,e+\epsilon[}^{(N)}$ are
frequently used, with the ``thickness'' $\epsilon$ serving as a
regularization parameter \cite{clmp2}.

Under the following ``mean field type'' assumption the convergence
towards $S(e)$ and the exponential convergence of $\delta_{N}$ towards
$\mu^{e}$ in Theorem \ref{thm:LDP for micro intro} follows from
results in \cite{e-s,e-h-t}: 
\begin{equation}
\lim_{N\rightarrow\infty}\sup_{X^{N}}\left|E^{(N)}(x_{1},...x_{N})-E(\delta_{N}(x_{1},...x_{N}))\right|=0\label{eq:strong reg assu on E intro}
\end{equation}
for a continuous functional $E(\mu)$ on $\mathcal{P}(X),$where $X$
is assumed compact. More precisely, as shown in \cite{e-s,e-h-t},
given $a,b\in\R$ the law of the empirical measure $\delta_{N}$ -
viewed as a random variable on $(X^{N},\mu_{]a,b[}^{(N)})$ - satisfies
a Large Deviation Principle (LDP) at speed $N$ with a rate functional
which is equal to the following functional, up to an additive constant:
\begin{equation}
I(\mu)=-S(\mu)\,\,\text{if \ensuremath{E(\mu)\in]a,b[,\,\,\,I(\mu)=\infty\,\,\text{if \ensuremath{E(\mu)\notin]a,b[}}}}\label{eq:def of I}
\end{equation}
In contrast, in the present setup both $E^{(N)}(x_{1},...x_{N})$
and $E(\mu)$ are highly singular. For example, $E^{(N)}(x_{1},...x_{N})\rightarrow\infty$
as soon as two points merge and $E(\delta_{N}(x_{1},...x_{N}))=\infty$
for any $(x_{1},...x_{N})\in X^{N}.$ For such singular situations
the aforementioned LDP fails, in general. Indeed, as shown in forthcoming
work \cite{ber12} the LDP may fail for two different reasons. First
of all, it may happen that the infimum of the functional $I(\mu)$
appearing in formula \ref{eq:def of I} equals $-\infty.$ Secondly,
even if the infimum in question is finite and uniquely attained, the
LDP may fail already in the low-energy region. For example, this is
the case when $E^{(N)}$ is a singular repulsive power-law (e.g. the
Coulomb interaction in a compact domain $X$ of $\R^{n}$ for $n\geq3).$ 

While the proof of the microcanonical LDP in \cite{e-s,e-h-t} is
reduced to Sanov's classical LDP (corresponding to the case when $E^{(N)}$
is constant), the proof of Theorem \ref{thm:LDP for micro intro}
is inspired by the notion of ``equivalence of ensembles'' in statistical
mechanics. It suggests that in certain situations the microcanonical
ensemble $(X^{N},\mu^{e})$ is, when $N\rightarrow\infty,$ equivalent
to the \emph{canonical ensemble }$(X^{N},\nu_{\beta}^{(N)}),$ where
$\beta\in\R$ and $\nu_{\beta}^{(N)}$ is the \emph{Gibbs measure},
whose density is proportional to $e^{-\beta E^{(N)}},$ representing
a thermal equilibrium state at inverse temperature $\beta.$ More
precisely, following \cite{e-s,e-h-t} \emph{(thermodynamic) equivalence
of ensembles} is said to hold at energy level $e$ if the specific
entropy $S$ coincides with its concave envelope $S^{**}$ at $e.$
If moreover $S$ is differentiable at $e,$ then the inverse temperature
$\beta$ corresponding to the energy level $e$ is 
\[
\beta:=dS(e)/de
\]
(which is the classical thermodynamic definition of inverse temperature)
and if $S$ is strictly concave this relation can be reversed. Here,
after establishing the differentiability and strict concavity of $S(e)$
(as stated in Theorem \ref{thm:micro variational principle low energ intro})
the convergence results in Theorem \ref{thm:LDP for micro intro}
are deduced from the Large Deviation Principle (LDP) for the corresponding
Gibbs measures at $\beta>0,$ established in \cite{berm8,berm8 comma 5}.
More generally, the convergence results are shown to hold in a rather
general setup where $E^{(N)}$ is assumed quasi-superharmonic and
$E(\mu)$ is convex (Theorem \ref{thm:LDP for quasi super}). 

An interesting feature that appears when $e$ is in the high energy
region, $e>e_{0},$ is that, by Theorem \ref{thm:micro var principle Fano intro},
the energy level $e$ corresponds to $\beta<0,$ i.e. a state of \emph{negative}
temperature. Accordingly, the convergence results in Theorem \ref{thm:n is one intro},
where $\dim X=1,$ are deduced from the LDP for the corresponding
Gibbs measures with $\beta<0$ in \cite{berm11b,be2} (that build
on \cite{clmp,k}).

The proof of Theorems \ref{thm:micro variational principle low energ intro},
\ref{thm:micro var principle Fano intro} exploit that, in general,
the Legendre transform $S^{*}(\beta)$ of $S(e)$ coincides with the\emph{
free energy} $F(\beta)$ at inverse temperature $\beta$ (formula
\ref{eq:def of F beta of F beta mu intro}). The proofs are reduced
to showing that $F(\beta)$ is strictly concave and $C^{1}-$differentiable.
The fact that the functional $E(\mu)$ is not continuous leads to
some new features, as compared to the setup in \cite{e-s,e-h-t},
further explored in \cite{ber12}. In particular, in order for the
equivalence of ensembles to hold when $e$ decreases towards the minimum
of $E(\mu)$ the measure $dV$ needs, as shown in \cite{be1,ber9},
to satisfy a condition dubbed the ``energy approximation property'',
which in the present setup admits a potential-theoretic characterization
(see Prop \ref{Prop:properties when E convex and energy app} and
the proof of Lemma \ref{lem:on polarized has affine cont and energy appr}). 

\subsubsection{The infimum of the Mabuchi functional and Aubin's method of continuity}

As observed in \cite{berm6}, in the present complex-geometric setup
the free energy $F(\beta)$ coincides with the infimum of the \emph{twisted
Mabuchi functional} on the space of Kähler metrics in $c_{1}(L)$
\cite{ma} and the inverse temperature $\beta$ corresponds to the
parameter $\beta$ in the twisted Kähler-Einstein equation \ref{eq:twisted KE intro}.
In the course of the proof of Theorem \ref{thm:micro var principle Fano intro}
we extend some essentially well-known results for $F(\beta)$ and
Aubin's method of continuity (going back to \cite{b-m}) - concerning
the case when $L$ is ample and $X$ and $dV$ are smooth - to the
case when $L$ is merely big and $X$ and $dV$ are allowed to be
singular. For example, for any Fano variety $X$ Lemma \ref{lem:conv towards KE on Fano}
shows that if $X$ admits some Kähler-Einstein metric and $\theta>0$
on the regular locus of $X,$ then the solution $\omega_{\beta}$
of Aubin's equation \ref{eq:Aubin intro} converges, as $\beta$ is
decreased to $-1$ to the unique Kähler-Einstein metric $\omega_{\text{KE}}$
on $X$ with minimal energy (i.e. its volume form minimizes $E(\mu)$
among all Kähler-Einstein metrics on $X).$ In contrast to the method
in \cite{b-m} - which requires that $X$ be non-singular, since it
relies on linearising the equations - we use a variational approach,
building on \cite{bbegz}. 

\subsubsection{Point vortices}

The case when $n=1$ and $X$ is the Riemann sphere is closely related
to results in \cite{clmp2,k,k-w}, concerning Onsager's statistical
mechanical approach to the incompressible 2D Euler equation \cite{o}.
In this approach the vorticity of the fluid (i.e. the rotation of
the velocity field) is replaced by a sum of $N$ Dirac masses located
at points $(x_{1},..,x_{N}).$ Formally, the corresponding solution
of the Euler equations yields a Hamiltonian flow on $X^{N},$ whose
large time dynamics is - assuming ergodicity - described by corresponding
microcanonical ensemble, with $E^{(N)}$ defined as sum of pair interactions
$W(x_{i},x_{j}),$ where $-W(x,y)$ is a Green function for the Laplacian.
In particular, Onsager made a striking prediction about the existence
of negative temperature states, corresponding to energy levels $e>e_{0},$
which was experimentally confirmed only recently in 2D quantum superfluids
\cite{gau,jo}. The connection to \cite{clmp2,k,k-w} is detailed
in Section \ref{sec:Comparison-with-point} - the main new feature
is the exponential convergence. It should, however, be stressed that
the situation when $n>1$ is considerably more non-linear than the
case when $n=1.$ Indeed, when $n>1$ the role of the Laplace operator
is played by the complex Monge-Ampère operator, which is fully non-linear.
Consequently, while $E^{(N)}$ allows for a straightforward regularization
in the $n=1$ case\textemdash achieved by regularizing the Green function
of the Laplacian $-$ an effective regularization for $E^{(N)}$ when
$n>1$ remains elusive.

\subsection{Acknowledgements}

I am greatful to Mingchen Xia for discussions and very helpful comments
and corrections. Thanks also to Rolf Andreasson, Sébastien Boucksom,
Emanuele Caglioti, Klas Modin for discussions. This work was supported
by a Wallenberg Scholar grant from the Knut and Alice Wallenberg foundation.

\section{General (abstract) results}

\subsection{Preliminaries}

\subsubsection{Concavity}

Recall that a function $\phi$ on a convex subset $C$ of $\R^{d}$
taking values in $]-\infty,\infty]$ is said to be \emph{convex} on
$C$ if for any given two points $x_{0}$ and $x_{1}$ and $t\in]0,1[$
\[
\phi(tx_{0}+(1-t)x_{1})\leq t\phi(x_{0})+(1-t)\phi(x_{1})
\]
 and\emph{ strictly convex} on $C$ if the inequality above is strict
for any $t\in]0,1[.$ A function $f$ on $C$ is \emph{(strictly)
concave} if $-f$ is (strictly) convex. Here we will be mainly concerned
with the case when $d=1.$ In this case, if $f$ is concave and finite
on a closed interval $C\subset\R$, but not strictly convex, then
there exist two points $x_{0}$ and $x_{1}$ in $C$ such that $f$
is affine on $[x_{0},x_{1}].$ 

If $\phi$ is a convex function on $\R^{d}$ then its\emph{ subdifferential}
$(\partial\phi)$ at a point $x_{0}\in\R^{d}$ is defined as the convex
set
\begin{equation}
(\partial\phi)(x_{0}):=\left\{ y_{0}:\,\phi(x_{0})+y_{0}\cdot(x-x_{0})\leq\phi(x)\,\,\,\forall x\in\R^{d}\right\} \label{eq:sub gradient prop}
\end{equation}
In particular, if $\phi(x_{0})=\infty,$ then $(\partial\phi)(x_{0})$
is empty. Similarly, if $f$ is concave on $\R^{d}$ then its \emph{superdifferential}
$(\partial f)(x_{0})$ is defined as above, but reversing the inequality.
In other words, $(\partial f)(x_{0}):=$ $-(\partial(-f)(x_{0}).$
In the case when $f$ is concave on $\R$ and finite in a neighborhood
of $x_{0}$ 
\[
(\partial f)(x)=[f'(x+),f'(x-)],
\]
 where $f'(x+)$ and $f'(x-)$ denote the right and left derivatives
of $f$ at $x,$ respectively. In particular, $f$ is differentiable
at $x$ iff $(\partial f)(x)$ consists of a single point. If $f$
is a function on $\R^{d}$ taking values in $[-\infty,\infty]$ its
(concave) \emph{Legendre transform }is the usc and concave function
on $\R^{d}$ (taking values in $[-\infty,\infty[$) defined by

\[
f^{*}(y):=\inf_{x\in\R^{n}}\left(\left\langle x,y\right\rangle -f(x)\right),
\]
 and satisfies $f^{**}=f.$ It follows readily from the definitions
that 
\begin{equation}
y\in\partial f(x)\iff x\in\partial f^{*}(y).\label{eq:y in gradient iff x in gradient}
\end{equation}

\begin{rem}
\label{rem:top vector space}All the previous properties apply more
generally when $x\in\mathcal{M}$ for a locally convex topological
vector space if $y$ is taken to be in the topological dual $\mathcal{M}^{*}$
of $\mathcal{M}$ \cite{d-z}.
\end{rem}

We will make use of the following basic lemmas (see the appendix of
\cite{ber9} for proofs).
\begin{lem}
\label{lem:diff implies dual strc conc}Let $f$ be a concave function
on $\R$ and assume that $f$ is differentiable in a neighborhood
of $[x_{0},x_{1}].$ Then $f^{*}$ is strictly concave in the interior
of $[y_{0},y_{1}]:=[f'(x_{1}),f'(x_{0})].$ 
\end{lem}

Note that, in general, $f^{**}\geq f.$ Concerning the strict inequality
we have the following
\begin{lem}
\label{lem:convex envol affine}Let $f$ be a function on $\R$ such
that $\sup_{\R}f<\infty$ and $U\Subset\R$ a non-empty open set where
$f$ is finite and usc. Then $\{f^{**}>f\}\cap U$ is open in $U$
and $f^{**}$ is affine on $\{f^{**}>f\}\cap U.$ 
\end{lem}

\subsection{General setup}

The following setup is referred to as the (Very General Setup) in
\cite{ber9} (when $X$ is allowed to be non-compact). Let $X$ be
a compact topological space, $E(\mu)$ a lsc functional on the space
$\mathcal{P}(X)$ of all probability measures on $X$ and $dV$ a
measure on $X$ such that $e_{0}<\infty,$ where
\[
e_{0}:=E(\mu_{0})<\infty,\,\,\mu_{0}:=dV/\int_{X}dV,\,\,\,e_{\text{min}}:=\inf_{\mathcal{P}(X)}E(\mu),\,\,\,e_{\text{max}}:=\sup_{\mathcal{P}(X)}E(\mu)
\]
(in the complex-geometric setup $e_{\text{min }}=0$ and $e_{\text{max}}=\infty)$.

\subsubsection{Entropy vs free energy}

In this general setup we define the entropy $S(\mu)$ (relative to
$dV)$ and the corresponding \emph{specific entropy} $S(e)$ (at given
energy level $e\in\R$) as in formula \ref{eq:def of S intro}. Given
$\beta\in\R$ the \emph{free energy $F(\beta)$ }(at inverse temperature
$\beta$) is defined by 

\begin{equation}
F(\beta)=\inf_{\mathcal{\mu\in P}(X)}F_{\beta}(\mu),\,\,\,F_{\beta}(\mu):=\beta E(\mu)-S(\mu),\label{eq:def of F beta of F beta mu intro}
\end{equation}
(where $F(\mu)$ is defined to be equal to $+\infty$ if $S(\mu)=-\infty).$
It follows ready from the definitions that, $F(\beta)$ is the Legendre
transform of $S(e):$ 
\begin{equation}
F(\beta)=S^{*}(\beta)\label{eq:F is S star}
\end{equation}
 (but, in general, it may happen that $F^{*}\neq S,$ since $S$ need
to be concave nor usc \cite{clmp2,e-h-t}). We recall the following
result \cite[Thm 4.4a]{e-h-t} \cite[Lemma 4.1]{ber9}:
\begin{lem}
\label{lemma:macrostate equivalence}(macrostate equivalence of ensembles).
Assume that $S^{**}(e)=S(e)>-\infty$ and that $S^{**}=S>-\infty$
in a neighborhood of $e.$ If $\mu^{e}$ is a maximal entropy measure
with energy $e,$ i.e. $S(\mu^{e})=S(e)$ and $E(\mu)=e,$ then $\mu^{e}$
minimizes the free energy functional $F_{\beta}(\mu)$ for any $\beta\in\partial S(e).$ 
\end{lem}

\subsubsection{The microcanonical vs the canonical ensembles}

Given a lsc function $E^{(N)}$ on the $N-$fold product $X^{N}$
of $X,e\in\R$ and $\epsilon>0$ we define the corresponding \emph{lower}
and \emph{upper microcanonical measures} $\mu_{]e-\epsilon,e[}^{(N)}$
and $\mu_{]e,e+\epsilon[}^{(N)}$ as the probability measures on $X^{N}$
defined by formula \ref{eq:def of lower micro intro} and formula
\ref{eq:def of upper micro}, respectively. Given $\beta\in\R$ the
corresponding \emph{Gibbs measure} at inverse temperature $\beta$
(aka the \emph{canonical measure}) is defined by 
\begin{equation}
\nu_{\beta}^{(N)}:=e^{-\beta NE^{(N)}}dV^{\otimes N}/Z_{N,\beta},\,\,\,\,Z_{N,\beta}:=\int_{X^{N}}e^{-\beta NE^{(N)}}dV^{\otimes N},\label{eq:Gibbs meas}
\end{equation}
 which is a well-defined probability measure on $X^{N}$ when $Z_{N,\beta}<\infty.$

\subsection{Exponential convergence }

Recall that a sequence $Y_{N}$ of random variables taking values
in a metric space $(\mathcal{Y},d)$ is said to convergence \emph{exponentially}
(in probability) towards an element $y\in\mathcal{Y}$ at \emph{speed}
$R_{N}$ (where $R_{N}$ is a sequence tending to infinity) if for
any $\delta>0$ there exists a constant $C>0$ such that 
\[
\text{Prob \ensuremath{\left\{ d(Y_{N},y)\geq\delta\right\} }}\leq Ce^{-R_{N}/C}.
\]
 We will use the following criterion for exponential convergence in
the space $\mathcal{P}(X)$ of probability measures on a compact topological
space $X,$ endowed with its standard weak topology (which is metrizable
if $X$is a metric space).
\begin{lem}
\label{lem:Cramer variant}Let $X$ be a compact topological space
and $\Gamma_{N}$ a sequence of probability measures on $\mathcal{P}(X),$
$R_{N}$ a sequence in $\R$ tending to infinity and assume that the
function on $C^{0}(X)$ defined by
\begin{equation}
\Lambda(u\text{)}:=\limsup_{N\rightarrow\infty}R_{N}^{-1}\log\left\langle \Gamma_{N},e^{R_{N}\left\langle u,\cdot\right\rangle }\right\rangle ,\,\,\,u\in C^{0}(X)\label{eq:def of Lambda u}
\end{equation}
is finite and Gateaux differentiable at $0\in C^{0}(X)$ with differential
$\mu_{0}\in\mathcal{P}(X).$ Then, for any compact subset $\mathcal{K}$
of $\mathcal{P}(X)$ not containing $\mu_{0},$ there exists a constant
$C>0$ such that
\[
\Gamma_{N}(\mathcal{K})\leq Ce^{-R_{N}/C}.
\]
 In particular, if $\Gamma_{N}$ is the law of a random variable $Y_{N},$
then $Y_{N}$ converges exponentially at speed $R_{N}$ towards $\mu_{0},$
as $N\rightarrow\infty.$ Moreover, if $X$ is a compact manifold,
then $C^{0}(X)$ can be replaced with $C^{k}(X)$ for any positive
integer $k.$
\end{lem}

\begin{proof}
We start with the case when $X$ is a compact topological space and
endow $C^{0}(X)$ with the Banach norm $\left\Vert \cdot\right\Vert $
defined by the sup-norm, called the \emph{strong }topology. 
\begin{equation}
\text{Claim:}\,\,\,\Lambda\,\,\text{is lower-semicontinuous (lsc) on }C^{0}(X).\label{eq:claim pf exp conv}
\end{equation}
To prove this first note that the function $\Lambda_{N}$ on $C^{0}(X)$
defined by 
\[
\Lambda_{N}(u\text{)}:=R_{N}^{-1}\log\left\langle \Gamma_{N},e^{R_{N}\left\langle u,\cdot\right\rangle }\right\rangle 
\]
 is Lipschitz continuous with Lipschitz constant $1.$ Indeed, since
$\Gamma_{N}$ is a probability measure on $\mathcal{P}(X)$ the function
$\Lambda_{N}$ is increasing wrt the standard order relation on $C^{0}(X)$
and satisfies $\Lambda_{N}(\cdot+c)=\Lambda_{N}(\cdot)+c$ for any
$c\in\R.$ It follows that, for any given $u_{0}$ and $u_{1}$ in
$C^{0}(X),$ 
\[
\Lambda_{N}(u_{0})\leq\Lambda_{N}(u_{1}+\left\Vert u_{0}-u_{1}\right\Vert )=\Lambda_{N}(u_{1})+\left\Vert u_{0}-u_{1}\right\Vert ,
\]
 which proves that $\Lambda_{N}$ has Lipschitz constant $1$ (after
interchanging the roles of $u_{0}$ and $u_{1}).$ Finally, the claim
above follows form the fact that the point-wise limsup $\Lambda$
of a sequence of functions $\Lambda_{N}$ having the same Lipschitz
constant (say $1)$ is lower semi-continuous. Indeed, 
\[
\Lambda(u):=\lim_{n\rightarrow\infty}\left(\sup_{m\geq n}\Lambda_{N}(u\text{)}\right)\leq\lim_{n\rightarrow\infty}\left(\sup_{m\geq n}\Lambda_{N}(u_{j}\text{)}+\left\Vert u-u_{j}\right\Vert \right)\leq\lim_{n\rightarrow\infty}\left(\sup_{m\geq n}\Lambda_{N}(u_{j}\text{)}\right)+\left\Vert u-u_{j}\right\Vert ,
\]
 showing that $\Lambda(u)\leq\Lambda(u_{j})+\left\Vert u-u_{j}\right\Vert .$
Letting $j\rightarrow\infty$ thus concludes the proof of the claim
above.

In the rest of the proof it will be convenient to follow the notation
in \cite{d-z} involving convex functions (rather than concave functions)
and thus define the Legendre transform $\Lambda^{*}$ of a function
$\Lambda$ on a topological vector space $\mathcal{U}$ as the following
convex lsc function on $\mathcal{U}^{*}:$
\[
\Lambda^{*}(\mu):=\sup_{u\in\mathcal{M}^{*}}\left(\left\langle \mu,u\right\rangle -\Lambda(u)\right)
\]
and the \emph{sub-differential }$\partial\phi(x)$ of a convex function
$\phi$ on a topological vector space $\mathcal{X}$ as the set of
all $y\in\mathcal{X}^{*}$ such that $\phi(\cdot)\geq\phi(x)+\left\langle y,\cdot-x\right\rangle .$
By \cite[Thm 4.5.3]{d-z}, for any compact subset $\mathcal{K}$ of
a Hausdorff topological vector space $\mathcal{M}$ 
\begin{equation}
\limsup_{N\rightarrow\infty}R_{N}^{-1}\log\Gamma_{N}(\mathcal{K})\leq-\inf_{\mathcal{K}}\Lambda^{*},\label{eq:upper bound on Laplace transform}
\end{equation}
 where $\Lambda$ is defined by formula \ref{eq:def of Lambda u}
on $\mathcal{M}^{*}.$ Given a compact topological space $X$ we will
take $\mathcal{M}$ to be the space $\mathcal{M}(X)$ of all signed
measures on $X,$ endowed with the weak{*}-topology wrt the test functions
$C^{0}(X).$ This topology is \emph{Hausdorff.} Indeed, in general,
the weak{*} topology is locally convex (with semi-norms $\left|\left\langle \cdot,u\right\rangle \right|$
for $u\in C^{0}(X)).$ In this case it also separated (i.e. $\left\langle \mu,u\right\rangle =0$
for all $u\in C^{0}(X)$ implies $\mu=0).$ Hence, it is Hausdorff.
Then the topological space $\mathcal{M}(X)$ can be identified with
the topological dual $C^{0}(X)^{*}$ of $C^{0}(X)$ (by Riesz representation
theorem). Moreover, the natural embedding of $C^{0}(X)$ into $C^{0}(X)^{**}$
is a bijection. Indeed, if $F$ is a linear continuous function on
$\mathcal{M}(X)$ we get $F(\mu)=\left\langle \mu,u\right\rangle $
for the continuous function $u(x):=F(\delta_{x})$ on $X$ (by expressing
$\mu\in\mathcal{M}(X)$ as a limit of a finite linear combinations
of Dirac masses). 

Since $\Lambda^{*}$ is lsc (being the sup of a family of continuous
convex functionals), it will, by the inequality \ref{eq:upper bound on Laplace transform},
be enough to show that $\Lambda^{*}$ has a unique minimum (at $\mu_{0}).$
In fact, in the present setup both $\Lambda^{*}$ and $\Lambda$ are
lsc convex functions. Indeed, according to the claim \ref{eq:claim pf exp conv}
$\Lambda$ is lsc on $C^{0}(X)$ and, in general, both the sup and
the limsup of a family of convex functions are convex. Next note that
the infimum of $\Lambda^{*}$ is attained at $\mu\in C^{0}(X){}^{*}$
iff $0\in\partial\Lambda^{*}$ (by the very definition of the sub-differential).
Now, in general, for any lsc convex functional $\Lambda$ defined
on a locally convex topological vector space $\mathcal{U}$ (here
$C^{0}(X),$ endowed with the strong topology) we have that 
\begin{equation}
\mu\in\partial\Lambda(u)\iff u\in\partial\Lambda^{*}(\mu).\label{eq:duality}
\end{equation}
Indeed, in general if a function $\Lambda(u)$ on a topological vector
space $\mathcal{U}$ is convex then, by the very definition of sub-differential,
\[
\Lambda(u)+\Lambda^{*}(\mu)=\left\langle \mu,u\right\rangle \iff\mu\in\partial\Lambda(u),
\]
 where $\Lambda^{*}$ is defined on $\mathcal{U}^{*}.$ But any lsc
convex function $\Lambda$ on a locally convex topological vector
space satisfies $(\Lambda^{*})^{*}=\Lambda$ (by the Hahn-Banach theorem;
see \cite[Lemma 4.5.8]{d-z}). Thus, applying the equivalence \ref{eq:duality}
to the convex function $\Lambda^{*}$ on $\mathcal{U}^{*}$ gives
\[
\Lambda^{*}(\mu)+\Lambda(u)=\left\langle \mu,u\right\rangle \iff u\in\partial\Lambda^{*}(\mu)
\]
 which proves \ref{eq:duality}. Hence, the infimum of $\Lambda^{*}$
is attained at $\mu\in\mathcal{U}^{*}$ iff $\mu\in\partial\Lambda(0).$
But since $\Lambda$ is assumed Gateaux differentiable at $0$ with
differential $\mu_{0}$ we have that $\partial\Lambda(0)=\{\mu_{0}\}.$
Indeed, this is classical when $\mathcal{U}=\R.$ If $\mu\in\partial\Lambda(0)$
restricting to an affine line thus forces $\left\langle \mu,u\right\rangle =\left\langle \mu_{0},u\right\rangle $
for any $u\in\mathcal{U}$ which implies that $\mu=\mu_{0},$ since
$\mathcal{U}=\mathcal{M}^{*}$ (as sets). Hence, $\mu=\mu_{0},$ as
desired. 

Next assume that $X$ is a compact manifold. Denote by $H^{s}(X)$
the Sobolev space of functions $f$ with $s$ derivatives in $L^{2}(X).$
We endow $H^{s}(X)$ with the strong topology, defined by the standard
Hilbert norm on $H^{s}(X)$ induced by a fixed metric $g$ on $X:$
$\left\Vert f\right\Vert _{s}^{2}:=\left\langle f+(-\Delta_{g})^{s}f,fdV_{g}\right\rangle ,$
where $\Delta_{g}$ and $dV_{g}$ denote the Laplacian and volume
form wrt $g.$ We denote by $H^{s}(X)^{*}$ the topological dual of
$H^{s}(X),$ endowed with the weak{*} topology. By the Riesz representation
theorem in Hilbert spaces the natural inclusion of $H^{s}(X)$ into
$H^{s}(X)^{**}$ is a bijection. It also follows from the Riesz representation
that $H^{s}(X)^{*}$is separated, hence\emph{ Hausdorff.} Furthermore,
by the Sobolev embedding theorem \cite{au3}, there is a continuous
embedding 
\[
H^{s}(X)\hookrightarrow C^{k}(X),
\]
 for $s$ sufficiently large, defined by the inclusion map. We fix
such an $s.$ Combining the Sobolev embedding with the basic fact
that $H^{s}(X)$ is dense in $C^{0}(X)$ (endowed with the sup-norm)
yields a continuous embedding 
\[
j:\,\mathcal{P}(X)\hookrightarrow H^{s}(X)^{*},\,\:\mu\mapsto\int\cdot\mu.
\]
Now fix a compact subset $\mathcal{K}$ of $\mathcal{P}(X).$ Since
continuous maps preserve compact subsets, $j(\mathcal{K})$ is compact
in $\mathcal{X}:=H^{s}(X)^{*}.$ Hence, applying the inequality \ref{eq:upper bound on Laplace transform}
to $\mathcal{M}:=H^{s}(X)^{*}$ gives
\[
\limsup_{N\rightarrow\infty}R_{N}^{-1}\log\Gamma_{N}(\mathcal{K})=\limsup_{N\rightarrow\infty}R_{N}^{-1}\log\left((j_{*}\Gamma_{N})(j(\mathcal{K}))\right)\leq-\inf_{j(\mathcal{K})}\Lambda^{*},
\]
 where $\Lambda$ is defined by formula \ref{eq:def of Lambda u}
on $H^{s}(X)^{**},$ which as a set may be identified with $H^{s}(X),$
which is embedded in $C^{k}(X).$ It will thus be enough to show that
the function $\Lambda^{*}$ on $H^{s}(\mathcal{X})^{*}$ has a unique
minimum at $\mu_{0}.$ But since we have a continuous embedding $H^{s}(X)\hookrightarrow C^{0}(X)$
we get, as before, that $\Lambda$ is lsc on $H^{s}(X)$ (endowed
with the strong topology). We can then thus conclude the proof by
repeating the previous argument, that only uses that $\mathcal{U}(:=H^{s}(X)$
with the strong topology) is a locally convex topology vector space.
\end{proof}
The previous lemma can be seen as an infinite dimensional version
of the criterion for exponential convergence of random variables with
values in $\R^{m}$ in \cite[Thm II.6.3]{el}. We will also have use
for the following one-dimensional version of the previous lemma (which
coincides with the case $m=1$ of \cite[Thm II.6.3]{el} when $T=\infty)$:
\begin{lem}
\label{lem:Cramer variant i R}Let $E_{N}$ be a sequence of real-valued
random variables. Assume that there exists $T>0$ such that for any
$t\in]-T,T[$ 
\[
\lim_{N\rightarrow\infty}R_{N}^{-1}\log\E(e^{tR_{N}E_{N}})=f(t)
\]
for sequence of numbers $R_{N}\rightarrow\infty$ and a function $f$
on $]-T,T[$ that is differentiable at $t=0.$ Then $E_{N}$ converges
exponentially in probability at speed $R_{N}$, as $N\rightarrow\infty,$
to $f'(0).$ 
\end{lem}

\begin{proof}
When the law of $E_{N}$ is supported on a compact subset of $\R$
the result follows from the previous lemma. To prove the general case
observe that for a given $e\in\R,$ we have
\[
\text{Prob \ensuremath{\left\{ E_{N}\geq e+\delta\right\} }}\leq e^{-R_{N}t(e+\delta)}\E(e^{tR_{N}E_{N}})
\]
for any $t\in\R.$ Hence, for any $t\in[-T,T],$ 
\[
\limsup_{N\rightarrow\infty}R_{N}^{-1}\log\text{Prob \ensuremath{\left\{ E_{N}\geq e+\delta\right\} }}\leq-t(e+\delta)+f(t)
\]
In particular, taking $e=f'(0)$ and using the assumption that $f(t)=tf'(0)+to(t),$

\[
\limsup_{N\rightarrow\infty}R_{N}^{-1}\log\text{Prob \ensuremath{\left\{ E_{N}\geq e+\delta\right\} }}\leq-t\delta+to(t)=-t\left(\delta+o(t)\right)
\]
 where $o(t)\rightarrow0$ as $t\rightarrow0.$ In particular, given
any $\delta>0$ we can take $t$ sufficiently small and strictly positive
so that the rhs above is strictly negative, as desired. Finally, repeating
the same argument with $E_{N}$ replaced by $-E_{N}$ concludes the
proof of the exponential convergence. 
\end{proof}

\subsection{General convergence towards the specific entropy as $N\rightarrow\infty$}

We start with the following elementary
\begin{lem}
\label{lem:differentiability of sup}Let $G_{0}$ and $v$ be functions
on a topological space $\mathcal{X}$ and set $G_{t}:=G_{0}+tv$ for
$t\in[-T,T].$ Assume that $G_{t}$ admits a unique maximizer $x_{t}\in\mathcal{X}$
for any $t\in[-T,T]$ and that $t\mapsto v(x_{t})$ is continuous.
Then 
\[
g(t):=\sup_{\mathcal{X}}(G_{0}+tv)
\]
is differentiable at $t=0$ and $g'(0)=v(x_{0}).$
\end{lem}

\begin{proof}
This is shown as in the appendix of \cite{ber-ber} (where it was
only assumed that $G_{0}$ admits a unique minimizer, since $G_{0}$
was assumed usc). Note that in the present setup \cite[formula 6.1 ]{ber-ber}
holds, by assumption.
\end{proof}
\begin{lem}
\label{lem:F cont diff}Assume that, for any $\beta$ in $]\beta_{0},\beta_{1}[,$
$F_{\beta}(\mu)$ admits a unique minimizer $\mu_{\beta}$ and that
$E(\mu_{\beta})$ is continuous. Then $F$ is differentiable, $F'(\beta)=E(\mu_{\beta})$
and, as a consequence, $F'$ is continuous and $S(e)$ is finite on
$\partial F(]\beta_{0},\beta_{1}[).$
\end{lem}

\begin{proof}
The differentiability of $F$ and the formula $F'(\beta)=E(\mu_{\beta})$
follows from the previous lemma, using the assumption that $E(\mu_{\beta})$
is continuous. Since $F'(\beta)=E(\mu_{\beta})$ this shows that,
in fact, $F'$ is continuous. Finally, take $e=F'(\beta)=E(\mu_{\beta}).$
Since $F(\beta)$ is finite we have that $S(\mu_{\beta})>-\infty$
and hence $S(e)$ is finite. 
\end{proof}
\begin{lem}
\label{lem:S strictly concave on image of dF}Assume that the assumptions
in the previous lemma are satisfied and that $S$ is usc in the interior
of $\partial F(]\beta_{0},\beta_{1}[).$ Then, on the interior of
$\partial F(]\beta_{0},\beta_{1}[),$ $S=F^{*}$ and $S$ is strictly
concave and 
\begin{equation}
S(F'(\beta))=F^{*}(F'(\beta))=\beta F'(\beta)-F(\beta)=S(\mu_{\beta}).\label{eq:formula in Prop S strictly conc on image of dF}
\end{equation}
In particular, $\mu_{\beta}$ is the unique maximum entropy measure
of energy $F'(\beta).$
\end{lem}

\begin{proof}
By the previous lemma $F(\beta)$ is differentiable for any $\beta$
in $]\beta_{0},\beta_{1}[.$ It follows from Lemma \ref{lem:diff implies dual strc conc}
that $F^{*}$ is strictly concave in the interior of $\partial F(]\beta_{0},\beta_{1}[).$
Next, let us show that $S$ is concave in the interior of $\partial F(]\beta_{0},\beta_{1}[)$
and that $S=F^{*}$ there. Assume, in order to get a contradiction,
that $S$ is \emph{not} concave. This means, by Lemma \ref{lem:convex envol affine},
that there exists a non-empty open interval $I$ in $\partial F(]\beta_{0},\beta_{1}[)$
such that
\[
S^{**}>S\,\,\text{on\,}I\subset\partial F(]\beta_{0},\beta_{1}[)
\]
It thus follows from Lemma \ref{lem:convex envol affine} that $S^{**}$
is affine on some open interval in $\partial F(]\beta_{0},\beta_{1}[).$
Now, in general $F(\beta)=S^{*}(\beta)$ (by \ref{eq:F is S star})
and hence $F^{*}(e)=S^{**}(e).$ But this means that $F^{*}$ is affine
on some open interval in $\partial F(]\beta_{0},\beta_{1}[),$ contradicting
that $F^{*}$ is strictly concave there. Hence, $S$ is concave and
$S=S^{**}$ in the interior of $\partial F(]\beta_{0},\beta_{1}[).$
This means that $S=F^{*}.$ Since, as explained above, $F^{*}$ is
strictly concave in $\partial F(]\beta_{0},\beta_{1}[)$ this shows
that $S$ is, in fact, strictly concave in $\partial F(]\beta_{0},\beta_{1}[).$
The second equality in formula \ref{eq:formula in Prop S strictly conc on image of dF}
follows from basic concave analysis. Finally, to prove the last equality
in formula \ref{eq:formula in Prop S strictly conc on image of dF},
recall that by Lemma \ref{lemma:macrostate equivalence} any optimizer
for $S(\mu)$ on $\{E(\mu)=e\}$ for $e$ an interior element in $\partial F(]\beta_{0},\beta_{1}[)$
minimizes $F_{\beta}(\mu)$ for $\beta$ such that $F'(\beta)=e.$
Hence, by assumption, the optimizer in question is $\mu_{\beta}.$ 
\end{proof}
\begin{lem}
\label{lem:conv in prob av E N}Assume that, for any $\beta$ in $]\beta_{0},\beta_{1}[,$$F(\beta)$
is differentiable and
\[
-\lim_{N\rightarrow\infty}N^{-1}\log Z_{N,\beta}=F(\beta).
\]
Then the function $E^{(N)},$ viewed as a random variable on the canonical
ensemble $(X^{N},\nu_{\beta}^{(N)}),$ converges in probability, as
$N\rightarrow\infty,$ to $F'(\beta).$ More precisely, the convergence
is exponential at speed $N.$
\end{lem}

\begin{proof}
This follows directly from Lemma \ref{lem:Cramer variant i R}. Indeed,
\[
\E(e^{tNE^{(N)}})=Z_{N,(\beta-t)}/Z_{N,\beta}
\]
Hence, $f(t)=-F(\beta+t)+F(\beta)$ which is differentiable wrt $t$
at $t=0,$ by assumption and $f'(0)=F'(\beta).$ 
\end{proof}
We are now ready for the proof of following result, which is inspired
by the case of point vortices in a compact domain of $\R^{2}$ in
\cite[Thm 4.2]{clmp2}.
\begin{thm}
\label{thm:general conv of mean entropy }Assume that $F(\beta)\in C^{1}\left(]\beta_{0},\beta_{1}[\right)$
and that $S(e)$ is usc on $\partial F\left(]\beta_{0},\beta_{1}[\right).$
If for any $\beta$ in $]\beta_{0},\beta_{1}[$ 
\begin{equation}
-\lim_{N\rightarrow\infty}N^{-1}\log Z_{N,\beta}=F(\beta),\label{eq:assumption conv towards F beta}
\end{equation}
 then 
\begin{itemize}
\item for any $e$ in the interior of $\partial F\left(]\beta_{0},\beta_{1}[\cap]0,\infty[\right)$
\[
\lim_{N\rightarrow\infty}N^{-1}\log\int_{\left\{ e-\epsilon<E^{(N)}<e\right\} }dV^{\otimes N}=S(e)
\]
 and when $\beta_{0}=0$ the previous convergence also holds for $E=e_{0}(=F'(0)).$ 
\item For any $e$ in the interior of $\partial F\left(]\beta_{0},\beta_{1}[\cap]-\infty,0[\right)$
\[
\lim_{N\rightarrow\infty}N^{-1}\log\int_{\left\{ e<E^{(N)}<e+\epsilon\right\} }dV^{\otimes N}=S(e)
\]
 and when $\beta_{1}=0$ the previous convergence also holds for $E=e_{0}(=F'(0)),$
if $E(\mu_{\beta})$ is assumed to be continuous as $\beta$ increases
to $0.$ 
\end{itemize}
\end{thm}

\begin{proof}
First consider the case when $e\in\partial F\left(]\beta_{0},\beta_{1}[\cap]0,\infty[\right)).$
For any fixed $\beta\geq0,$ inserting $1=e^{NE^{(N)}}e^{-NE^{(N)}}$
into the integrand and estimating $e^{NE^{(N)}}\leq e^{Ne}$ gives
\[
N^{-1}\log\int_{\left\{ e-\epsilon<E^{(N)}<e\right\} }dV^{\otimes N}\leq\beta e+N^{-1}\log Z_{N,\beta},
\]
In particular, taking $\beta\in]\beta_{0},\beta_{1}[$ such that $F'(\beta)=e$
and letting first $N\rightarrow\infty$ and, by assumption, the rhs
above converges towards $\beta e-F(\beta)$ which equals $S(e)$ (by
Lemma \ref{lem:S strictly concave on image of dF}). To prove the
lower bound note that inserting $1=(Z_{N,\beta}e^{N\beta E^{(N)}})e^{-N\beta E^{(N)}}/Z_{N,\beta}$
into the integrand and using Jensen's inequality gives, for any given
$\beta>0,$ 
\[
N^{-1}\log\int_{\left\{ e-\epsilon<E^{(N)}<e\right\} }dV^{\otimes N}\geq\beta\int1_{\left\{ e-\epsilon<E^{(N)}<e\right\} }E^{(N)}\mu_{\beta}^{(N)}+N^{-1}\log Z_{N,\beta},
\]
 where $\mu_{\beta}^{(N)}$ denotes the Gibbs measure \ref{eq:Gibbs meas}.
Since $e$ is in the interior of $\partial F\left(]\beta_{0},\beta_{1}[\cap]0,\infty[\right),$
there exists a sequence $e_{j}:=f'(\beta_{j})$ such that $e-\epsilon<e_{j}<e$
and $e_{j}$ increases towards $e$ (using Lemma \ref{lem:F cont diff}).
We make the following
\begin{equation}
\text{Claim:\,}\,\,\lim_{N\rightarrow\infty}\int1_{\left\{ e-\epsilon<E^{(N)}<e\right\} }E^{(N)}\mu_{\beta_{j}}^{(N)}=F'(\beta_{j}).\label{eq:claim}
\end{equation}
 Accepting this for the moment and using the assumed convergence \ref{eq:assumption conv towards F beta}
thus gives 
\[
\liminf_{N\rightarrow\infty}\log\int_{\left\{ e-\epsilon<E^{(N)}<e\right\} }dV^{\otimes N}\geq\beta_{j}F'(\beta_{j})-F(\beta_{j})=S(e_{j}).
\]
Finally, letting $j\rightarrow\infty$ and using that $S(e)$ is continuous
on $\partial F(]\beta_{0},\beta_{1}[$ (since its finite and concave
there) concludes the proof of the case when $e\in\partial F\left(]\beta_{0},\beta_{1}[\cap]0,\infty[\right))$
once we have proved the claim \ref{eq:claim}. To this end it will
be enough to prove that, as $N\rightarrow\infty,$ $E^{(N)}\rightarrow F'(\beta_{j})$
in probability (wrt $\mu_{\beta}^{(N)}).$ But this follows directly
from the previous lemma and the assumed convergence \ref{eq:assumption conv towards F beta}

Next consider the case when $\beta_{0}=0$ and $e=e_{0}$ (in the
first point). Since $E$ is lsc and $dV/(\int_{X}dV)$ is the unique
maximizer of $F_{0}(\mu)$ (by Jensen) it follows that $F'(\beta)$
increases to $e_{0}$ as $\beta$ decreases to $0.$ We can then repeat
the previous argument with $E=F'(0)(=e_{0}).$ 

Finally, consider the case when $\in\partial F\left(]\beta_{0},\beta_{1}[\cap]-\infty,0[\right)$
for $\beta<0$ (in the second point). The convergence towards $S(e)$
then follows from applying the previous argument with $E^{(N)}$ replaced
by $-E^{(N)}.$ But, when $e=e_{0}$ we also need to assume that $F'(\beta)$
converges to $e_{0}$ when $\beta$ increases to $0$ (since $-E$
is not assumed lsc).
\end{proof}
\begin{rem}
When $F'(\beta)$ is assumed to be strictly decreasing on $]\beta_{0},\beta_{1}[,$
the set $\partial F\left(]\beta_{0},\beta_{1}[\right)$ is automatically
open (and non-empty). This assumption will be shown to be satisfied
in the complex-geometric setup, appearing in the introduction of the
paper.
\end{rem}

\subsection{General properties of the entropy and free energy}

We first recall the following three results from \cite{ber9} involving
the following definitions:
\begin{defn}
\label{def:def of affine cont and energy appr}Given a functional
$E(\mu)$ on $\partial F\left(]\beta_{0},\beta_{1}[\cap]0,\infty[\right)$$\mathcal{P}(X)$
\begin{itemize}
\item $E(\mu)$ is said to have the \emph{affine continuity property }if
for any $\mu_{1}\in\mathcal{P}(X)$ such that $E(\mu_{1})<\infty$
and $S(\mu_{1})>-\infty$ the function $t\mapsto E(\mu_{0}(1-t)+t\mu_{1})$
is continuous on $[0,1].$ 
\item A measure $dV$ on $X$ is said to have the \emph{Energy Approximation
Property} if for any $\mu\in\mathcal{P}(X)$ there exists a sequence
$\mu_{j}\in\mathcal{P}(X)$ converging weakly towards $\mu$ such
that $\mu_{j}$ is absolutely continuous with respect to $dV$ and
$E(\mu_{j})\rightarrow E(\mu)$
\end{itemize}
\end{defn}

The following result \cite[Lemma 3.2]{ber9} shows that $S(e)$ is
unimodal with a maximum at $e=e_{0}$ under suitable assumptions.
\begin{lem}
\label{lem:mono of S} (monotonicity of $S(e))$ 
\begin{itemize}
\item If $E(\mu)$ is convex on $\mathcal{P}(X),$ then $S(e)$ is increasing
for $e\leq e_{0}$ and strictly increasing in the interval where $S(e)>-\infty.$
In particular, 
\[
S(e)=\sup_{E(\mu)\leq e}S(\mu)
\]
\item If $E(\mu)$ has the affine continuity property, then $S(e)$ is decreasing
for $e\geq e_{0}$ and strictly decreasing in the interval where $S(e)>-\infty.$
In particular, 
\[
S(e)=\sup_{E(\mu)\geq e}S(\mu)
\]
More precisely, in the second point there is no need to assume that
$E$ is lsc on $\mathcal{P}(X)$ and thus it also follows that $S(e)$
is increasing for $e\leq E(\mu_{0}).$
\end{itemize}
\end{lem}

The following result \cite[Prop 4.3]{ber9} ensures that $S(e)$ is
finite:
\begin{lem}
\label{lem:energy appr implies S finite etc} Assume that $\mu_{0}$
has the energy approximation property. For any $e\in]e_{min},e_{0}]$
$S(e)$ is finite and there exists a maximum entropy measure $\mu^{e}.$
If moreover $E(\mu)$ has the affine continuity property, then $S(e)$
is finite on $]e_{min},e_{max}[.$ 
\end{lem}

The following result \cite[Prop 5.1]{ber9} shows that $S(e)$ is
concave in the low energy region if $E(\mu)$ is convex: 
\begin{lem}
\label{lem:E convex plus energy approx implies S concave} If $E(\mu)$
is convex, then $S(e)$ is concave on $]-\infty,e_{0}].$ 
\end{lem}

We next establish two more rather general results:
\begin{prop}
\label{prop:general results for S and F when E convex}Assume that
$E(\mu)$ is convex. Then, for any $\beta>0,$ $F_{\beta}$ admits
a unique minimizer $\mu_{\beta}$ and $E(\mu_{\beta})$ is continuous
wrt $\beta\in]0,\infty[.$ As a consequence, $F\in C^{1}\left(]0,\infty[\right)$
(by Lemma \ref{lem:F cont diff}). If moreover $\mu_{0}$ has the
energy approximation property, then, for any $e\in]e_{min},e_{0}[,$
$S(e)=F^{*}(e),$ $S(e)$ is strictly concave and continuous and the
sup defining $S(e)$ is attained at a unique measure $\mu^{e}$ (i.e.
there exists a unique maximum entropy measure with energy $e).$
\end{prop}

\begin{proof}
Fix $\beta\geq0.$ By replacing $E$ with $E+C$ for a large constant
$C$ we may as well assume that $E\geq0.$ Since $E$ and $-S$ are
both lsc so is $F_{\beta}.$ As a consequence, $F_{\beta}$ admits
a minimizer $\mu_{\beta}$ on $\mathcal{P}(X)$ (since $X$ and thus
also $\mathcal{P}(X)$ is compact and $\{F_{\beta}<\infty\}$ is non-empty,
using $E(\mu_{0})<\infty$ ). In particular, $F(\beta)$ is finite.
Next, since $E(\mu)$ is assumed convex and $-S(\mu)$ is always strictly
convex, on the convex set where it is finite (by Hölder's inequality),
it follows that $F_{\beta}$ is strictly convex (where it is finite)
and thus $\mu_{\beta}$ is unique. To prove that $E(\mu_{\beta})$
is continuous in $\beta$ take $\beta_{j}\rightarrow\beta$ in $[0,\infty[.$
Since $E$ is lsc it will be enough to show that $\limsup_{j}E(\mu_{\beta_{j}})\leq E(\mu_{\beta}).$
But, since $F(\beta)$ is concave and finite on $]0,\infty[$ it it
is also continuous. Hence, $\beta_{j}E(\mu_{\beta_{j}})-S(\mu_{\beta_{j}})\rightarrow\beta E(\mu_{\beta})-S(\mu_{\beta}).$
Since $-S(\mu_{\beta})$ is lsc it follows that $\limsup_{j}E(\mu_{\beta_{j}})\leq E(\mu_{\beta}),$
as desired. Finally, by the previous lemma we have that $S^{**}(e)=S(e)$
and, by Lemma \ref{lem:energy appr implies S finite etc}, $S(e)$
is finite, Hence, by Lemma \ref{lem:diff implies dual strc conc}
$S(e)$ is strictly convex and continuous on $]e_{min},e_{0}[$ and
for any $e\in]e_{min},e_{0}[$ there exists, by Lemma \ref{lem:energy appr implies S finite etc},
a maximum entropy measure $\mu^{e}$ which is unique, by Lemma \ref{lemma:macrostate equivalence}.
\end{proof}
\begin{prop}
\label{Prop:properties when E convex and energy app}Assume that $E$
is convex, $\mu_{0}$ has the energy approximation property and that
$\mu_{0}$ does not minimize $E,$ i.e. $e_{\min}<e_{0}.$ Then all
the assumptions on $E$ and $F_{\beta}$ is Theorem \ref{thm:general conv of mean entropy }
are satisfied and $\partial F(]\beta_{0},\beta_{1}[=]e_{\min},e_{0}[.$
\end{prop}

\begin{proof}
All that remains is to note that, as pointed out in the proof of Prop
5.3 in \cite{ber9}, $\partial F(]\beta_{0},\beta_{1}[=]e_{\min},e_{0}[.$
In other words $F'(\beta)\rightarrow e_{0}$ as $\beta\rightarrow0$
and $F'(\beta)\rightarrow e_{\text{min}}$ as $\beta\rightarrow\infty$
(using in the first case that $E$ is lsc and in the latter case that
$\mu_{0}$ is assumed to have the energy approximation property).
\end{proof}

\section{Microcanonical convergence results for quasi-superharmonic $E^{(N)}$ }

We start with the following general result, where $\delta_{N}$ is
viewed as a random variable wrt the probability measure \ref{eq:def of lower micro intro}
on $X^{N}:$
\begin{thm}
Let $X$ be compact and assume that $E$ is lsc convex on $\mathcal{P}(X)$
and that $dV$ has the energy approximation property. Given $u\in C^{0}(X),$
set $dV_{u}:=e^{-u}dV$ and assume that 

\begin{equation}
-\lim_{N\rightarrow\infty}N^{-1}\log\int_{X^{N}}e^{-N\beta E^{(N)}}dV_{u}^{\otimes N}=\inf_{\mathcal{P}(X)}\left(F_{\beta}(\mu)+\left\langle u,\mu\right\rangle \right)\label{eq:conv of Z N with u}
\end{equation}
 for any $u\in C^{0}(X)$ and $\beta>0.$ Then, given $e\in]e_{min},e_{0}[$
and $\epsilon>0$ the corresponding empirical measure $\delta_{N}$
converges \emph{exponentially }in probability at speed $N$ towards
the unique maximum entropy measure $\mu^{e}$ of energy $e.$
\end{thm}

\begin{proof}
We will apply Lemma \ref{lem:conv in prob av E N} with $\Gamma_{N}$
the law of $\delta_{N}.$ We can express 
\[
\E(e^{-N\left\langle u,\delta_{N}\right\rangle })=\int_{\left\{ E^{(N)}\in]e-\epsilon,e[\right\} }dV_{u}^{\otimes N}/\int_{\left\{ E^{(N)}\in]e-\epsilon,e[\right\} }dV^{\otimes N}.
\]
 Next assume that $u$ satisfies the following inequality: 
\begin{equation}
e<e_{0}(u):=E(dV_{u}/\int_{X}dV_{u})\label{eq:ineq e not u}
\end{equation}
Then the assumptions in Theorem \ref{thm:general conv of mean entropy }
are satisfied for $(E^{(N)},E,dV_{u}),$ using Prop \ref{Prop:properties when E convex and energy app}.
Hence, denoting by $S_{u}$ the entropy wrt the reference measure
$dV_{u},$
\begin{equation}
\lim_{N\rightarrow\infty}N^{-1}\log\E(e^{-N\left\langle u,\delta_{N}\right\rangle })=S_{u}(e)-S(e).\label{eq:conv of Exp in pf thm quasi}
\end{equation}
Denoting by $\mathcal{P}(X)^{\leq e}$ the subspace of $\mathcal{P}(X)$
consisting of $\mu$ such that $E(\mu)\leq e,$ we can express
\begin{equation}
S_{u}(e)=\sup_{\mu\in\mathcal{P}(X)^{\leq e},}\left\{ S(\mu)-\left\langle u,\mu\right\rangle \right\} =S_{u}(\mu^{e,u})\label{eq:form for S u in pf Thm quasi}
\end{equation}
 where $\mu^{e,u}$ denotes the unique element where the corresponding
sup is attained (by Prop \ref{prop:general results for S and F when E convex}).
Now fix an arbitrary $u\in C^{0}(X).$ Then $tu$ satisfies the inequality
\ref{eq:ineq e not u} for $t$ sufficiently close to $0$ (since,
by assumption, the inequality holds when $t=0$ and $E$ is lsc).
Hence, by Lemma \ref{lem:Cramer variant} it will be enough to show
that $t\mapsto-S_{tu}(e)$ is Gateaux differentiable and that 
\[
-dS_{tu}(e)/dt_{|t=0}=\left\langle u,\mu^{e}\right\rangle .
\]
 To prove this first note that, using the second equality in formula
\ref{eq:form for S u in pf Thm quasi} and Lemma \ref{lem:differentiability of sup}
it is, since the sup in \ref{eq:form for S u in pf Thm quasi} is
attained uniquely at $\mu^{e,u},$ enough to show that $t\mapsto\mu^{e,tu}$
is continuous. To this end, consider any limit point $\mu_{*}$ of
$\mu^{e,tu}$ in $\mathcal{P}(X)$ as $t\rightarrow0.$ Since $E$
is lsc we have $E(\mu_{*})\leq e.$ Hence, all that remains is to
show that $S(\mu_{*})\geq S(\mu^{e}).$ First, since $S$ is usc we
have $S(\mu_{*})\geq\limsup_{t\rightarrow0}S(\mu^{e,tu}).$ Secondly,
since $u$ is bounded and $\mu^{e,tu}\in\mathcal{P}(X),$ we may as
well replace $S(\mu^{e,tu})$ with $S_{tu}(\mu^{e,tu}),$ i.e. with
$S_{tu}(e),$ when $t\rightarrow0.$ But, the function $t\mapsto S_{tu}(e)$
is convex (by formula \ref{eq:form for S u in pf Thm quasi}) and
finite. As a consequence, it is continuous, giving $\lim_{t\rightarrow0}S_{tu}(\mu^{e,tu})=S_{0}(e)=S(\mu^{e}),$
which concludes the proof. 
\end{proof}

\subsection{The quasi-superharmonic case}

Let $(X,g)$ be a compact Riemannian manifold and $E^{(N)}$ a sequence
of symmetric lsc functions on $X^{N}.$ We will say that $E^{(N)}$
is \emph{uniformly quasi-superharmonic} if there exists a positive
constant $C$ such that, for all $N,$
\[
\Delta_{g}E^{(N)}\leq C\,\,\,\text{on \ensuremath{X^{N}}},
\]
 where $\Delta_{g}$ denotes the Laplacian on $X^{N}$ endowed with
the product Riemannian metric. We next recall the notion of Gamma-convergence
\cite{bra}:
\begin{defn}
\label{def:Gamma}A sequence of functions $E^{(N)}$ on a topological
space \emph{$\mathcal{M}$ is said to Gamma$-$converge }to a function
$E$ on $\mathcal{M}$ if 
\begin{equation}
\begin{array}{ccc}
\mu_{N}\rightarrow\mu\,\mbox{in\,}\mathcal{M} & \implies & \liminf_{N\rightarrow\infty}E^{(N)}(\mu_{N})\geq E(\mu)\\
\forall\mu & \exists\mu_{N}\rightarrow\mu\,\mbox{in\,}\mathcal{M}: & \lim_{N\rightarrow\infty}E^{(N)}(\mu_{N})=E(\mu)
\end{array}\label{eq:def of gamma conv}
\end{equation}
The limiting functional $E$ is automatically lower semi-continuous
on $\mathcal{M}.$ In the present setup we start with a sequence of
symmetric functions $E^{(N)}$ on $X^{N},$ i.e. $E^{(N)}$ is invariant
under the permutation group $S_{N}$ of $\{1,...,N\}.$ Under the
embedding $\delta_{N}:X^{N}/S_{N}\hookrightarrow\mathcal{P}(X),$
defined by the empirical measure $\delta_{N}$ (formula \ref{eq:emp measure intro})
we then identify $E^{(N)}$ with a sequence of functions on $\mathcal{P}(X)$
taking the value $\infty$ on $\mathcal{P}(X)-\delta_{N}(X^{N}/S_{N}).$
\end{defn}

Given a measure $dV,$ $e\in\R$ and $\epsilon>0$ we will view the
empirical measure $\delta_{N}$ on $X$ as a random measure, i.e.
a random variable on the probability space $(X^{N},\mu_{]e-\epsilon,e[}^{(N)}).$
\begin{thm}
\label{thm:LDP for quasi super}Let $(X,g)$ be a compact Riemannian
manifold and assume that $E^{(N)}$ is lsc and uniformly quasi-super
harmonic on $X^{N}$ and Gamma-converges towards a functional $E(\mu)$
on $\mathcal{P}(X)$ which is convex and that $dV_{g}$ has the energy
approximation property. Then for any given continuous volume form
$dV$ and $e\in]e_{\text{min}},e_{0}]$
\begin{equation}
\lim_{N\rightarrow\infty}N^{-1}\log\int_{\left\{ E^{(N)}\in]e-\epsilon,e[\right\} }dV^{\otimes N}=S(e)\label{eq:conv towards S(e) in quasi}
\end{equation}
 and $\delta_{N}$ converges in probability towards the unique maximum
entropy measure $\mu^{e}$ of energy $e.$ Additionally, if $e\in]0,e_{0}[,$
then the convergence towards $\mu^{e}$ is exponential at speed $N.$ 
\end{thm}

\begin{proof}
First note that $(E^{(N)},E,dV)$ satisfies the assumptions in Theorem
\ref{thm:general conv of mean entropy }, as follows from Prop \ref{Prop:properties when E convex and energy app}
(using the convexity of $E)$ and the LDP for the Gibbs measures established
in \cite[Thm 4.1]{berm8} (using the quasi-superharmonicity). Hence,
the convergence \ref{eq:conv towards S(e) in quasi} towards $S(e)$
follows from \ref{thm:general conv of mean entropy }. In fact, the
LDP in \cite[Thm 4.1]{berm8} implies that the convergence \ref{eq:conv of Z N with u}
holds for any $u\in C^{0}(X).$ Hence, the exponential convergence
of $\delta_{N}$ towards $\mu_{e}$ follows from the previous theorem,
when $e\in]e_{\text{min}},e_{0}[.$ Next we give a different proof
of the weaker statement that $\delta_{N}$ converges in probability
towards $\mu^{e}$ that also applies when $e=e_{0}.$ We use the approach
in \cite{berm11b} (which is a reformulation of the method introduced
in \cite{m-s}) Given $\mu_{N}\in\mathcal{P}(X^{N})$ set $D^{(N)}(\mu_{N})=-N^{-1}S(\mu_{N}),$
where $S(\mu_{N})$ denotes the entropy of $\mu_{N}$ relative to
$dV^{\otimes N}.$ First, since $\beta:=S'(e)>0$ we get $\beta\int_{X^{N}}E^{(N)}\mu_{]e-\epsilon,e[}^{(N)}\leq\beta e,$
so that
\[
\limsup_{N\rightarrow\infty}\left(\beta\int_{X^{N}}E^{(N)}\mu_{]e-\epsilon,e[}^{(N)}+D^{(N)}(\mu_{]e-\epsilon,e[}^{(N)})\right)\leq\beta e-S(e)=F(\beta)
\]
 using the convergence of $D^{(N)}(\mu_{]e-\epsilon,e[}^{(N)})$ towards
$-S(e)$ in formula \ref{eq:conv towards S(e) in quasi}, and that
$F(\beta)=S^{*}(\beta).$ Next, by the well-known sub-additivity of
$D^{(N)}$ 
\begin{equation}
\int_{\mathcal{P}(X)}\Gamma_{\infty}^{e}D^{(1)}(\mu)\leq\liminf_{N\rightarrow\infty}D^{(N)}(\mu_{]e-\epsilon,e[}^{(N)})\label{eq:sub-add D-1}
\end{equation}
 for any accumulation point $\Gamma_{\infty}^{e}$ of $(\delta_{N})_{*}\mu_{]e-\epsilon,e[}^{(N)}$
in $\mathcal{P}\left(\mathcal{P}(X)\right)$ (see \cite{ro-r} and
\cite[Thm 5.5]{h-m} for generalizations). Moreover, 
\[
\int_{\mathcal{P}(X)}\Gamma_{\infty}^{e}E(\mu)\leq\liminf_{N\rightarrow\infty}\int_{X^{N}}E^{(N)}\mu_{]e-\epsilon,e[}^{(N)}.
\]
 Indeed, the follows directly from the definition of Gamma-convergence
and Fatou's lemma with varying measures \cite[Lemma 3.2]{serfo}.
All in all this gives
\[
\int_{\mathcal{P}(X)}\Gamma_{\infty}^{e}F_{\beta}(\mu)\leq F_{\beta}(\mu^{e}).
\]
 Since $\mu^{e}$ is the unique minimizer of $F_{\beta}$ (by Theorem
\ref{thm:lower energy entro text} and Lemma \ref{lemma:macrostate equivalence})
this forces $\Gamma_{\infty}^{e}=\delta_{\mu^{e}},$ as desired (also
using that $F_{\beta}$ is lsc on $\mathcal{P}(X)$).
\end{proof}

\section{Entropy and free energy on complex manifolds }

In this section we employ the setup introduced in Section \ref{subsec:Setup intro},
where $X$ is an $n-$dimension compact complex manifold. We will
denote by $J$ the complex structure on $X,$ which defines an endomorphism
on the real tangent bundle $TX$ of $X$ such that $J^{2}=-I.$ We
decompose, as usual, $TX\otimes\C=TX^{1,0}\oplus TX^{0,1}$ of $TX$
into the eigenspaces of $J$ with eigenvalues $i$ and $-i$ respectively
(and likewise for the dual of $TX)$. Accordingly, the exterior derivative
$d$ is decomposed as $d=\partial+\bar{\partial}.$ We set $d^{c}:=-J^{*}d/4\pi.$
We will use additive notation for tensor products of line bundles.
In particular, the dual of the canonical line bundle $K_{X}$ of $X$
is denoted by $-K_{X},$ i.e. $-K_{X}:=\Lambda^{n}T^{1,0}X.$

\subsection{\label{subsec:Preliminaries polarized}Preliminaries}

\subsubsection{\label{subsec:kahle}Kähler metrics, Ricci curvature and volume forms
with divisorial singularities}

Recall that a $J-$invariant closed real form $\omega$ on a complex
manifold $X$ is said to be \emph{Kähler }if $\omega>0$ in the sense
that the corresponding symmetric two-tensor $\omega(\cdot,J\cdot)$
is positive definite, i.e. defines a Riemannian metric. In practice,
one then identifies the Kähler form $\omega$ with the corresponding
Kähler metric. Likewise, the Ricci curvature of a Kähler metric $\omega$
may be identified with the two-form
\begin{equation}
\text{Ric \ensuremath{\omega}}=-\frac{i}{2\pi}\partial\bar{\partial}\log\frac{\omega^{n}}{i^{n^{2}}dz\wedge d\bar{z}},\,\,\,dz:=dz_{1}\wedge...\wedge dz_{n}\label{eq:formula for Ricci}
\end{equation}
 wrt any choice of local holomorphic coordinates $z_{i}.$ A measure
$dV$ on $X$ (assumed finite, $\int_{X}dV<\infty)$ is is said to
have \emph{divisorial singularities} if, in terms of local holomorphic
coordinates $z_{i}$, 
\begin{equation}
dV=e^{-v}\prod_{i}|f_{i}|^{-2c_{i}}i^{n^{2}}dz\wedge d\bar{z},\label{eq:def of dV with div sing}
\end{equation}
 for irreducible holomorphic functions $f_{i},$ constants $c_{i}\in\R$
and a bounded function $v$ that is smooth on the complement of the
union of $f_{i}^{-1}(0).$ We denote by $\Delta$ the corresponding
divisor on $X,$ i.e. the following formal sum of analytic subvarities:
\begin{equation}
\Delta:=\sum_{i}c_{i}f_{i}^{-1}(0).\label{eq:def of Delta}
\end{equation}

\subsubsection{Line bundles and notions of positivity}

Let $L\rightarrow X$ be a holomorphic line bundle and $\left\Vert \cdot\right\Vert $
a (Hermitian) metric on $L.$ The (normalized) curvature of $\left\Vert \cdot\right\Vert $
is the two-form $\theta$ on $X,$ locally defined on an open subset
$U\subset X$ by $-\pi^{-1}idd^{c}\log\left\Vert s_{U}\right\Vert ,$
where $s_{U}$ is any local non-vanishing holomorphic section of $L.$
It represents the first Chern class $c_{1}(L)$ of $L$ in the de
Rham cohomology group $H^{2}(X,\R)$ (the normalization ensures that
the curvature has integral periods). More generally, if the metric
is singular, i.e. locally $\log\left\Vert s_{U}\right\Vert \in L_{\text{loc}}^{1},$
then the curvature defines a current on $X.$ Throughout this work
the line bundle $L$ is assumed \emph{big,} which in algebro-geometric
terms means that the asymptotics \ref{eq:def of N k} hold with $\text{vol \ensuremath{(L)>0.}}$
In analytic terms, this equivalently means that $L$ admits some singular
metric with strictly positive curvature (in the sense that the curvature
current is locally bounded from below by a Kähler form). A line bundle
$L$ is \emph{ample} iff it admits a smooth metric with strictly positive
curvature (i.e. the curvature form defines a Kähler form). Finally,
a line bundle $L$ is \emph{nef} it admits some singular metric $\left\Vert \cdot\right\Vert $
with positive curvature current such that $\log\left\Vert s_{U}\right\Vert $
has no Lelong numbers (which equivalently means that $\left\Vert s_{U}\right\Vert ^{-p}$
is locally integrable for all $p\in]0,\infty[).$

\subsubsection{Ricci curvature of $dV$}

Given a smooth volume form $dV$ on $X$ we define $\text{Ric}dV$
as the curvature of the metric on $-K_{X}$ induced by $dV$$.$ Concretely,
$\text{Ric}dV$ is thus defined by replacing $\omega^{n}$ with $dV$
in formula \ref{eq:formula for Ricci}. More generally, if $dV$ is
a measure on $X,$ assumed absolutely continuous wrt Lebesgue measure
and such that log of the local densities of $dV$ are integrable,
then $\text{Ric}dV$ is the current defined as the curvature current
of the singular metric on $-K_{X}$ induced by $dV.$ In particular,
if $dV$ has divisorial singularities (formula \ref{eq:def of dV with div sing}),
then 
\[
\text{Ric}dV=[\Delta]+dd^{c}v,
\]
 where $\Delta$ denotes the current of integration along the divisor
$\Delta$ on $X.$

\subsubsection{\label{subsec:Pluripotential-theory}Pluripotential theory}

Recall that $\theta$ denotes the curvture form of a fixed smooth
metric on $L.$ Given $u\in L^{1}(X),$ set 
\[
\theta_{u}:=\theta+dd^{c}u\,\,\,\,\left(dd^{c}u:=\frac{i}{2\pi}\partial\bar{\partial}u\right)
\]
 which defines a current on $X$ of degree $2$ (or more precisely
of bidegree $(1,1)).$ We will denote by $\text{PSH }(X,\theta)$
the subspace of $L^{1}(X)$ consisting of all strongly upper semi-continuous
functions $u$ on $X$ taking values in $[-\infty,\infty[$ such that
$\theta_{u}\geq0$ holds in the sense of currents. We will denote
by $\theta_{u}^{j}$ the $j-$fold non-pluripolar product \cite{bbegz}
(which coincides with the $n-$ fold exterior product when $\theta_{u}$
is smooth, or in $L_{\text{loc }}^{\infty}).$ In particular, $\theta_{u}^{n}$
is the \emph{non-pluripolar Monge-Ampère measure} of Given $\mu\in\mathcal{P}(X)$
the following complex Monge-Ampère equation - known as the\emph{ Calabi-Yau
equation} - 
\begin{equation}
\frac{1}{\text{vol\ensuremath{(L)}}}\theta_{u}^{n}=\mu,\,\,\,u\in\text{PSH }(X,\theta)\label{eq:MA eq with mu}
\end{equation}
 admits a solution iff $\mu$ does not charge pluripolar subsets of
$X$ and the solution $u_{\mu}$ is unique mod $\R$ \cite{bbegz,bbgz}.
Consider the envelope
\begin{equation}
u_{\theta}(x):=\sup\left\{ u(x):\,u\in\text{PSH \ensuremath{(X,\theta)},}\,\,u\leq0\right\} ,\label{eq:def of env}
\end{equation}
 which defines a function in $\text{PSH }(X,\theta)$ and set
\[
\mathcal{E}^{1}(X,\theta):=\left\{ u\in\text{PSH }(X,\theta):\,\int_{X}\theta_{u}^{n}=\text{vol}\ensuremath{(L)},\,\,\int_{X}(u_{\theta}-u)\theta_{u}^{n}<\infty\right\} .
\]
 Consider the following functional on $\mathcal{E}^{1}(X,\theta):$
\begin{equation}
\mathcal{E}(u):=\int_{X}\sum_{j=0}^{n}(u_{\theta}-u)\theta_{u}^{j}\wedge\theta_{u_{\theta}}^{n-j}.\label{eq:primitive appears}
\end{equation}
(this functional is denoted by $E$ in \cite{bbgz}, but here we shall,
following \cite{berm6}, reserve capital letters for functionals defined
on $\mathcal{P}(X)$). 

\subsubsection{\label{subsec:The pluricomplex energy}The pluricomplex energy $E(\mu)$}

For the present purposes it will be convenient to define the pluricomplex
energy $E(\mu)$ of a measure $\mu$ (wrt $\theta)$, introduced in
\cite{bbgz} (where it is called $E^{*})$, by the characterization
in \cite[Thm A]{bbgz}: $E(\mu)<\infty$ iff there exists $u_{\mu}\in\mathcal{E}^{1}(X,\theta)$
such that $\theta_{u_{\mu}}^{n}/\text{vol\ensuremath{(L)}}=\mu$ and
in that case
\[
E(\mu):=\frac{1}{\text{vol\ensuremath{(L)}}(n+1)}\mathcal{E}(u_{\mu})-\int_{X}u_{\mu}\mu
\]

\begin{thm}
\label{thm:prop of energies}The pluricomplex energy $E(\mu)$ is
\emph{convex }and\emph{ lsc} on $\mathcal{P}(X)$ \cite{bbgz}. Moreover,
if $L$ is big and nef, then \cite{bbegz}
\begin{equation}
E\,\text{is continuous on \ensuremath{\left\{  \ensuremath{\mu:\,\,S(\mu)\geq-C}\right\} } }\label{eq:energy entrop cont}
\end{equation}
 for any given constant $C.$ 
\end{thm}

In \cite[Thm 2.17]{bbegz} the continuity result \ref{eq:energy entrop cont}
was shown under the stronger assumption that $L$ be the pull-back
to $X$ of an ample line bundle under a resolution of singularities.
But this assumption was only used to make sure that elements in $\mathcal{E}^{1}(X,\theta)$
have no Lelong numbers and, as subsequently shown in \cite{d-n-l},
this holds more generally when $L$ is big and nef.

\subsubsection{\label{subsec:Twisted}Twisted Kähler-Einstein equations}

Given the data $(\theta,dV)$ and a $\beta\in\R$ we consider the
corresponding twisted Kähler-Einstein equation 
\begin{equation}
\mbox{\ensuremath{\mbox{Ric}}\ensuremath{\omega}}+\beta\omega=\beta\theta+\ensuremath{\mbox{Ric}}\ensuremath{dV}\label{eq:twisted KE text}
\end{equation}
 for a positive current $\omega$ in the first Chern class $c_{1}(L)$
of $L$ such that $\omega^{n}$ is locally absolutely continuous wrt
Lebesgue measure. Expressing $\omega=\theta+dd^{c}u,$ equation \ref{eq:twisted KE text}
is equivalent to the following complex Monge-Ampère equation for $u$
in $\text{PSH }(X,\theta):$ 
\begin{equation}
\frac{1}{\text{vol\ensuremath{(L)}}}\theta_{u}^{n}=e^{\beta u}dV\label{eq:ma eq with beta}
\end{equation}
Conjecturally, when $dV$ has divisorial singularities the following
regularity result holds for any big line bundle $L$: the solution
$\omega$ defines a Kähler metric on a Zariski open subset of $X$
(the intersection of the ample locus of $X$ with the complement of
the support of $\Delta$). This regularity result is known to hold
when $L$ is nef \cite[Thm B.1]{bbegz} and when the section ring
$R(L)$ of $L$ is finitely generated (which reduces to the case when
$L$ is nef and big, as in \cite[Section 6.2]{begz}, by replacing
$X$ with a non-singular resolution of $\text{Proj \ensuremath{R(L)}}$).
In particular, the regularity result holds for $L=K_{X}$ when $X$
is a variety of general type and when $X$ is a Fano variety and $L=-K_{X}$
(as in the setup of Section \ref{subsec:The-high-energy intro}). 
\begin{rem}
If $L$ is ample and the divisor $\Delta$ (formula \ref{eq:def of Delta})
is assumed to have simple normal crossings, then a stronger regularity
result hold: $\omega$ is a Kähler metric with edge-cone singularities
along $\Delta$ \cite{g-p,jmr,m-r}.
\end{rem}

\subsubsection{\label{subsec:K=0000E4hler-Einstein-metrics}Kähler-Einstein metrics}

Now assume that the data $(\theta,dV)$ satisfies the following compatibility
condition: 
\begin{equation}
\ensuremath{\text{Ric}}\ensuremath{dV-[\Delta]=-\beta\theta.}\label{eq:theta is Ricci-1}
\end{equation}
Then a solution $\omega$ to the corresponding equation \ref{eq:twisted KE intro}
for $\beta\in\{\pm1,0\}$ is a \emph{Kähler-Einstein metric} for the
pair $(X,\Delta),$ i.e. $\omega$ is a solution to the equation 
\begin{equation}
\mbox{\ensuremath{\mbox{Ric}}\ensuremath{\omega}}=-\beta\omega+[\Delta].\label{eq:log KE eq text}
\end{equation}
By \cite{bbegz}, the solution $\omega$ is a bona fide smooth Kähler-Einstein
metric on $X-\text{supp\ensuremath{(\Delta)}.}$ When $\beta=\pm1$
the compatibility condition \ref{eq:theta is Ricci-1} implies that
$L=\pm(K_{X}+\Delta),$ where the divisor $\Delta$ has been identified
with the corresponding line bundle in the standard way \cite{bbegz,berman6ii}.
Moreover, the compatibility condition amounts to taking the metric
$\left\Vert \cdot\right\Vert $ on $L$ to be the one induced by any
multiple of $dV$ on $X-\text{supp\ensuremath{(\Delta)}.}$
\begin{rem}
By \cite{g-p,jmr,m-r}, when $X$ is smooth and $\Delta$ has simple
simple normal crossings the solution $\omega$ of equation \ref{eq:log KE eq text}
is a Kähler-Einstein metric with edge-cone singularities along $\Delta.$
\end{rem}

\subsubsection{\label{subsec:Singular_complex_varieties}Singular complex varieties}

The previous setup essentially contains the case when $X$ is a (possibly
singular) compact complex variety if $(X,\Delta)$ is assumed to have
log terminal singularities \cite{ko0}. We will briefly recall how
this works in the simplest case when the divisor $\Delta$ (formula
\ref{eq:def of Delta}) is trivial. The assumption that $X$ has log
terminal singularities is equivalent to $X$ being a normal variety
such that $K_{X}$ is defined as a $\Q-$line bundle on $X$ and that
the volume form on the regular locus $X_{\text{reg }}$ of $X$ satisfies
$\int_{X_{\text{reg}}}dV<\infty$ \cite{bbegz}. It thus extends to
a finite measure on $X$ (not charging the singular locus of $X).$
Fixing a resolution of singularities, i.e. a birational holomorphic
surjection $g:\,Y\rightarrow X$ from a compact complex manifold $Y$
to $X,$ the pullback $g^{*}L$ is a big line bundle on $Y$ and $g^{*}dV$
is a volume form on $Y$ with divisorial singularities. Accordingly,
all objects on $X$ can be identified with their pull-back to $Y.$
When $L\rightarrow X$ is ample, solutions to the twisted Kähler-Einstein
equation \ref{eq:twisted KE text} correspond to bona fide Kähler
metrics on $X_{\text{reg}}$ solving \ref{eq:twisted KE text} on
$X_{\text{reg}}$ and whose volume on $X_{\text{reg }}$ coincide
with $\text{vol \ensuremath{(L)}. }$ We will be particularly concerned
with the case when $X$ is a \emph{Fano variety. }This means that
$-K_{X}$ is ample.

\subsubsection{\label{subsec:Relations-to-the standard}Relations to the standard
functionals in Kähler geometry}

Consider the case when $X$ is non-singular, $L$ is ample and $\mu\in\mathcal{P}(X).$
Writing $\mu=\theta_{u}^{n}$/$\text{Vol \ensuremath{(L)}}$ for $u\in\text{Psh \ensuremath{(X,\theta)}},$
we can express $E(\mu)=(I-J)(u),$ where $I$ and $J$ are the standard
energy functionals in Kähler geometry, in the notation of \cite{au2,bbgz}.
Moreover, when $\beta=\pm1$ and $(\theta,dV)$ satisfies the compatibility
relation \ref{eq:theta is Ricci}, $F_{\beta}(\mu)=\mathcal{M}(u),$
where $\mathcal{M}$ is Mabuchi's K-energy functional associated to
$L=\pm K_{X}.$ For general data $(L,\theta,dV,\beta)$ (with $L$
ample) $F_{\beta}(\mu)=\mathcal{M}_{\beta}(u),$ where $\mathcal{M}_{\beta}(u)$
is the corresponding twisted Mabuchi functional (see \cite[Section 4.1]{berm6}).
These relations extend to the singular setup considered in \cite{bbegz}.

\subsection{Properties of the entropy and free energy}
\begin{lem}
\label{lem:on polarized has affine cont and energy appr}When $L$
is big, $dV$ has the energy approximation property and when $L$
is big and nef, the pluricomplex energy $E$ has the affine continuity
property (see Definition \ref{def:def of affine cont and energy appr}).
\end{lem}

\begin{proof}
The energy approximation property of any volume form $dV$ is essentially
well-known. In fact, it follows from results in \cite{be0,b-b-w}
that this property holds more generally for measures $dV$ which are
\emph{strongly determining} for $\text{PSH \ensuremath{(X,\theta)}}$
in the following sense: for any $u,v\in C^{0}(X)$ we have that $u\leq v$
a.e. $dV$ implies $u\leq v$ on $X.$ This is shown as follows. First,
as shown in \cite{be1}, in general a measure $dV$ has the energy
approximation property iff $F_{\beta}(\mu)$ Gamma-converges towards
$E$ as $\beta\rightarrow\infty$ iff for any $v\in C^{0}(X)$ 
\[
\lim_{\beta\rightarrow\infty}\inf_{\mu\in\mathcal{P}(X)}\left(F_{\beta}(\mu)+\left\langle v,\mu\right\rangle \right)=\inf_{\mu\in\mathcal{P}(X)}\left(E(\mu)+\left\langle v,\mu\right\rangle \right).
\]
First consider the case when $v=0.$ Then the inf in the lhs above
is attained at $\mu_{\beta}(=\theta_{u_{\beta}}^{n}/\text{vol \ensuremath{(L)}}),$
where $u_{\beta}$ is a solution to the Monge-Ampère equation \ref{eq:ma eq with beta}
and we may assume that $\sup_{X}u_{\beta}=0.$ A direct computation
yields
\begin{equation}
\inf_{\mu\in\mathcal{P}(X)}F_{\beta}(\mu)=F_{\beta}(\mu_{\beta})=-\mathcal{E}(u_{\beta})+\frac{1}{\beta}\log\int_{X}e^{\beta u}dV\label{eq:pf of lemma energy approx for L big}
\end{equation}
By \cite[Thm 2.1]{be0}, combined with\cite[Thm 1.14]{b-b-w}, assuming
that $dV$ is determining, $u_{\beta}$ converges in $L^{1}(X)$ towards
the envelope $u_{\theta}$ in formula \ref{eq:def of env} and the
first and second term in formula \ref{eq:pf of lemma energy approx for L big}
converge towards $-\mathcal{E}(u_{\theta})$ and $\sup_{X}u_{\theta}(=0),$
respectively. Finally, by \cite[Thm 5.3]{bbgz} the limit $-\mathcal{E}(u_{\theta})$
coincides with the inf of $E(\mu),$ as desired. The proof for a general
$v$ follows from replacing $\theta$ with $\theta+dd^{c}v.$ Strictly
speaking, it was assumed in \cite[Thm 1.14]{be0} that $L$ is ample,
but the same argument applies when $L$ is big. Alternatively, when
$L$ is big and $dV$ has a density in $L^{p}$ for $p>1$ (which,
for example,  is the case when $dV$ has divisorial singularities)
it follows from \cite[Thm 1.2]{be0} that $\left\Vert u_{\beta}-u_{\theta}\right\Vert _{L^{\infty}(X)}$
converges towards $0.$ The rest of argument proceeds as before. 

Next assume that $L$ is big and nef and consider an affine curve
$\mu_{t}:[0,1]\rightarrow\mathcal{P}(X)$ and assume that $E(\mu_{1})<\infty$
and $S(\mu_{1})>-\infty.$ Since $E$ and $-S$ are convex and finite
at the end point of the affine curve it follows that $E(\mu_{t})<\infty$
and that $S(\mu_{t})\geq-C.$ Hence, it follows from the continuity
result \ref{eq:energy entrop cont} that $E(\mu_{t})$ is continuous
wrt $t.$ 
\end{proof}
\begin{prop}
\label{prop:general prop of entropy and free on complex}Given a line
big line bundle $L$ assume, when $e>e_{0}$ that $L$ is also nef,
then
\begin{itemize}
\item $S(e)$ is finite and continuous on $[0,\infty[,$ strictly increasing
on $[0,e_{0}]$ and strictly decreasing on $[e_{0},\infty[.$
\item For any $e$ there exists a maximum entropy measure $\mu^{e}.$ 
\item For $e\in[0,e_{0}],$ $S(e)$ is strictly concave and $\mu^{e}$ is
unique.
\item $F(\beta)$ is differentiable on $]0,\infty[$ and $(F')[0,\infty]=[0,e_{0}]$
(where the end points are defined as limit of $F'(\beta))$ for $\beta\in]0,\infty[).$
\end{itemize}
\end{prop}

\begin{proof}
The finiteness follows from combining Lemma \ref{lem:energy appr implies S finite etc}
and Lemma \ref{lem:on polarized has affine cont and energy appr}
also using that $E$ is not bounded from above on $\mathcal{P}(X)$
(for example, $E(\mu)=\infty$ when $\mu$ charges a pluripolar subset,
such as any analytic subvariety of $X$). The monotonicity follows
from combining Lemma \ref{lem:mono of S} and Lemma \ref{lem:on polarized has affine cont and energy appr}.
Next, assume that $e\rightarrow e_{*}$ and denote by $\mu_{*}$ any
limit point of $\mu^{e}$ in the compact space $\mathcal{P}(X).$
Since $S(e)$ is finite and concave we have that for any $e\in[e_{1},e_{2}]$
$S(\mu^{e})\ge C,$ where $C=\min\{S(e_{1}),S(e_{2})\}.$ It thus
follows form the continuity result \ref{eq:energy entrop cont} that
$E(\mu^{e})\rightarrow E(\mu^{e_{*}}).$ Accordingly, to prove that
$S(e)$ is continuous at $e_{*}$ all that remains is to show that
$S_{v}(\mu_{*})\geq S(e_{*}).$ But since $S(\mu)$ is usc we have
$S(\mu_{*})\geq\limsup_{e\rightarrow e_{*}}S(\mu^{e})=S(e_{*})$ using
the continuity of $S(e)$ in the last equality. This argument also
shows that for any given $e_{*}$ there exists a maximum entropy measure,
namely $\mu_{*}$ in the previous argument. Its uniqueness when $e_{*}\leq e_{0}$
and the strict concavity of $S(e)$ on $[0,e_{0}]$ follows from Prop
\ref{prop:general results for S and F when E convex}, using that
$E$ is convex.
\end{proof}
\begin{thm}
\label{thm:minimizer satisf twisted}\cite{berm6,bbegz}Given $\beta\in\R$
any minimizer $\mu_{\beta}$ of $F(\mu)$ is the normalized volume
form of a Kähler metric $\omega_{\beta}$ solving the twisted KE equation
\ref{eq:twisted KE text}. 
\end{thm}

\section{The low energy region on complex manifolds}

\subsection{A refinement of Theorem \ref{thm:micro variational principle low energ intro} }

The following result is a slight refinement of Theorem \ref{thm:micro variational principle low energ intro},
stated in the introduction.
\begin{thm}
\label{thm:lower energy entro text}Given a number $e\in]0,e_{0}[$
the following holds: 
\begin{itemize}
\item $S$ coincides with its concave envelope $S^{**}$ at $e.$ 
\item There exists a unique maximizer $\mu^{e}$ of the entropy $S(\mu)$
on the subspace of all probability measure $\mu$ satisfying $E(\mu)=e$
or equivalently $E(\mu)\leq e.$ In particular, 
\[
S(\mu^{e})=S(e).
\]
\item The maximizer $\mu^{e}$ is the normalized volume form $\omega_{\beta}^{n}/\int_{X}\omega_{\beta}^{n}$
of the unique solution $\omega_{\beta}$ to the twisted Kähler-Einstein
equation \ref{eq:twisted KE text} such that $E(\omega_{\beta}^{n}/\int_{X}\omega_{\beta}^{n})=e.$ 
\item The corresponding function $e\mapsto\beta(e)$ gives a strictly decreasing
continuous map between $]0,e_{0})[$ and $]0,\infty[$ with the property
that
\[
\beta(e)=dS(e)/de
\]
and its inverse is given by the map $\beta\mapsto dF(\beta)/d\beta.$ 
\item In particular, $S(e)$ and $F(\beta)$ are $C^{1}-$smooth, strictly
increasing and strictly concave on $]0,e_{0}[$ and $]0,\infty[,$
respectively.
\end{itemize}
\end{thm}

The proof of the previous theorem follows directly from combining
Prop \ref{prop:general prop of entropy and free on complex} with
the following 
\begin{lem}
\label{lem:e strict decreasing for pos beta}Assume that it is not
the case that $\theta\geq0$ and $\theta^{n}=CdV$ for some constant
C. Then $e(\beta)$ is strictly decreasing on $]0,e_{0}[$ 
\end{lem}

\begin{proof}
Assume that $e(\beta_{0})=e(\beta_{1})$ with $\beta_{0}<\beta_{1}.$
By Prop \ref{prop:general prop of entropy and free on complex} it
follows that $f'(\beta)=e(\beta_{0})$ on $]\beta_{0},\beta_{1}[.$
Hence, $F(\beta_{1})-F(\beta_{0})=e(\beta_{0})(\beta_{1}-\beta_{0}).$
Since $F(\beta)=\beta e(\beta)-S(\mu_{\beta})$ it follows that $S(\mu_{\beta_{0}})=S(\mu_{\beta_{1}}).$
This means, since $e(\beta)=E(\mu_{\beta}),$ that $\mu_{\beta_{0}}$
and $\mu_{\beta_{1}}$ both minimize $F_{\beta_{1}}.$ But, since
$F_{\beta}(\mu)$ has a unique minimizer (by Prop \ref{prop:general results for S and F when E convex})
this forces $\mu_{\beta_{0}}=\mu_{\beta_{1}}$ It thus follows from
the uniqueness mod $\R$ of solutions to the complex Monge-Ampère
equation \ref{eq:MA eq with mu} that we can express $\mu_{\beta_{0}}=\mu_{\beta_{1}}=(dd^{c}u+\theta)^{n}/\text{\text{vol\ensuremath{(L)}}}$
for some $u\in\text{PSH}(X,\theta).$ But then Theorem \ref{thm:minimizer satisf twisted}
yields, by equation \ref{eq:ma eq with beta}, 
\[
e^{\beta_{0}(u+C_{0})}dV=e^{\beta_{1}(u+C_{1})}dV
\]
 for some constants $C_{0}$ and $C_{1}.$ Hence, $(\beta_{0}-\beta_{1})u=C_{1}-C_{0}$
a.e wrt $dV,$ which (by the assumption on $dV)$ implies $(\beta_{0}-\beta_{1})u=C_{1}-C_{0}$
on all of $X.$ Thus, $u$ is constant on $X.$ But then the equation
\ref{eq:ma eq with beta} forces $\theta^{n}=CdV$ for some constant
$C.$ This is precisely the case which has been excluded in the setup
introduced in Section \ref{subsec:Setup intro}. 
\end{proof}
\begin{rem}
\label{rem:e beta constant}If $\theta\geq0$ and $\theta^{n}=CdV$
for some constant $C,$ then $e(\beta)$ is constant on $]0,e_{0}[.$
Indeed, then $dV/\int_{X}dV$ minimizes both $-S(\mu)$ and $E(\mu)$
on $\mathcal{P}(X).$ In particular, it minimizes $F_{\beta}(\mu)$
for any $\beta>0,$ forcing $e(\beta)$ to be constant.
\end{rem}

\subsubsection{\label{subsec:Real-analyticity-when beta pos} Real-analyticity}

Consider the case when $X$ is non-singular, $L$ is ample and $dV$
is smooth. Then we have the following regularity result refining the
$C^{1}-$regularity in the previous theorem.
\begin{prop}
The functions $F(\beta)$ and $S(e)$ are real-analytic on $]0,\infty[$
and $]0,E_{0}[,$ respectively.
\end{prop}

\begin{proof}
The unique solution $u_{\beta}$ to equation \ref{eq:ma eq with beta}
is real-analytic in $\beta,$ by the argument using the implicit function
theorem in Banach spaces in the proof of \cite[Thm 7.9]{berm11}.
As a consequence, so is $F'(\beta).$ Next, by the previous theorem,
the derivative $S'(e)$ is invertible and may be expressed as the
inverse of the invertible map $\beta\mapsto F'(\beta).$ It follows
that $S'(e)$ is real-analytic in $e.$ Hence, so is $S(e).$ 
\end{proof}
The previous proposition holds, more generally, when $dV$ has divisorial
singularities along a divisor $\Delta$ with simple normal crossings.
Indeed, one can then instead use the implicit function theorem in
an appropriate Banach space (a Hölder ``edge'' space), using the
results announced in \cite{m-r} (see also \cite{do2,jmr} for the
case when the components of $\Delta$ do not intersect). 

\subsection{Proof of Theorem \ref{thm:micro variational principle low energ intro} }

First consider the case when $dV$ is a smooth volume form. Then  Theorem
\ref{thm:LDP for micro intro} follows from Theorem \ref{thm:LDP for quasi super}.
Indeed, $E^{(N)}$ Gamma converges towards $E,$ which is convex and
$dV$ has the energy approximation property (by Lemma \ref{lem:on polarized has affine cont and energy appr}).
To prove the general case it is, by the proof of Theorem \ref{thm:LDP for quasi super},
enough to know that the LDP for the corresponding Gibbs measures hold
for $\beta>0.$ But this is the content of \cite[Thm 5.1]{berm8 comma 5}.

\subsection{Proof of Cor \ref{cor:conv towards S etc in low energ intro}}

The convergence in probability in Theorem \ref{thm:micro variational principle low energ intro}
implies, in particular, the convergence of the expectations of $\delta_{N}$,
which amounts to the convergence of the first marginals in formula
\ref{eq:conv of first marginal in Cor}. 

\section{The high energy region on Fano varieties}

In this section we consider the setup of a Fano variety, introduced
in Section \ref{subsec:The-high-energy intro}. In particular, $L=-K_{X}$
and $\text{Ric }dV=\theta\geq0.$
\begin{thm}
\label{thm:unique on Fano}(Uniqueness) For any $\beta\in]-1,0[$
a solution to Aubin's equation \ref{eq:Aubin intro} - if it exists
- is unique. Moreover, for $\beta=-1,$ i.e. when $\omega$ is a \emph{Kähler-Einstein
(KE) metric, }the uniqueness holds modulo the action of the group
$\text{Aut }(X)_{0}.$
\end{thm}

\begin{proof}
When $X$ (and $\theta$ is smooth) this is shown in \cite{b-m} and
\cite{bern} and when $X$ is singular and $\beta=-1$ this is shown
in \cite{bbegz} and \cite[III]{c-d-s}, building on \cite{bbegz}.
Exactly the same proofs show, when $\beta\in]-1,0[,$ that any two
solution $\omega^{0}$ and $\omega^{1}$ can be connected by a curve
$\omega^{t}$ of solutions, $\omega^{t}=g_{t}^{*}\omega_{0},$ where
$g_{t}$ is the flow of a holomorphic vector field on $X.$ Hence,
applying $g_{t}^{*}$ to Aubin's equation for $\omega^{0}$ forces
$g_{t}^{*}\theta=\theta.$ But, by \cite[Thm 1.2]{berm9a} (which
applies more generally when $\theta$ is a positive klt current, i.e.
the local potentials of the pullback of $\theta$ to a resolution
of singularities of $X$ are in $L_{\text{loc}}^{p}$ for some $p>1$)
this implies that $\theta=0$ or that $g_{t}$ is the identity for
all $t.$ 
\end{proof}
Combining the previous theorem with Theorem \ref{thm:minimizer satisf twisted}
reveals that $F_{\beta}$ has at unique minimizer $\mu_{\beta}$ when
$\beta>-R(X).$ Its energy is denoted by $e(\beta).$
\begin{lem}
\label{lem: e beta contin etc on Fano}$e(\beta)$ is continuous on
$]-R(X),\infty[$ and if $X$ does not admit a KE metric, then $e(\beta)\rightarrow\infty$
as $\beta$ decreases to $-R(X).$ Moreover, if $X$ is admits a KE
metric, then $R(X)=-1$ and $e(\beta)\rightarrow e_{c}$ as $e$ increases
towards $e_{c}.$ 
\end{lem}

\begin{proof}
The continuity of $e(\beta)$ on $]-R(X),0]$ follows from the continuity
in Theorem \ref{thm:prop of energies}, precisely as in the proof
of Prop \ref{prop:general prop of entropy and free on complex}. When
$X$ does not admit a KE metric it is well-known that $e(\beta)\rightarrow\infty$
- a variational proof that applies in the singular setup is given
in the proof of \cite[Cor 1.3]{ber10}. Next, assume that $X$ admits
a KE metric and denote by $\mu_{\text{KE }}$the corresponding normalized
volume form. By \cite[Thm 4.8iii]{bbegz} $\mu_{\text{KE}}$ minimizes
$F_{-1}(\mu).$ It follows (as is well-known) that $R(X)=-1.$ Indeed,
given $\beta=-1+\epsilon$ for $\epsilon>0$ we get $F_{-1+\epsilon}\geq\epsilon E+F(\mu_{\text{KE}}).$
The existence of a minimizer of $F_{-1+\epsilon}$ then follows from
applying the continuity in Theorem \ref{thm:prop of energies} along
any sequence $\mu_{j}$ such that $F_{-1+\epsilon}(\mu_{j})$ converges
towards the inf of $F_{-1+\epsilon}.$ Any limit point of $\mu_{j}$
is thus a minimizer, which, by Theorem \ref{thm:minimizer satisf twisted},
solves the twisted KE with $\beta=-1+\epsilon,$ as desired. To prove
the last statement of the lemma first observe that, for any given
normalized KE volume form $\mu_{\text{KE}},$ 
\begin{equation}
E(\mu_{\beta})\leq E(\mu_{\text{KE }})\label{eq:E mu beta smaller than E mu KE}
\end{equation}
 Indeed, since $\mu_{\beta}$ minimizes $F_{\beta}$ we have (with
$\epsilon$ such that $\beta=-1+\epsilon$)
\[
F_{\beta}(\mu_{\beta}):=F_{-1}(\mu_{\beta})+\epsilon E(\mu_{\beta})\leq F_{\beta}(\mu_{\text{KE}}):=F_{-1}(\mu_{\text{KE}})+\epsilon E(\mu_{\text{KE}}).
\]
 But $F_{-1}(\mu_{\text{KE}})\leq F_{-1}(\mu_{\beta})$ (since $\mu_{\text{KE}}$
minimizes $F_{-1}(\mu)),$ which proves the desired inequality \ref{eq:E mu beta smaller than E mu KE}
(since $\epsilon>0).$ Next, it follows that when $\beta\in[-1,0]$
\begin{equation}
-S(\mu_{\beta})\leq F(0)+E(\mu_{\text{KE }})\label{eq:bound on minus S}
\end{equation}
 Indeed, since $F(\beta)$ is decreasing
\[
F(\beta):=F_{\beta}(\mu_{\beta})=\beta E(\mu_{\beta})-S(\mu_{\beta})\leq F(0).
\]
Since $\beta\in[-1,0]$ this proves the inequality \ref{eq:bound on minus S}.
But then it follows from the continuity of $E(\mu)$ in Thm \ref{thm:prop of energies}
and the upper-semicontinuity of $S(\mu)$ that any limit point $\mu_{-1}$
of $\mu_{\beta}$ minimizes $F_{-1}$ and thus $\mu_{-1}$ is KE.
Moreover, by the inequality \ref{eq:E mu beta smaller than E mu KE}
it satisfies $E(\mu_{-1})\leq E(\mu_{\text{KE }}).$ Hence, $e_{c}=E(\mu_{-1})$
for any limit point, as desired. 
\end{proof}
\begin{lem}
\label{lem:conv towards KE on Fano}Assume that $X$ admits a KE metric.
If $\text{Aut }(X)_{0}$ is trivial 
\begin{equation}
\lim_{\beta\rightarrow-1^{+}}\mu_{\beta}=\mu_{\text{KE }},\label{eq:conv towards ke in lemma}
\end{equation}
 where $\mu_{\text{KE }}$ is the normalized volume form of the unique
KE metric on $X.$ Moreover, if $\text{Aut }(X)_{0}$ is non-trivial
and it is assumed that $\theta>0$ on $X_{\text{reg}},$ then there
exists a unique KE volume form minimizing $E(\mu)$ (i.e. the inf
defining $E_{c}$ is attained at $\mu_{KE}$) and the convergence
\ref{eq:conv towards ke in lemma} holds.
\end{lem}

\begin{proof}
In view of the argument towards end of the proof of the previous lemma,
all that remains is to show that $E$ admits a unique minimizer on
the space of KE volume forms in $\mathcal{P}(X).$ When $\text{Aut }(X)_{0}$
is trivial this follows from the uniqueness of KE metrics (Theorem
\ref{thm:unique on Fano}). We thus consider the case when $\text{Aut }(X)_{0}$
is non-trivial and assume that $\theta>0$ on $X_{\text{reg}}.$ Setting
$G:=\text{Aut }(X)_{0}$ the space of KE volume forms in $\mathcal{P}(X)$
may (by Theorem \ref{thm:unique on Fano}) be expressed as the orbit
$G\cdot\mu_{\text{KE}},$ where $G$ acts by pull-back on any fixed
KE volume form $\mu_{\text{KE}}.$ Moreover, by \cite[III, Thm 4]{c-d-s},
$G$ is a complex reductive group or more precisely: $G$ is the complexification
of the compact group $K$ of all $g\in G$ fixing $\mu_{\text{KE}}.$
Hence, the space of all KE volume forms may be identified with $G/K.$
Now take $\mu_{\text{KE}}$ to be any normlaized KE volume form with
minimal $E(\mu)$ (its existence was established towards the end of
the previous proof). In particular, given any another normalized KE
volume form $\mu_{1}$ there is a one-parameter subgroup $g_{t}$
in $G$ fixing $K$ such that $g_{0}$ is the identity and $g_{t}^{*}\mu_{\text{KE }}=\mu_{1}$
when $t=1.$ To prove that $\mu_{\text{KE}}$ is uniquely determined
it is enough to prove that $t\mapsto E(g_{t}^{*}\mu_{\text{KE }})$
is strictly convex. But this follows from Lemma \ref{lem:E strict convex along geod}
in the appendix, using the well-known fact that, in general, when
$L=-K_{X}$ we can express $\left(g_{t}^{*}(\theta_{u_{0}})\right)^{n}=(\theta_{u_{t}})^{n},$
for a bounded regular geodesic $u_{t}$ in $\text{PSH}(X,\theta)\cap L^{\infty}(X),$
if $g_{t}^{*}(\theta_{u_{0}})$ is $K-$invariant, $u_{0}\in\text{PSH}(X,\theta)\cap L^{\infty}(X)$
and $u_{0}$ is smooth on $X_{\text{reg}}$ (indeed, $dd^{c}\theta+U$
may, locally over $X_{\text{reg}}$, be expressed as the pull-back
of a form on $X,$ which forces $(dd^{c}\theta+U)^{n+1}=0$).
\end{proof}
\begin{rem}
The previous lemma shows that the solution $\omega_{\beta}$ to Aubin's
equation \ref{eq:Aubin intro} converges in energy towards $\omega_{\text{KE}}$
(i.e. in the strong topology introduced in \cite{bbegz}). It then
follows from results in \cite{bbegz} that $\left\Vert u_{\beta}-u_{\text{KE}}\right\Vert _{L^{\infty}(X)}\rightarrow0$
as $\beta\rightarrow-1$ and $u_{\beta}\rightarrow u_{\text{KE }}$
in the $C_{\text{loc}}^{\infty}-$topology on $X_{\text{reg }}.$
Indeed, expressing $u_{\beta}=v_{\beta}+\sup_{X}u_{\beta}$ it is
enough to prove the convergence for the sup-normalized potential $v_{\beta}$
of $\omega_{\beta}.$ By the previous two lemmas $E(\omega_{\beta}^{n}/\text{Vol}(L))\rightarrow E(\omega_{-1}^{n}/\text{Vol}(L)),$
which implies that $\mathcal{E}(v_{\beta})\rightarrow\mathcal{E}(v_{-1}).$
It thus follows from \cite[Prop 1.4]{bbegz} that $\left\Vert e^{-v_{\beta}}\right\Vert _{L^{p}(dV)}\rightarrow\left\Vert e^{-v_{-1}}\right\Vert _{L^{p}(dV)}$
for any $p\geq1.$ As is well-known this implies, by Kolodziej type
estimates, that $\left\Vert v_{\beta}-v_{-1}\right\Vert _{L^{\infty}(X)}\rightarrow0$
(see the proof of \cite[Prop 2.12]{be0}). Finally, it follows from
\cite[Thm 11.1]{bbegz}, combined with standard local Evans-Krylov
estimates, that $v_{\beta}\rightarrow v_{-1}$ in the $C_{\text{loc}}^{\infty}-$topology
on $X_{\text{reg }}.$ 
\end{rem}

The following lemma is shown precisely as Lemma \ref{lem:e strict decreasing for pos beta}:
\begin{lem}
$e(\beta)$ is strictly decreasing on $]-R(X),0[$ 
\end{lem}

Given the results above the proof of the following result, which is
a slight refinement of Theorem \ref{thm:micro var principle Fano intro}
stated in the introduction, proceeds precisely as the proof of Theorem
\ref{thm:lower energy entro text}.
\begin{thm}
\label{thm:var pr Fano text}Assume given $e\in[0,e_{c}[.$ When $e_{c}<\infty$
we also allow $e=e_{c},$ but if $\text{Aut }(X)_{0}$ is non-trivial
the uniqueness results then require the additional assumption $\theta>0$
on $X_{\text{reg}}.$ 
\begin{itemize}
\item $S$ coincides with its concave envelope $S^{**}$ at $e.$ 
\item There exists a unique maximizer $\mu^{e}$ of the entropy $S(\mu)$
on the subspace of all probability measure $\mu$ with pluri-complex
energy $E(\mu)=e$ (or equivalently $E(\mu)\geq e).$ In particular,
\[
S(\mu^{e})=S(e),
\]
 where $S(e)$ is the specific entropy.
\item The maximizer $\mu^{e}$ is the normalized volume form $\omega_{\beta}^{n}/\int_{X}\omega_{\beta}^{n}$
of the unique solution $\omega_{\beta}$ to Aubin's equation \ref{eq:Aubin intro}
such that $E(\omega_{\beta}^{n}/\int_{X}\omega_{\beta}^{n})=e,$ where
$\beta<0$ iff $e>e_{0}.$ 
\item the corresponding function $e\mapsto\beta(e)$ gives a strictly decreasing
continuous map between $[0,e_{c}]$ and $[-R(X),\infty]$ with the
property that
\[
\beta(e)=dS(e)/de
\]
\item As a consequence, $S(e)$ strictly concave on $]0,e_{c}[,$ strictly
increasing on $]0,e_{0}[$ and strictly decreasing on $]e_{0},e_{c}[.$ 
\item The corresponding function $e\mapsto\beta(e)$ gives a strictly decreasing
continuous map between $[0,e_{c}]$ and $[-R(X),\infty]$ with the
property that
\[
\beta(e)=dS(e)/de
\]
and its inverse is given by the map $\beta\mapsto dF(\beta)/d\beta.$ 
\item In particular, $S(e)$ and $F(\beta)$ are $C^{1}-$smooth and strictly
concave on $]0,e_{c}[$ and $]0,\infty[,$ respectively. Moreover,
$F(\beta)$ is strictly increasing on $]0,\infty[,$ while $S(e)$
is strictly increasing on $]0,e_{0}[$ and strictly decreasing on
$]E,e_{c}[.$ 
\end{itemize}
\end{thm}

\subsection{Proof of Cor \ref{cor:not K poly}}

By Theorem \ref{thm:micro var principle Fano intro} $S(e)$ is strictly
concave on $[0,e_{c}[.$ Hence, if $X$ does not admit a KE metric,
i.e. $e_{c}=\infty,$ then $S(e)$ is strictly concave on $]0,\infty[.$
Conversely, assume that $X$ admits a KE metric. By Theorem \ref{thm:micro var principle Fano intro}
\begin{equation}
\lim_{\beta\rightarrow-1^{+}}F'(\beta)=e_{c},\,\,\,\lim_{e\rightarrow e_{c}^{-}}S'(e)=-1.\label{eq:pf cor K polyst S not concave}
\end{equation}
 Moreover, if $\text{Aut}(X)_{0}$ is assumed non-trivial, then
\begin{equation}
\{F(\beta)>-\infty\}=[-1,\infty[.\label{eq:F finite in pf Cor K poly}
\end{equation}
This follows from well-known results, using that $F_{-1}(\mu)$ can
be identified with Mabuchi's K-energy functional \cite{ma}. For completeness
we provide a proof. Assume, to get a contradiction, that there exists
$\epsilon>0$ and $C>0$ such that $F(-1-\epsilon)\geq-C.$ This means
that 
\begin{equation}
F_{-1}(\mu)\geq\epsilon E(\mu)+C\,\:\text{on \ensuremath{\mathcal{P}(X).}}\label{eq:coerciv ineq}
\end{equation}
Next, using the notation in the proof of Lemma \ref{lem:conv towards KE on Fano},
note that $F_{-1}(\mu_{t})$ is constant along the curve $\mu_{t}:=g_{t}^{*}\mu_{\text{KE }}.$
Indeed, since $\mu_{t}$ is KE it minimizes $F_{-1}(\mu).$ But then
the inequality \ref{eq:coerciv ineq} contradicts that $E(\mu_{t})\rightarrow\infty$
as $t\rightarrow\infty,$ by Remark \ref{rem:E tends to infty}.

Now assume, in order to get a contradiction, that $S(e)$ is strictly
concave on $]0,\infty[.$ This implies that $S(e)$ is concave on
all of $\R$ (since $S(e)=-\infty$ when $e<0).$ In particular $S^{**}(e)=S(e)$
on $\R$ and hence $S=F^{*}$ and $F=S^{*}.$ But then it follows
from the convergence \ref{eq:pf cor K polyst S not concave} that
$S'(e)$ is affine when $e>e_{0}.$ Indeed, by general properties
of the Legendre transform of finite concave functions $\partial S\subset\overline{\{S^{*}>-\infty\}}.$
By \ref{eq:F finite in pf Cor K poly} this means that $\partial S\subset[-1,\infty[.$
Moreover, since $S(e)$ is concave the convergence of $S'(e)$ in
formula \ref{eq:pf cor K polyst S not concave} implies that $\partial S(e)\subset]-\infty,-1].$
Hence, $\partial S(s)=\{1\},$ on $]e_{c},\infty[,$ showing that
$S'$ is differentiable with $S'(e)=-1$ on $]e_{c},\infty[$. Thus
$S(e)$ is indeed affine there, which contradicts that $S(e)$ was
assumed strictly concave. 

\subsubsection{Real-analyticity }

Consider now the case when $X$ is non-singular and $dV$ is smooth.
Then we have the following regularity result refining the $C^{1}-$regularity
in Theorem \ref{thm:var pr Fano text}, shown precisely as in Section
\ref{subsec:Real-analyticity-when beta pos}:
\begin{prop}
The functions $F(\beta)$ and $S(e)$ are real-analytic on $]-R(X),\infty[$
and $]0,E_{c}[,$ respectively.
\end{prop}

Using that, in this case, $u_{\beta}$ depends smoothly on $\beta$
an alternative proof of the strict concavity of $F(\beta)$ can also
be given. Indeed, expressing $E(\mu)=(I-J)(u)$ (as in Section \ref{subsec:Relations-to-the standard})
\cite[Thm 5.1]{b-m} (or rather its proof) yields
\[
dE(\mu_{\beta})/dt=-c\int_{X}(\Delta_{\omega_{\beta}}-\beta)\frac{\partial u_{\beta}}{\partial\beta}\Delta_{\omega_{\beta}}\frac{\partial u_{\beta}}{\partial\beta}\omega_{\beta}^{n},
\]
where $c>0$ and $\Delta_{\omega_{\beta}}$ denotes the Laplacian
wrt the Kähler metric $\omega_{\beta}.$ The integral above is strictly
positive when $\beta>-1.$ Indeed, $-\Delta_{\omega_{\beta}}+\beta$
is a positive definite operator when $\theta>0$ (as follows from
the Bochner-Kodaira-Nakano identity, using that $\text{Ric \ensuremath{\omega_{\beta}}\ensuremath{\geq\omega_{\beta}}}$
). 

\subsection{Conditional convergence results and proof of Theorem \ref{thm:n is one intro}}
\begin{thm}
\label{thm:conditional Fano}Assume that for any $\beta\in]-R(X),0[:$
\begin{equation}
-\lim_{N\rightarrow\infty}N^{-1}\log Z_{N,\beta}=F(\beta).\label{eq:conv on Z in Thm cond Fano}
\end{equation}
 Then, for any $e\in[e_{0},e_{c}[$ (and $\epsilon\in]0,\infty]),$
\[
\lim_{N\rightarrow\infty}N^{-1}\log\int_{\left\{ E^{(N)}\in]e,e+\epsilon[\right\} }dV^{\otimes N}=S(e)
\]
Moreover, if the convergence \ref{eq:conv on Z in Thm cond Fano}
holds for\emph{ any} given $dV$ on $X$ such that $\text{Ric \ensuremath{dV:=}}\theta>0$
on $X_{\text{reg}},$ then for any given such $dV,$ the real random
variable $\left\langle \delta_{N},u\right\rangle $ converges in probability
towards $\left\langle \mu^{e},u\right\rangle ,$ for any given $u\in C^{0}(X_{\text{reg}})$
and, as a consequence, 
\begin{equation}
\lim_{N\rightarrow\infty}\int_{X^{N-1}}\mu_{]e,e+\epsilon[}^{(N)}=\mu^{e}.\label{eq:conv of first marginal in conditional}
\end{equation}
Additionally, if $X$ is non-singular then $\delta_{N}$ converges
exponentially in probability at speed $N$ towards $\mu^{e}.$ 
\end{thm}

\begin{proof}
First assume that $X$ is non-singular. Combining Theorems \ref{thm:general conv of mean entropy },
\ref{thm:var pr Fano text} with the assumption \ref{eq:conv on Z in Thm cond Fano}
proves the convergence towards $S(e).$ Next, we will deduce the exponential
convergence of $\delta_{N}$ from Lemma \ref{lem:Cramer variant i R}.
Fix $u\in C^{2}(X)$ and set $dV_{u}:=e^{-u}dV.$ As in the proof
of Theorem \ref{thm:LDP for quasi super} we can then express 
\[
\E(e^{-N\left\langle u,\delta_{N}\right\rangle })=\int_{\left\{ E^{(N)}\in]e,e+\epsilon[\right\} }dV_{u}^{\otimes N}/\int_{\left\{ E^{(N)}\in]e,e+\epsilon[\right\} }dV^{\otimes N}.
\]
The assumption that $\theta>0$ ensures that $\theta_{tu}:=\text{Ric \ensuremath{dV_{tu}}}>0$
for $t\in[-T,T]$ where $T$ is positive and sufficiently small. Moreover,
for $T$ sufficiently small $dV_{tu}/\int dV_{tu}\neq\text{\ensuremath{\theta_{tu}^{n}/\int\theta_{tu}^{n}},}$
since, by assumption, $dV\int dV\neq\text{\ensuremath{\theta^{n}/\int\theta^{n}}. }$Furthermore,
$e<e_{c}(tu)$ for $t$ sufficiently small (since $e<e_{c})$. Fix
such a small $T$ and consider $t\in[-T,T].$ Applying the assumed
convergence \ref{eq:conv on Z in Thm cond Fano} to $dV$ and $dV_{tu}$
yields
\[
\lim_{N\rightarrow\infty}N^{-1}\log\E(e^{-N\left\langle tu,\delta_{N}\right\rangle })=S_{tu}(e)-S(e),
\]
using the notation in the proof of Theorem \ref{thm:LDP for quasi super}.
As in the latter theorem we have 
\begin{equation}
S_{tu}(e)=S_{tu}(\mu^{e,tu})=\sup_{\mu\in\mathcal{P}(X)^{\geq e},}\left\{ S_{0}(\mu)-t\left\langle u,\mu\right\rangle \right\} \label{eq:form for S tu in pf Fano}
\end{equation}
where $\mu^{e,tu}$ is the unique element where the corresponding
sup is attained (by Theorem \ref{thm:micro var principle Fano intro}).
By Lemma \ref{lem:conv in prob av E N} it will be enough to show
that $t\mapsto-S_{tu}(e)$ is Gateaux differentiable and that 
\[
-dS_{tu}(e)/dt_{|t=0}=\left\langle u,\mu^{e}\right\rangle .
\]
This identity is shown precisely as in the proof of Theorem \ref{thm:LDP for quasi super}
once we have verified that $t\mapsto E(\mu^{e,tu})$ is continuous.
But, by formula \ref{eq:form for S tu in pf Fano}, $S(E(\mu^{e,tu})=S_{tu}(e)\geq-S_{0}(e)-CT.$
Hence, the continuity in question results from the continuity of $E$
in Theorem \ref{thm:prop of energies}. 

Next, consider the case when $X$ is singular. Note that $\mu^{e}$
does not charge $X-X_{\text{sing}}.$ Indeed, since $-S(\mu^{e})<\infty$
it follows that $\mu^{e}$ is absolutely continuous wrt $dV$ and
since $X-X_{\text{sing}}$ is a proper analytic subvariety we have
$dV(X-X_{\text{sing}})=0.$ Hence, by basic integration theory it
is enough to prove the convergence of $\left\langle \delta_{N},u\right\rangle $
when $u\in C^{0}(K)$ for any compact subset $K$ of $X_{\text{reg}}.$
Moreover, since a continuous function $X$ may be uniformly approximated
by smooth functions, it is, in fact, enough to consider the case when
$u$ is smooth and compactly supported in $X_{\text{reg }}.$ The
result of the proof thus proceeds precisely as before.

Finally, the consequence \ref{eq:conv of first marginal in conditional}
follows readily form the definitions. Indeed, by the permutation symmetry
of $\mu_{]e,e+\epsilon[}^{(N)},$
\[
\left\langle \int_{X^{N-1}}\mu_{]e,e+\epsilon[}^{(N)},u\right\rangle =\int_{X^{N}}\mu_{]e,e+\epsilon[}^{(N)}\left\langle \delta_{N},u\right\rangle ,
\]
 which, by the convergence in probability shown above, converges towards
$\left\langle \mu^{e},u\right\rangle .$
\end{proof}
The following result provides conditions ensuring convergence in probability,
without assuming that $\theta>0$ and that $X$ be non-singular.
\begin{prop}
\label{thm:cond II}Given $e\in[e_{0},e_{c}[$ (and $\epsilon>0)$
the corresponding random empirical measure $\delta_{N}$ converges
in probability towards $\mu^{e}$ if the convergence \ref{eq:conv on Z in Thm cond Fano}
holds for $\beta=\beta(e)(=S'(\beta))$ and at least one of the following
two conditions are satisfied for any accumulation point $\Gamma_{\infty}^{e}$
of $(\delta_{N})_{*}\mu_{]e,e+\epsilon[}^{(N)}$ in $\mathcal{P}\left(\mathcal{P}(X)\right):$
\[
\text{(i)\,}\Gamma_{\infty}^{e}\,\,\text{is supported on \ensuremath{\{E(\mu)\geq e\}\subset\mathcal{P}(X)} }
\]
\[
\text{(ii)\,}\limsup_{N\rightarrow\infty}\int_{X^{N}}E^{(N)}\mu_{]e,e+\epsilon[}^{(N)}\leq\int_{\mathcal{P}(X)}E(\mu)\Gamma_{\infty}^{e}
\]
\end{prop}

\begin{proof}
Using the notation in the proof of Theorem \ref{thm:LDP for quasi super}
we have, as in formula \ref{eq:sub-add D-1},
\begin{equation}
\int_{\mathcal{P}(X)}\Gamma_{\infty}^{e}D^{(1)}(\mu)\leq\liminf_{N\rightarrow\infty}D^{(N)}(\mu_{]e,e+\epsilon[}^{(N)})\label{eq:sub-add D}
\end{equation}
 for any accumulation point $\Gamma_{\infty}^{e}$ of $(\delta_{N})_{*}\mu_{]e,e+\epsilon[}^{(N)}$
in $\mathcal{P}\left(\mathcal{P}(X)\right).$ The right hand side
above equals $-S(e),$ as follows from combining Theorems \ref{thm:general conv of mean entropy },
\ref{thm:var pr Fano text} with the assumed convergence \ref{eq:conv on Z in Thm cond Fano}.
Hence,
\[
\int_{\mathcal{P}(X)}\Gamma_{\infty}^{e}D^{(1)}(\mu)\leq D^{(1)}(\mu^{e}).
\]
 If condition $(i)$ holds it thus follows that $\Gamma_{\infty}^{e}=\delta_{\mu^{e}},$
since $D^{(1)}$ is lsc on $\mathcal{P}(X)$ and $\mu^{e}$ is the
unique minimizer of $D^{(1)}(\mu)$ on $\{E(\mu)\geq e\}\subset\mathcal{P}(X)$
(by Theorem \ref{thm:micro var principle Fano intro} ).

Alternatively, if condition $(ii)$ holds we can proceed as in the
proof of Theorem \ref{thm:LDP for quasi super}, where $\beta:=S'(e)>0),$
now using that $\beta<0$ and that $\mu^{e}$ is still the unique
minimizer of $F_{\beta},$ by Theorem \ref{thm:var pr Fano text}
and Lemma \ref{lemma:macrostate equivalence} (also using that $F_{\beta}$
is still lsc on $\mathcal{P}(X)$ when $\beta>-R(X),$ as follows
from combining the continuity in Theorem \ref{thm:prop of energies}
with the lower bound $F(\beta-\epsilon)>-\infty$, for $\epsilon$
sufficiently small). 
\end{proof}
As shown in \cite{be2}, the convergence \ref{eq:conv on Z in Thm cond Fano}
holds for all $\beta\in]-R(X),0[$ iff the following energy bound
holds for any accumulation point $\Gamma_{\beta,\infty}(\mu)$ of
$(\delta_{N})_{*}\mu_{\beta}^{(N)}$ in $\mathcal{P}\left(\mathcal{P}(X)\right)$
(where $\mu_{\beta}^{(N)}$ denotes the Gibbs measure \ref{eq:Gibbs meas}):
\[
\limsup_{N\rightarrow\infty}\int_{X^{N}}E^{(N)}\mu_{\beta}^{(N)}\leq\int_{\mathcal{P}(X)}E(\mu)\Gamma_{\beta,\infty}(\mu).
\]
 Thus Prop \ref{thm:cond II} yields the convergence in probability
of $\delta_{N}$ under the condition that the mean energy satisfies
appropriate upper bounds in both the microcanonical and the canonical
ensembles. When $n=1$ condition $(ii)$ is equivalent to the 
\[
\limsup_{N\rightarrow\infty}\int_{X^{2}}-G(x,y)\left(\int_{X^{N-2}}\mu_{]e,e+\epsilon[}^{(N)}\right)\leq\int_{X^{2}}-G(x,y)\mu_{2}^{(\infty)},
\]
for any accumulation point $\mu_{2}^{(\infty)}$ of $\int_{X^{N-2}}\mu_{]e,e+\epsilon[}^{(N)}$
in $\mathcal{P}(X^{2}),$ where $G(x,y)$ denotes the usc Green function
for the Laplacian wrt the Kähler metric defined by $dV$ (see Section
\ref{subsec:The-case-of the two sphere}). 

\subsubsection{Proof of Theorem \ref{thm:n is one intro}}

When $n=1$ the convergence assumed in Theorem \ref{thm:conditional Fano}
follows directly from the LDP for the Gibbs measures with $\beta<0$
established in \cite{be2}. Hence, Theorem \ref{thm:n is one intro}
follows from Theorem \ref{thm:conditional Fano}.

\subsection{\label{subsec:Generalization-to-log Fano}Generalization to log Fano
varieties $(X,\Delta)$}

The results above generalize (with the same proofs) to \emph{log Fano
varieties} $(X,\Delta),$ consisting of a normal projective variety
$X$ and an effective divisor $\Delta$ on $X$ such that $-(K_{X}+\Delta)$
defines an ample $\Q-$line bundle \cite{bbegz,li1,l-x-z}. Compared
to the general setup, introduced in Section \ref{subsec:Setup intro},
one then takes $L:=-(K_{X}+\Delta)$ and fix the data $(dV,\theta)$
consisting of volume form $dV$ and $\theta$ satisfying the compatibility
relation \ref{eq:theta is Ricci-1}. In the generalization of the
results for Fano varieties above the regular locus $X_{\text{reg }}$
should then be replaced by the log regular locus, defined as the complement
of the support of $\Delta$ in $X_{\text{reg }}.$

For example, when $n=1$ a log Fano variety $(X,\Delta)$ consists
of a collection of $m$ points $p_{1},...,p_{m}$ on the Riemann sphere
$\P^{1}$ with weights $w_{i}\in]0,1[$ such that $2-\sum_{i}w_{i}<0.$
When $m\geq3$ the K-polystability of $(X,\Delta)$ amounts to Troyanov's
weight condition \cite{tr}
\[
w_{i}<\sum_{i\neq j}w_{j}
\]
 and when $m\leq2$ the condition for K-polystability is obtained
by replacing $<$ in the previous inequality with $\leq.$ Moreover,
\[
R(X,\Delta)=2\frac{1-\max_{i}w_{i}}{2-\sum_{i}w_{i}}
\]

In particular, the generalization of Theorem \ref{thm:n is one intro}
holds to log Fano curves $(\P^{1},\Delta)$ if, in the statement about
exponential convergence, $\delta_{N}$ is replaced by the real-valued
random variable $\left\langle \delta_{N},u\right\rangle ,$ for a
given $C^{2}-$smooth function $u,$ compactly supported function
in the complement of the support of $\Delta.$ One advantage of this
logarithmic setup is that if $(\P^{1},\Delta)$ is K-polystable and
$m\geq3,$ then Theorem \ref{thm:micro var principle Fano intro}
and Theorem \ref{thm:n is one intro} extend to $e\in]0,e_{c}+\delta[$
for some $\delta>0.$ Indeed, in that case $F_{\beta}(\mu)$ admits
a minimizer when $\beta\in]\beta_{0},\infty[$ for some $\beta_{0}<-1.$
Moreover, using the implicit function theorem it can be shown, as
in the proof of \cite[Thm 8]{be00}, that the corresponding equation
\ref{eq:ma eq with beta} has a unique solution for $\beta\in]\beta_{0},\infty[.$
As a consequence, $F_{\beta}(\mu)$ has a unique minimizer when $\beta\in]\beta_{0},\infty[.$
Moreover, by results in \cite{be2}, the corresponding Gibbs measures
satisfy a LDP when $\beta\in]\beta_{0},\infty[.$ We can thus apply
Prop \ref{thm:general conv of mean entropy } on $]\beta_{0},\infty[$
and proceed as before.

\subsection{Results on general compact Kähler manifolds}

Let now $X$ be a compact Kähler manifold and consider the general
setup introduced in Section \ref{subsec:Setup intro}. More generally,
assume that $dV$ has a density in $L^{p}(X)$ for some $p>1.$ As
shown very recently in \cite{l-p}, there then exists some $\beta_{0}\in]-\infty,0[$
such that the equation \ref{eq:ma eq with beta} has a unique solution
$u_{\beta}$ for any $\beta\in]\beta_{0},\infty[$ (the number $\beta_{0}$
depends on $n$ and the $L^{p}(X)-$norm of the density of $dV).$
As a consequence, the proof of Theorem \ref{thm:var pr Fano text}
reveals that the results in \ref{thm:var pr Fano text} hold on a
general compact Kähler manifold $X$ when $e<e(\beta_{0}),$ where
$e(\beta_{0}):=\lim_{\beta\searrow\beta_{0}}E(\mu_{\beta})$ for $\mu_{\beta}:=\theta_{u_{\beta}}^{n}/\text{vol}\ensuremath{(L)}.$

\section{\label{sec:Hamiltonian-flows-and}Hamiltonian flows and the complex
Euler-Monge\textendash Ampère equation}

Now specialize the setup introduced in Section \ref{subsec:Setup intro}
to the case when $L$ is ample and $dV$ is a smooth volume form on
$X.$ After rescaling, $dV$ can be expressed as $dV=\omega_{X}^{n},$
for a unique Kähler form $\omega_{X}$ in $c_{1}(L)$ (by solving
the Calabi-Yau equation \ref{eq:MA eq with mu}). The Kähler form
$\omega_{X}$ naturally induces a Kähler form $\omega_{X^{N}}$ on
$X^{N}$ that we normalize as follows 
\[
\omega_{X^{N}}:=\frac{1}{N}\left(\omega_{X}(x_{1})+....+\omega_{X}(x_{N})\right)
\]
The Hamiltonian flow of $E^{(N)}$ on the symplectic manifold $(X^{N},\omega_{X^{N}})$
can be expressed as follows, using that $\omega_{X}$ is a compatible
with the complex structure $J:$ 
\begin{equation}
\frac{dx_{i}(t)}{dt}=-J\nabla_{x_{i}}E^{(N)}\left((x_{1}(t),...,x_{N}(t))\right),\label{eq:Ham flow of E N}
\end{equation}
 where $\nabla$ denotes the gradient with respect to the corresponding
Kähler metric $\omega_{X}.$
\begin{lem}
Assume given $\boldsymbol{x}(0):=(x_{1}(0),....x_{N}(0))\in X^{N}$
such that $E^{(N)}(\boldsymbol{x}(0))=e\in\R.$ Then there exists
a unique long-time solution $\boldsymbol{x}(t):=(x_{1}(t),....x_{N}(t))\in X^{N}$
of the Hamiltonian flow of $E^{(N)}$ on $(X^{N},\omega_{X^{N}})$,
emanating from $\boldsymbol{x}(0).$ 
\end{lem}

\begin{proof}
Let $D_{N}$ be the proper closed subset of $X^{N}$ where $E^{(N)}=\infty.$
Since $X^{N}$ is compact and $E^{(N)}$ is smooth on $X^{N}-D_{N}$
the existence on some time-interval $[0,T_{N}]$ follows from standard
existence results for ODEs. Next, note that since the flow is Hamiltonian
the Hamiltonian $E^{(N)}$ is constant along the flow. But then is
follows that the distance of $\boldsymbol{x}(t)$ to $D_{N}$ is bounded
by a constant only depending on $E^{(N)}(\boldsymbol{x}(t)),$ which
means that the flow stays in a compact subset of $X^{N}-D_{N}.$ Hence,
the maximal time of existence $T_{N}$ is, in fact, equal to $\infty,$
as desired.
\end{proof}
The next conjecture proposes that the large $N-$limit of the Hamiltonian
flow \ref{eq:Ham flow of E N} converges to a solution to the evolution
equation, 

\begin{equation}
\frac{\partial\rho_{t}}{\partial t}=-\nabla\cdot(\rho_{t}J\nabla u_{\rho_{t}}),\,\,\,\,\frac{1}{\text{Vol \ensuremath{(L)}}}(dd^{c}u_{\rho_{t}}+\theta)^{n}=\rho_{t}\omega_{x}^{n},\label{eq:complex euler monge text}
\end{equation}
 that we shall call the\emph{  complex Euler-Monge\textendash Ampère
equation} on $X.$ A solution will be called \emph{smooth} if, for
any fixed $t>0,$ $\rho_{t}$ and $u_{\rho_{t}}$ are in $C^{\infty}(X)$
and $t\mapsto\rho_{t}(x)$ is differentiable for any fixed $x\in X.$
When considering the corresponding initial value problem we also assume
that $\rho_{t}\rightarrow\rho_{0}$ in $L^{\infty}(X),$ as $t\rightarrow0.$ 
\begin{conjecture}
\label{conj:mean field limit of flow}Assume that at time $t=0$ 
\begin{equation}
\lim_{N\rightarrow\infty}\delta_{N}\left((x_{1}(t),....,x_{N}(t))\right)=\rho_{0}dV,\label{eq:conv towards rho zero in conj}
\end{equation}
 then at any later time $t>0$
\[
\lim_{N\rightarrow\infty}\delta_{N}\left((x_{1}(t),....,x_{N}(t))\right)=\rho_{t}dV,
\]
 where $\rho_{t}$ is a solution of the corresponding initial problem
for the evolution equation \ref{eq:complex euler monge text}. 
\end{conjecture}

To make this conjecture precise the regularity of $\rho_{0}$, the
notion of a solution and the precise meaning of the convergence should
be specified.

It seems natural to ask is there exist global in time smooth solution
to equation \ref{eq:complex euler monge text};
\[
\text{Question:\,If \ensuremath{\rho_{0}\in C^{\infty}(X),\,\rho_{0}>0,}then is there a smooth solution \ensuremath{\rho_{t}} }\text{for any \ensuremath{t\in[0,\infty[?}}
\]
When $n=1$ the affirmative answer and the uniqueness of the solution
follows from well-known results for incompressible Euler type equations
in 2D, as discussed in Section \ref{sec:Comparison-with-point}. In
the present more general setup uniqueness of smooth solutions is proved
in Section \ref{subsec:A-stability-estimate} below.
\begin{rem}
\label{rem:SG}When $n=2$ and $X$ is the torus $(S^{1}\times iS^{1})\times(S^{1}\times iS^{1}),$
for $S^{1}=\R/\Z,$ $\theta$ and $dV$ are invariant and the initial
data is taken to be independent of the imaginary factors, existence
of \emph{weak }solutions follows from the existence of weak quasi-convex
solutions to the dual formulation of the semi-geostrophic equation
\cite{be-br,Loe}. The uniqueness of solutions to the latter equation
is established in \cite{Loe,f-t,Fi} under various regularity assumptions. 
\end{rem}

We next suggest two formal ``proofs'' of the convergence in Conjecture
\ref{conj:mean field limit of flow}. The starting point is the following
formula \cite[Prop 2.7]{berm6} for the differential $dE_{|\mu}$
of the pluricomplex energy $E$ restricted to the space of smooth
volume forms in $\mathcal{P}(X):$ 
\begin{equation}
dE_{|\mu}=-u_{\mu},\label{eq:diff of E}
\end{equation}
 where $u_{\mu}$ is a solution to the to the complex Monge-Ampère
equation \ref{eq:MA eq with mu} (using that, by duality, the differential
of a function on $\mathcal{P}(X)$ may, under appropriate regularity
assumptions, be represented by an element in $C^{0}(X)/\R).$ 

\subsection{Formal ``proof'' of Conjecture \ref{conj:mean field limit of flow}
using mean field limits}

Let $X$ be a compact manifold and $\boldsymbol{b}^{(N)}$ a sequence
of vector fields on $X^{N}.$ Consider the flow on $X^{N}$ defined
by the corresponding ODE:
\begin{equation}
\frac{d\boldsymbol{x}_{N}(t)}{dt}=\boldsymbol{b}^{(N)}(\boldsymbol{x}_{N}(t))\label{eq:flow of b N}
\end{equation}
Assume that the flows are symmetric and that $\boldsymbol{b}^{(N)}$
is of the form 
\[
\boldsymbol{b}^{(N)}(\boldsymbol{x}_{N}(t))=\left(b_{\delta_{N}(\boldsymbol{x}_{N}(t))}(x_{1}(t)),....,b_{\delta_{N}(\boldsymbol{x}_{N}(t))}(x_{1}(t))\right)
\]
where $b_{\mu}$ is a vector field on $X$ for any given $\mu\in\mathcal{P}(X).$
Then $\boldsymbol{b}^{(N)}$ is said to be of \emph{mean field type.}
Under suitable regularity assumptions for the map 
\[
\mathcal{P}(X)\times X\rightarrow TX,\,\,\,(\mu,x)\mapsto b_{\mu}(x)
\]
 (requiring, in particular, that $x\mapsto$ $b_{\mu}(x)$ be Lipschitz
continuous) the flows \ref{eq:flow of b N} admit a ``mean field
limit'' as $N\rightarrow\infty$ in the following sense: if at $t=0$
\[
\frac{1}{N}\sum_{i=1}^{N}\delta_{x_{i}(0)}\rightarrow\mu_{0},
\]
weakly in $\mathcal{P}(X),$ then, at any later time $t>0,$ 
\[
\frac{1}{N}\sum_{i=1}^{N}\delta_{x_{i}(t)}\rightarrow\mu_{t},
\]
weakly in $\mathcal{P}(X),$ where $\mu_{t}$ is the unique solution
to the following evolution equation on $\mathcal{P}(X):$ 
\[
\frac{d\mu_{t}}{dt}=\mathcal{L}_{b_{\mu_{t}}}\mu_{t},
\]
 where $\mathcal{L}_{b_{\mu_{t}}}$ denotes the Lie derivative of
$\mu_{t}$ along the vector field $b_{\mu_{t}}.$ This is essentially
well-known (for example, this convergence is the deterministic version
of the results in \cite{d-g} and \cite[Section 6]{m-m-w}). In particular,
if $X$ is endowed with a Riemannian metric $g$ and $\mu_{0}=\rho_{0}dV_{g}$
then $\mu_{t}=\rho_{t}dV$ where $\rho_{t}$ is a solution to the
following evolution equation 
\[
\frac{\partial\rho_{t}}{\partial t}=-\nabla\cdot(\rho_{t}b_{\rho_{t}}).
\]
 To see the connection to Conjecture \ref{conj:mean field limit of flow}
now assume that 
\[
\boldsymbol{b}^{(N)}(\boldsymbol{x}_{N}(t))=-J\nabla E^{(N)},
\]
 where 
\begin{equation}
E^{(N)}(\boldsymbol{x}_{N})=E(\delta_{N}(\boldsymbol{x}_{N}(t))\label{eq:E N as E of delta N text}
\end{equation}
 for a given sufficiently regular functional $E$ on $\mathcal{P}(X).$
Then, by the chain rule, 
\[
b_{\mu}(x)=-J\nabla v_{\mu}(x),\,\,\,\,dE_{|\mu}=v_{\mu}
\]
 where $v_{\mu}$ is a smooth function on $X$ representing the differential
of $E$ at $\mu.$ In the present setup the identity \ref{eq:E N as E of delta N text}
does not hold (in fact, $E(\delta_{N}(\boldsymbol{x}_{N}(t))=\infty).$
Bur using instead the Gamma-convergence of $E^{(N)}$ towards $E$
yields a formal ``proof'' of Conjecture \ref{conj:mean field limit of flow}.

\subsection{\label{subsec:Formal-proof using Ham}Formal ``proof'' of Conjecture
\ref{conj:mean field limit of flow} using an infinite dimensional
Hamiltonian structure}

We will next exploit that the equation \ref{eq:complex euler monge text}
can be viewed as a Hamiltonian flow on $\mathcal{P}(X),$ inspired
by the well-known case of the incompressible Euler equation in $\R^{2}$
\cite{m-w,m-w-r-s-s}. Recall that, in general, a symplectic form
$\omega_{X}$ on a manifold $X$ induces a \emph{Poisson structure}
on $X,$ defined by the following Poisson brackets on $C^{\infty}(X):$
\[
\{f_{1},f_{2}\}:=\omega_{X}(V_{f_{1}},V_{f_{2}}),
\]
 where $V_{f}$ denotes the Hamiltonian vector field on $X$ with
Hamiltonian $f.$ The Hamiltonian flow of a function $h$ on $X$
acts by pull-back on $C^{\infty}(X)$ and is equivalent to the flow
on $C^{\infty}(X)$ defined by 
\begin{equation}
\frac{df_{t}}{dt}=-\{f_{t},h\}\label{eq:Heisenberg eq}
\end{equation}
In the present Kähler setting we have 
\[
\{f_{1},f_{2}\}:=-\nabla f_{1}\cdot J\nabla f_{2}.
\]
As a consequence, the complex Euler-Monge-Ampère equation \ref{eq:complex euler monge text}
may be expressed as follows (using that $\nabla\cdot(J\nabla\cdot)=0$): 

\begin{equation}
\frac{\partial\rho_{t}}{\partial t}=\{\rho_{t},u_{\rho}\},\,\,\,u_{\rho}=-dE_{|\rho}\label{eq:Poisson form of complex euler}
\end{equation}
This suggests that the flow $\rho_{t}$ may be viewed as a Hamiltonian
flow on $\mathcal{P}(X)$ with Hamiltonian $E(\mu).$\emph{ }To make
this precise first observe that the Poisson structure on $X$ naturally
induces a Poisson structure\emph{ }on $\mathcal{P}(X):$ 
\[
\{F_{1},F_{2}\}(\mu):=\left\langle \{dF_{1},dF_{2}\},\mu\right\rangle :=\int_{X}\{dF_{1},dF_{2}\}\mu
\]
 where the differential $dF$ of a smooth function $F$ on $\mathcal{P}(X)$
has been identified with an element in $C^{\infty}(X)$ in the usual
way. A direct calculation (using one integration by parts) reveals
that 
\begin{equation}
\frac{\partial F(\rho_{t}dV)}{\partial t}=-\{F,E\}(\rho_{t}dV)\label{eq:Heisenberg eq on space P X}
\end{equation}
 in analogy with formula \ref{eq:Heisenberg eq}. While the Poisson
structure on $\mathcal{P}(X)$ does not correspond to a global symplectic,
$\mathcal{P}(X)$ admits a foliation of symplectic leaves for the
Poisson structure. To see this denote by $\mathcal{G}$ the infinite
dimensional Lie group of Hamiltonian diffeomorphisms of $(X,\omega_{X})$
and by $\mathfrak{g}$ its Lie algebra. The symplectic leaves in question
are the orbits of $\mathcal{G}$ in $\mathcal{P}(X)$ (acting by push-forward).
Indeed, $\left(C^{\infty}(X),\{.,.\}\right)$ naturally identifies
with the Lie algebra $\mathfrak{g}$ of $\mathcal{G}$ and $\mathcal{P}(X)$
with a convex subset of its dual $\mathfrak{g}{}^{*}.$ This means
that the Poisson structure for $\mathcal{P}(X)$ identifies with the
natural Lie algebra structure on $\mathfrak{g}^{*}.$ By general theory,
$\mathcal{G}$ admits a natural action on $\mathfrak{g}{}^{*}$ -
the co-adjoint action - and the corresponding $\mathcal{G}-$orbits
admit a canonical symplectic structure, defined by the \emph{Kirillov-Kostant-Souriau
symplectic form} \cite{m-w,m-w-r-s-s}.\emph{ }In the present setting
the co-adjoint action simply acts by push-forward on $\mathcal{P}(X).$
Formula \ref{eq:Heisenberg eq on space P X} thus yields the following
\begin{prop}
Given smooth initial data $\rho_{0}$ the flow $\rho_{t}$ defined
by a classical solution of the  complex Euler-Monge-Ampère equation
\ref{eq:complex euler monge text} is the Hamiltonian flow of the
pluricomplex energy $E$ on the $\mathcal{G}-$orbit $\mathcal{G}\cdot(\rho_{0}\omega_{X}^{n})$
in $\mathcal{P}(X)$ endowed with the Kirillov-Kostant-Souriau symplectic
form. 
\end{prop}

The connection to Conjecture \ref{conj:mean field limit of flow}
stems from the basic observation that the map $\delta_{N}:X^{N}\rightarrow\mathcal{P}(X)$
intertwines the corresponding Poisson structures and the symplectic
forms of the corresponding $\mathcal{G}-$orbits. This means that
we can view the Hamiltonian flow \ref{eq:Ham flow of E N} for $E^{(N)}$
on $X^{N}$, with given initial data $(x_{1}(0),....,x_{N}(0)).$
as the Hamiltonian flow of $E^{(N)}$ on the $\mathcal{G}-$orbit
of $\delta_{N}\left((x_{1}(0),....,x_{N}(0))\right).$ Since $E^{(N)}$
Gamma-converges towards $E$ on $\mathcal{P}(X)$ and the $\mathcal{G}-$orbit
of $\delta_{N}\left((x_{1}(0),....,x_{N}(0))\right)$ converges (in
an appropriate sense) to the $\mathcal{G}-$orbit of $\rho_{0}dV$
in $\mathcal{P}(X)$ (assuming the convergence \ref{eq:conv towards rho zero in conj}),
Conjecture \ref{conj:mean field limit of flow} thus formally follows
from the previous proposition. 

\subsection{\label{subsec:Conserved-quantities}Conserved quantities}

As discussed in Section \ref{subsec:Relations-to-Hamiltonian}, combining
Conjecture \ref{conj:mean field limit of flow} with an assumption
about ergodicity suggests that the density of a maximum entropy measures
$\mu^{e}$ with prescribed energy $e$ is a stationary solution of
the evolution equation \ref{eq:complex euler monge text}. When $e\in]0,e_{0}[$
(or when $e\in]0,e_{c}[$ in the Fano setup) this follows from combining
Theorem \ref{thm:micro variational principle low energ intro} (or
Theorem \ref{thm:micro var principle Fano intro}) with the following
result:
\begin{prop}
\label{prop:Casimir}The normalized volume form of any twisted Kähler-Einstein
metric $\omega_{\beta}$ is a stationary solution to the evolution
equation \ref{eq:complex euler monge text}. Moreover, in general,
the pluricomplex energy $E$ is constant along the evolution and so
are the functionals 
\[
I_{f}(\rho):=\int_{X}f(\rho)\omega_{X}^{n},
\]
 defined by any given smooth function $f$ on $[0,\infty[.$ As a
consequence, $S(\rho\omega_{X}^{n})$ is constant along the evolution. 
\end{prop}

\begin{proof}
For any given smooth function $h$ the vector field $J\nabla h$ is
divergence free, since it is a Hamiltonian vector field. Hence, the
evolution equation \ref{eq:complex euler monge text} can be expressed
as 
\[
\frac{\partial\rho_{t}}{\partial t}=-\nabla\rho_{t}\cdot J\nabla u_{\rho_{t}}
\]
 As a consequence, if $\rho$ has the property that $u_{\rho}$ is
a smooth function of $\rho:$ $u_{\rho}=f(\rho),$ then $\rho$ is
a stationary solution of equation\ref{eq:complex euler monge text}.
Indeed, by the chain rule $\nabla u_{\rho}=f'(\rho)\nabla\rho$ and
$\nabla\rho\cdot J\nabla\rho=0$ for any $\rho.$ In particular, taking
$f(\rho)=\beta^{-1}\log\rho$ for $\rho$ the density of the normalized
volume form of a twisted Kähler-Einstein metric corresponding to $\beta$
and using the equation \ref{eq:ma eq with beta} thus proves the last
statement of the proposition (using that $\rho>0).$ Similarly, 
\[
-\frac{\partial I_{f}(\rho_{t})}{\partial t}=\int\nabla f(\rho_{t})\cdot J\nabla u_{\rho_{t}}dV=-\int f(\rho_{t})\cdot\nabla(J\nabla u_{\rho_{t}})dV=0
\]
 using that $J\nabla u_{\rho_{t}}$ is divergence free. This applies,
in particular, to $f(t):=-(t+\epsilon)\log(t+\epsilon)$ and letting
$\epsilon\rightarrow0$ shows, using the monotone convergence theorem,
that $S(\rho_{t}\omega_{X}^{n})$ is constant along the evolution.
Finally, 
\[
\frac{\partial E(\rho_{t})}{\partial t}=-\int u_{\rho}\frac{\partial\rho_{t}}{\partial t}dV=\int\rho_{t}\nabla u_{\rho}\cdot J\nabla u_{\rho_{t}}dV=0
\]
\end{proof}
The previous proposition reveals that as soon as $\mu^{e}$ is uniquely
determined by the energy $e$ then it is stationary (since both $E(\mu)$
and $S(\mu)$ are preserved under the evolution). Moreover, when $e<e_{0},$
the stationary solution $\mu^{e}$ is \emph{formally stable} in the
sense of \cite{h-m-r-w}. More precisly, by the previous proposition
$F_{\beta}(\mu)$ is invariant under the flow for any $\beta$ and
for $\beta=S'(e)$ $\mu^{e}$ is a critical point of $F_{\beta}$
and the second variation of $F_{\beta}(\mu)$ is positive definite
at $\mu^{e}$ (since $F_{\beta}$ is strictly convex for any $\beta>0).$
In the Fano case this formal stability also holds when $e<e_{c}$
and $e=e_{c}$ if $\text{\ensuremath{\text{Aut \ensuremath{(X})}_{0}}}$
is trivial, using the well-known convexity properties of the (twisted)
Mabuchi functional with respect to the Mabuchi metric \cite{ma}. 

\subsubsection{Relations to mixing theory}

Now fix a smooth strictly positive probability density $\rho_{0}$
and assume that there exists a smooth solution $\rho_{t}$ to the
evolution equation \ref{eq:complex euler monge text} for $t\in]0,\infty[.$
Denote by $\Omega(\rho_{0})$ the limit points of the curve $t\mapsto\rho_{t}dV$
in\emph{ $\mathcal{P}(X)$ }as $t\rightarrow\infty.$ In general,
since $S(\rho_{t})$ and $E(\rho_{t})$ are constant it follows from
the continuity in Theorem \ref{thm:prop of energies} and the upper-semicontinuity
of $S(\mu)$ that 
\[
E(\mu)=E(\rho_{0}dV),\,\,\,S(\mu)\geq S(\rho_{0}dV)\,\,\,\forall\mu\in\Omega(\rho_{0})
\]
 More generally, $I_{f}(\mu)\geq I_{f}(\rho_{0}dV)$ for any smooth
$f$ such that $I_{f}$ is usc on $\mathcal{P}(X).$ This is in line
with the mixing theory for the Euler equation in \cite{shn}.

\subsection{\label{subsec:A-stability-estimate}A stability estimate and uniqueness
of solutions}
\begin{lem}
Let $\rho_{t}$ be a smooth solution to the evolution equation \ref{eq:complex euler monge text}.
Then $\sup\rho_{t}$ and $\inf\rho_{t}$ are independent of $t.$
Moreover, if $\inf_{X}\rho_{0}>0,$ then exists a strictly positive
constant $C$ only depending on upper bounds on $\sup_{X}\rho_{0}$
and $1/\inf_{X}\rho_{0}$ such that 
\[
C^{-1}\omega_{X}\leq\theta_{u_{t}}\leq C\omega_{X}
\]
\end{lem}

\begin{proof}
Given $p\in\R$ and $\epsilon>0$ set $f_{\epsilon,p}(t):=(t+\epsilon)^{p}.$
By Prop \ref{prop:Casimir} $I_{f_{\epsilon,p}}(\rho)$ is independent
of $t.$ Letting first $\epsilon\rightarrow0$ and then $p\rightarrow\pm\infty$
thus proves that $\sup\rho_{t}$ and $\inf\rho_{t}$ are independent
of $t.$ Hence, we have uniform bounds $\theta_{u_{t}}^{n}/dV\leq C_{0}$
and $\theta_{u_{t}}^{n}/dV\geq/C_{0}$ for some positive constant
$C_{0}$ By Yau's Laplacian and uniform estimates for the complex
Monge-Ampère \cite{y} it follows that $\text{Tr}\theta_{u_{t}}\leq A,$
where the trace is defined with respect to $\omega_{X}$ and $A$
only depends on an upper bound on $C_{0}$ (see e.g. \cite[Thm 11.1]{bbegz}
for a generalization of this estimate). But then $\theta_{u_{t}}^{n}/dV\geq C_{0}$
implies that we also have a lower bound $\text{Tr}\theta_{u}\geq B>0$
for a uniform constant $B.$ Indeed, for any positive definite matrix
$A$ we have that $\text{Tr}A^{-1}\leq c_{n}(\det A)^{-1}(\text{Tr}A)^{n-1}.$
Applying this inequality to $\theta_{u}$ gives a uniform strict lower
bound on the eigenvalues of $\theta_{u}$ with respect to $\omega_{u}.$ 
\end{proof}
Given a function $f\in L^{2}(X,dV)$ such that $\int fdV=0$ denote
by $\left\Vert f\right\Vert _{-1}$ its Sobolev $H^{-1}-$norm, i.e.
$\left\Vert f\right\Vert _{-1}^{2}:=-\left\langle \Delta^{-1}f,f\right\rangle ,$
where $\Delta$ denotes the Laplacian wrt the fixed Kähler metric
$\omega_{X}.$ The following $H^{-1}-$stability estimate appears
- to the best of the author's knowledge - to be new already in the
case of the semi-geostrophic equation (Remark \ref{rem:SG}).
\begin{thm}
(Stability) Let $\rho_{t}$ and $\rho_{t}'$ be two smooth solutions
to equation \ref{eq:complex euler monge text}. If $\inf_{X}\rho_{0}>0,$
then 
\[
\left\Vert \rho_{t}'-\rho_{t}\right\Vert _{-1}\leq\left\Vert \rho_{0}'-\rho_{0}\right\Vert _{-1}e^{Ct}
\]
 for a constant $C$ only depending on upper bounds on $\sup_{X}\rho_{0}$
and $1/\inf_{X}\rho_{0}.$
\end{thm}

\begin{proof}
Given two solutions $\rho_{t}$ and $\rho_{t}'$ we define $U_{t}$
and $U_{t}'$ (mod $\R)$ by $\Delta U_{t}=\rho_{t}-1/\int dV$ and
$\Delta U_{t}'=\rho_{t}'-1/\int dV.$ It will be enough to show that
\begin{equation}
\frac{d}{dt}\left(\left\Vert \nabla(U_{t}-U_{t}')\right\Vert ^{2}\right)\leq2C\left\Vert \nabla(U_{t}-U_{t}')\right\Vert ^{2}\label{eq:diff ineq in pf uniqu}
\end{equation}
To prove the inequality \ref{eq:diff ineq in pf uniqu} first observe
that 
\[
\frac{1}{2}\frac{d}{dt}\left\langle \nabla(U_{t}-U_{t}'),\nabla(U_{t}-U_{t}')\right\rangle =-\left\langle \frac{d}{dt}\Delta(U_{t}-U_{t}'),(U_{t}-U_{t}')\right\rangle =-\left\langle \frac{d\rho_{t}}{dt}-\frac{d\rho_{t}'}{dt},(U_{t}-U_{t}')\right\rangle =
\]
\[
=\left\langle (\rho_{t}J\nabla u_{t})-\rho_{t}'J\nabla u_{t}'),\nabla(U_{t}-U_{t}')\right\rangle =A+B
\]
 where $A=\left\langle \rho_{t}J\nabla(u_{t}-u_{t}'),\nabla(U_{t}-U_{t}')\right\rangle $
and $B=\left\langle (\rho_{t}-\rho_{t}')J\nabla u_{t}',\nabla(U_{t}-U_{t}')\right\rangle .$
Let us first show that $B=0.$ Integrating by parts gives 
\[
B:=\left\langle \nabla\cdot\nabla(U_{t}-U_{t}')J\nabla u_{t}',\nabla(U_{t}-U_{t}')\right\rangle =-\left\langle \nabla(U_{t}-U_{t}')(J\nabla u_{t}'),\nabla\cdot\nabla(U_{t}-U_{t}')\right\rangle -0=-B
\]
using that $\nabla\cdot(J\nabla u_{t}')=0,$ since $J\nabla f$ is
a Hamiltonian vector field (for any function $f)$ and thus its divergence
vanishes. Hence, $B=0.$ Next, by the Cauchy-Schwartz inequality
\[
A\leq C_{1}\left\Vert \nabla(u_{t}-u_{t}')\right\Vert \left\Vert \nabla(U_{t}-U_{t}')\right\Vert ,\,\,\,C_{1}:=\sup_{X\times[0,T]}\rho_{t}(x)=\sup_{X}\rho_{0},
\]
(using the previous lemma in the last equality). All that remains
is thus to show that
\begin{equation}
\left\Vert \nabla(u_{t}-u_{t}')\right\Vert \leq1/\delta^{n-1}\left\Vert \nabla(U_{t}-U_{t}')\right\Vert ,\label{eq:ineq for u inter terms of U}
\end{equation}
 where $\delta$ satisfies $\theta_{u_{t}}\geq\delta\omega_{X}$ and
$\delta>0$ (the existence of $\delta$ is provided by the previous
lemma). To prove this inequality first observe that for any two smooth
$u$ and $u'$ in $\text{PSH \ensuremath{(X,\theta)}}$ we have
\[
\int d(u-u')\wedge d^{c}(u-u')\wedge\theta_{u}^{n-1}\leq-\int(u-u')\left(\theta_{u}^{n}-\theta_{u'}^{n}\right),
\]
 as follows directly from expanding $\theta_{u}^{n}-\theta_{u'}^{n}=dd^{c}(u-u')\wedge(\theta_{u}^{n-1}+...).$
Hence, 
\[
\delta^{n-1}\left\Vert \nabla(u_{t}-u_{t}')\right\Vert ^{2}\leq-\int(u-u')\Delta(U_{t}-U_{t}')dV=\int\nabla(u-u')\cdot\nabla(U_{t}-U_{t}')dV
\]
 Applying the Cauchy-Schwartz inequality to the rhs above and dividing
both sides by $\left\Vert \nabla(u_{t}-u_{t}')\right\Vert $ concludes
the proof of the inequality \ref{eq:ineq for u inter terms of U}. 
\end{proof}
\begin{cor}
\label{cor:(Uniqueness)Let--and}(Uniqueness)\label{cor:unique complex euler}Let
$\rho_{t}$ and $\rho_{t}'$ be two smooth solutions to equation \ref{eq:complex euler monge text}
coinciding for $t=0.$ If $\rho_{0}>0,$ then $\rho_{t}=\rho_{t}'$
for $t\geq0.$ 
\end{cor}

\section{\label{sec:Comparison-with-point}Comparison with point vortices
and the 2D Euler Equation}

\subsection{\label{subsec:The-case-of the two sphere}The case of the two-sphere}

Consider the case when $X$ is the Riemann sphere (which, by stereographic
projection, may be identified with the unit sphere in $\R^{3})$ and
$L$ is the hyperplane line bundle $\mathcal{O}(1).$ We can then
express \cite[Ex 5.6]{berm8}

\[
E^{(N)}(x_{1},...,x_{N})=-\frac{1}{N(N-1)}\sum_{i<j\leq N}G_{\theta}(x_{i},x_{j}),\,\,\,\,E(\mu):=-\frac{1}{2}\int_{X^{2}}G_{\theta}\mu\otimes\mu.
\]
where $G(x,y)$ is usc, symmetric and satisfies

\[
d_{x}d_{x}^{c}G(x,y)=\delta_{y}-\theta
\]
 The function $E^{(N)}(x_{1},...,x_{N})$ above appears as the energy
per particle in Onsager's statistical mechanical description of the
2D incompressible Euler equations, adapted to the two- sphere \cite{ki2,k-w}. 
\begin{rem}
Results akin to Theorem \ref{thm:n is one intro} (but without the
exponential convergence) appear in \cite{ki2} , but a different regularization
of the microcanonical measure $\mu^{e}$ is employed, depending on
a variance-parameter $\sigma^{2}$ (that amounts to approximating
$\delta(t-e)$ with a Gaussian centered at $t=e).$ It has the virtue
that it bypasses the issue of equivalence of ensembles, but one drawback
is that after letting $N$ tend to infinity one also needs to let
$\sigma$ tend to zero in order to obtain $\mu^{e}.$ Differential-geometric
interpretations of $\mu^{e},$ for $e=e_{c},$ are also stressed in
\cite{ki2} (using a different, but equivalent formulation of equation
\ref{eq:twisted KE intro} for $\beta=-1).$
\end{rem}

To see the connection to the 2D Euler equations note that the fixed
volume form $dV$ defines a symplectic form on $X,$ corresponding
to a Riemannian metric, compatible with the complex structure $J,$
namely $dV(\cdot,\cdot J).$ Denote by $\Delta$ the corresponding
Laplacian, normalized so that $\Delta vdV=dd^{c}v.$ Equation \ref{eq:Poisson form of complex euler}
then reads
\begin{equation}
\frac{\partial\rho_{t}}{\partial t}=\{\rho,u\},\,\,\,\,\Delta u=\rho-\theta/dV.\label{eq:eq for pho n one}
\end{equation}
 Setting $q_{t}:=\Delta u_{t}$ this equivalently means that
\begin{equation}
\frac{\partial q_{t}}{\partial t}=\{q_{t}+\frac{\theta}{dV},u_{t}\},\,\,\,\,\Delta u_{t}=q_{t}.\label{eq:eq for q with theta}
\end{equation}
 In particular, if $\theta/dV$ is constant this equation becomes
\[
\frac{\partial q_{t}}{\partial t}=\{q_{t},u_{t}\},\,\,\,\,\Delta u_{t}=q_{t}
\]
 which is precisely the vorticity formulation of the incompressible
Euler equation on the two-sphere (endowed with the metric defined
by $dV)$. Note that here we only consider initial data satisfying
$\Delta u>-\theta/dV.$ Moreover, in the present setup we have excluded
precisely the case when $\theta/dV$ is constant, In fact, in that
case $e(\beta)=e_{0}$ for all $\beta>-1.$ 
\begin{rem}
The equation \ref{eq:eq for q with theta} appears in geophysical
fluid dynamics as the \emph{quasi-geostrophic equation }(aka the\emph{
barotropic vorticity equation}) with $\theta/dV$ representing the
topography; see \cite{c-c-k} for global existence and uniqueness
results. In plasma physics the equation \ref{eq:eq for pho n one}
appears as the guiding center approximation and quasi-neutral limit
of the \emph{Vlasov-Poisson equation} for electron plasmas with $\theta$
playing representing a magnetic flux or a neutralizing back-ground
charge density \cite{g-s,br}.
\end{rem}

\subsection{Compact domains in $\R^{2}$}

Let now $X$ be a compact domain in $\R^{2}$ with smooth boundary
and denote by $dV$ Lebesgue measure on $X.$ Let $G(x,y)$ be the
Green function for the Dirichlet problem in $X$ for the Laplace operator,
i.e. $G(x,y)$ is an usc symmetric function on $X^{2},$ normalized
as follows:
\[
d_{x}d_{x}^{c}G(x,y)=\delta(x-y),\,\,\,\,G(\cdot,y)_{|\partial X}=0
\]
 Denote by $\gamma(x)$ the Robin function of $X$ and set 
\begin{equation}
E^{(N)}(x_{1},...,x_{N}):=-\sum_{i<j\leq N}\frac{G(x_{i},x_{j})}{N^{2}}+\sum_{i\leq N}\frac{\gamma(x_{i})}{N^{2}},\,\,\,E(\mu):=-\frac{1}{2}\int_{X^{2}}G\mu\otimes\mu\label{eq:def of E N for domains}
\end{equation}
The corresponding microcanonical measures $\mu_{]e-\epsilon,e[}^{(N)}$
and $\mu_{]e,e+\epsilon[}^{(N)}$ are studied in \cite{clmp2} in
the context of Onsager's original point vortex model. In this setup
$e_{c}$ may be defined as 
\[
e_{c}:=\lim_{\beta\rightarrow-1}F'(\beta)
\]
(with the present normalization $\beta=-1$ corresponds to $\beta=-8\pi$
in \cite{clmp2}). The following analog of Theorem \ref{thm:LDP for micro intro}
was shown in \cite{clmp2} when $e<e_{c}:$

\begin{equation}
\text{(i)}\,\lim_{N\rightarrow\infty}N^{-1}\log\int_{\left\{ E^{(N)}\in]e-\epsilon,e[\right\} }dV^{\otimes N}=S(e),\,\,\,\text{(ii)}\lim_{N\rightarrow\infty}\delta_{N}=\mu^{e},\label{eq:conv towards S(e) in conj Fano intro-1}
\end{equation}
 where the convergence of $\delta_{N}$ holds in probability and $\mu^{e}$
denotes the unique maximum entropy measure $\mu^{e}$ with energy
$e.$ The main virtue of the present approach is that it also yields
exponential convergence. An analog of Theorem \ref{thm:micro var principle Fano intro}
is also established in \cite{clmp2} with a rather different proof.
The main difference is that the proof in \cite{clmp2} exploits that
any maximizer of $S(\mu)$ on $\{E(\mu)=e\}$ is of the form $\Delta u_{\beta}$
where $u_{\beta}$ satisfies the following \emph{mean field equation}
(which is analogous to equation \ref{eq:ma eq with beta}): 
\begin{equation}
\Delta u=\frac{e^{\beta u}}{\int_{\Omega}e^{\beta u}dV},\,\,\,u_{|\partial X}=0\label{eq:mfe domain}
\end{equation}
 for some $\beta\in\R$ (by \cite[Prop 2.3]{clmp2}). It would be
interesting to extend this result to $n>1$ in the present complex-geometric
setup. Anyhow, an important point of the proofs of Theorems \ref{thm:micro variational principle low energ intro},
\ref{thm:micro var principle Fano intro} is that they bypass this
difficulty. 

\subsubsection{Domains of the first and second kind vs K-polystability}

Following \cite{clmp2}, a domain $X$ is said to be \emph{of the
first} \emph{kind }if $e_{c}=\infty$ and \emph{of the second kind}
if $e_{c}<\infty.$ By \cite{b-l}, $X$ is a domain of the second
kind iff there exists a solution to the equation \ref{eq:mfe domain}
at $\beta=-1$ (which is always uniquely determined) and this condition
can be characterized in terms of the Robin function of $X$ (for example,
a disc is of the first kind, while a long and thin ellipse is of the
second kind). Thus, comparing with section \ref{subsec:The-high-energy intro},
domains of the second kind play the role of K-polystable Fano manifolds. 

\section{Appendix}

Let $X$ be a complex projective variety of dimension $n$ and $\theta$
a closed positive $(1,1)-$current on $X$ with locally bounded potentials.
Given a real curve $u_{t}$ in $\text{PSH}(X,\theta)\cap L^{\infty}(X)$
we will denote by $U$ the function on $X\times\C$ defined by $U(x,\tau):=u_{t}(x),$
where $t$ denotes the real part of $\tau.$ We will identify forms
on $X$ with their pull-backs to $X\times\C$ under natural projection
map $X\times\C\rightarrow X.$ A curve $u_{t}$ in $\text{PSH}(X,\theta)\cap L^{\infty}(X)$
is said to be a \emph{(psh) geodesic }if 

\begin{equation}
\theta_{U}^{n+1}=0,\,\,\,\theta_{U}\geq0,\,\,\,\,\,\,\text{\ensuremath{\left(\theta_{U}:=\ensuremath{\theta+dd^{c}U}\right)}}\label{eq:geod eq}
\end{equation}
We will say that $u_{t}$ is a \emph{bounded regular geodesic} if
it is a geodesic in $\text{PSH}(X,\theta)$ such that $u_{t}\in C^{\infty}(X_{\text{reg }})$
for any fixed $t.$
\begin{lem}
\label{lem:E strict convex along geod}Let $X$ be a Fano variety
and assume that $\theta>0$ on $X_{\text{reg}}$ in the sense that
$\theta$ is bounded from below by some Kähler form on $X_{\text{reg}}.$
Then $t\mapsto E(\theta_{u_{t}}^{n}/\text{vol\ensuremath{(L)}})$
is strictly convex for any regular bounded geodesic $u_{t}$ in $\text{PSH}(X,\theta),$
where $E(\mu)$ denote the pluricomplex energy of $\mu$ with respect
to $\theta.$ 
\end{lem}

\begin{proof}
We first show that 
\[
\frac{1}{2^{2}}\frac{d^{2}E(\theta_{u_{t}}^{n}/\text{vol\ensuremath{(L)}})}{d^{2}t}d\tau\wedge d\bar{\tau}=\frac{n}{\text{vol\ensuremath{(L)}}}\int\theta_{U}^{n}\wedge\theta,
\]
which proves that $t\mapsto E(\theta_{u_{t}}^{n}/\text{vol\ensuremath{(L)}})$
convex as long as $\theta\geq0.$ To this end first observe that 
\begin{equation}
\text{vol\ensuremath{(L)}}E(\theta_{u}^{n}/\text{vol\ensuremath{(L)}})=-\frac{n}{(n+1)}\mathcal{E}(u)+\mathcal{E}_{\theta}(u),\label{eq:formula for E in terms of e theta}
\end{equation}
where 
\[
\mathcal{E}_{\theta}(u):=\int_{X}\sum_{j=0}^{n-1}(-u)\left(\theta_{u}^{j}\wedge\theta^{n-1-j}\right)\wedge\theta.
\]
This follows directly from rewriting $1/(n+1)=1-n/(n+1).$ Next, recall
that 
\[
(i)\,\frac{1}{(n+1)2^{2}}\frac{d^{2}\mathcal{E}(u_{t})}{d^{2}t}id\tau\wedge d\bar{\tau}=\int_{X}\theta_{U}^{n+1},\,\,\,(ii)\,\frac{1}{n2^{2}}\frac{d^{2}\mathcal{E}_{\theta}(u_{t})}{d^{2}t}id\tau\wedge d\bar{\tau}=\int_{X}\theta_{U}^{n}\wedge\theta
\]
(see, for example, \cite[Prop 6.2]{bbgz} for formula $(i)$ and $(ii)$
is shown in precisely the same way). Hence, formula \ref{eq:formula for E in terms of e theta}
follows from the geodesic equation \ref{eq:geod eq}. 

Next, note that it will be enough to show that 
\begin{equation}
\int_{X}\theta_{U}^{n}\wedge\theta\geq\int_{X}d\dot{u_{t}}\wedge d^{c}\dot{u_{t}}\wedge\theta\wedge\theta_{u_{t}}^{n-2}\label{eq:ineq in pf lemma E strict convex}
\end{equation}
Indeed, since $\theta_{u_{t}}>0$ and $\theta>0$ on $X_{\text{reg}}$
this implies that, if $d^{2}E(\theta_{u_{t}}^{n}/\text{vol\ensuremath{(L)}})/dt=0$
for all $t,$ then $d\dot{u_{t}}=0$ on $X_{\text{reg}}.$ This means
that there exists a constant $C$ such that $u_{t}=u_{0}+Ct$ on $X_{\text{reg}},$
which implies that $u_{t}=u_{0}+Ct$ on all of $X$ (since $u_{t}\in\text{PSH}(X,\theta)\cap L^{\infty}).$
Hence, $\theta_{u_{t}}^{n}/\text{vol\ensuremath{(L)}}$ is independent
of $t.$ Finally, to prove the inequality \ref{eq:ineq in pf lemma E strict convex}
we decompose 
\begin{equation}
\partial\overline{\partial}U=\partial_{X}\overline{\partial}_{X}u_{t}+\left(\ddot{u}_{t}d\tau\wedge d\bar{\tau}+\partial_{X}\dot{u_{t}}\wedge d\bar{\tau}-\overline{\partial}_{X}\dot{u_{t}}\wedge d\tau\right),\,\,\,\,\dot{u_{t}}:=\frac{\partial u_{t}}{2\partial t},\,\,\ddot{u}_{t}:=\frac{\partial^{2}u_{t}}{2^{2}\partial^{2}t}\label{eq:decompos ddbar U}
\end{equation}
Applying the binomial expansion to the geodesic equation \ref{eq:geod eq}
yields 
\[
\ddot{u}_{t}\theta_{u_{t}}^{n}=nd\dot{u_{t}}\wedge d^{c}\dot{u_{t}}\wedge\theta_{u_{t}}^{n-1},
\]
 using that $\theta_{u_{t}}^{n+1}=0$ and $\binom{n+1}{2}/\binom{n+1}{1}=n/2.$
Likewise, expanding $\theta_{U}^{n}\wedge\theta$ yields
\[
\int_{X}\theta_{U}^{n}\wedge\theta\ddot{u}_{t}=\int_{X}\ddot{u}_{t}\theta_{u_{t}}^{n-1}\wedge\theta-(n-1)\int_{X}d\dot{u_{t}}\wedge d^{c}\dot{u_{t}}\wedge\theta\wedge\theta_{u_{t}}^{n-2},
\]
 using that $\binom{n}{2}/\binom{n}{1}=(n-1)/2.$ Next, we can express
$\theta_{u_{t}}^{n-1}\wedge\theta=\text{tr}\theta\theta_{u_{t}}^{n},$
where $\text{tr}\theta$ denotes the trace of $\theta$ with respect
to $\theta_{u_{t}}.$ Hence, 
\[
\int_{X}\ddot{u}_{t}\theta_{u_{t}}^{n-1}\wedge\theta=n\int_{X}d\dot{u_{t}}\wedge d^{c}\dot{u_{t}}\wedge\text{tr}\theta\theta_{u_{t}}\wedge\theta_{u_{t}}^{n-2}\geq n\int_{X}d\dot{u_{t}}\wedge d^{c}\dot{u_{t}}\wedge\theta\wedge\theta_{u_{t}}^{n-2},
\]
 using that $\text{tr}\theta\theta_{u_{t}}\geq\theta.$ Since $n-(n-1)=1$
this proves the inequality \ref{eq:ineq in pf lemma E strict convex}. 
\end{proof}
\begin{rem}
\label{rem:E tends to infty}It follows form the previous lemma that
without assuming any positivity assumption on $\theta$ 
\begin{equation}
\lim_{t\rightarrow\infty}E(\theta_{u_{t}}^{n}/\text{vol\ensuremath{(L))}=\ensuremath{\infty} }\label{eq:pf Remark E to infinty}
\end{equation}
Indeed, fix $\theta_{0}$ such that $\theta_{0}>0$ on $X_{\text{reg }}.$
Then $\left|E_{\theta}(\mu)-E_{\theta_{0}}(\mu)\right|\leq C_{0}$
for a constant $C_{0}$ as follows readily from the definition of
$E_{\theta}(\mu).$ Hence, we may as well assume that $\theta>0.$
The previous lemma then shows that $t\mapsto E(\theta_{u_{t}}^{n}/\text{vol\ensuremath{(L))} }$
is strictly convex. Since $E_{\theta}(\mu)\geq0$ this implies \ref{eq:pf Remark E to infinty}. 
\end{rem}

\end{document}